\documentclass[10pt]{amsart}

\usepackage{fullpage}

\usepackage{amsmath}
\usepackage{amsfonts}
\usepackage{amssymb}
\usepackage{graphicx}
\usepackage{mathrsfs}
\usepackage{amsthm}
\usepackage{enumerate}
\usepackage{leftidx}
\usepackage{hyperref}
\usepackage{extarrows}
\usepackage[all]{xy}
\usepackage{stmaryrd}
\usepackage{wasysym}
\usepackage[OT2,T1]{fontenc}
\usepackage{epstopdf}
\usepackage {hyperref}
\numberwithin{equation}{subsection}
\newtheorem{thm}[subsubsection]{Theorem}
\newtheorem*{thm*}{Theorem}

\newtheorem{cor}[subsubsection]{Corollary}
\newtheorem{lem}[subsubsection]{Lemma}
\newtheorem{prop}[subsubsection]{Proposition}
\theoremstyle{definition}
\newtheorem{defn}[subsubsection]{Definition}
\newtheorem{ass}[subsubsection]{Assumption}
\theoremstyle{remark}
\newtheorem{rem}[subsubsection]{Remark}
\newtheorem{eg}[subsubsection]{Example}

\DeclareMathOperator{\ad}{ad}

\DeclareMathOperator{\der}{der}

\DeclareMathOperator{\et}{\acute{e}t }

\DeclareMathOperator{\Gal}{Gal}

\DeclareMathOperator{\Sh}{Sh}

\DeclareMathOperator{\ord}{ord}

\newcommand{\C}{\ensuremath{\mathbb{C}}}
\newcommand{\Z}{\ensuremath{\mathbb{Z}}}

\newcommand{\A}{\ensuremath{\mathbb{A}}}

\newcommand{\Fpbar}{\ensuremath{\overline{\mathbb{F}}_p}}

\newcommand{\Q}{\ensuremath{\mathbb{Q}}}

\newcommand{\Ok}{\ensuremath{\mathcal{O}}}

\newcommand{\pdiv}{\ensuremath{\mathscr{G}}}
\newcommand{\Spec}{\ensuremath{\mathrm{Spec}\ }}

\newcommand{\s}{\ensuremath{\tilde{s}}}

\newcommand{\Adm}{\ensuremath{\text{Adm}}}

\newcommand{\fka}{\ensuremath{\mathfrak{a}}}

\newcommand{\fkc}{\ensuremath{\mathfrak{c}}}
\newcommand{\fkd}{\ensuremath{\mathfrak{d}}}

\newcommand{\fkf}{\ensuremath{\mathfrak{f}}}

\newcommand{\fkh}{\ensuremath{\mathfrak{h}}}

\newcommand{\fkl}{\ensuremath{\mathfrak{l}}}
\newcommand{\fkm}{\ensuremath{\mathfrak{m}}}

\newcommand{\fkp}{\ensuremath{\mathfrak{p}}}
\newcommand{\fkq}{\ensuremath{\mathfrak{q}}}

\newcommand{\fks}{\ensuremath{\mathfrak{s}}}

\newcommand{\fkB}{\ensuremath{\mathfrak{B}}}

\newcommand{\fkG}{\ensuremath{\mathfrak{G}}}

\newcommand{\bbA}{\ensuremath{\mathbb{A}}}

\newcommand{\bbC}{\ensuremath{\mathbb{C}}}
\newcommand{\bbD}{\ensuremath{\mathbb{D}}}

\newcommand{\bbF}{\ensuremath{\mathbb{F}}}
\newcommand{\bbG}{\ensuremath{\mathbb{G}}}

\newcommand{\bbP}{\ensuremath{\mathbb{P}}}
\newcommand{\bbQ}{\ensuremath{\mathbb{Q}}}
\newcommand{\bbR}{\ensuremath{\mathbb{R}}}
\newcommand{\bbS}{\ensuremath{\mathbb{S}}}

\newcommand{\bbZ}{\ensuremath{\mathbb{Z}}}

\newcommand{\bfA}{\ensuremath{\mathbf{A}}}
\newcommand{\bfB}{\ensuremath{\mathbf{B}}}

\newcommand{\bfE}{\ensuremath{\mathbf{E}}}

\newcommand{\bfK}{\ensuremath{\mathbf{K}}}
\newcommand{\bfL}{\ensuremath{\mathbf{L}}}

\newcommand{\bfR}{\ensuremath{\mathbf{R}}}
\newcommand{\bfS}{\ensuremath{\mathbf{S}}}
\newcommand{\bfT}{\ensuremath{\mathbf{T}}}

\newcommand{\rmE}{\ensuremath{\mathrm{E}}}
\newcommand{\rmF}{\ensuremath{\mathrm{F}}}

\newcommand{\rmH}{\ensuremath{\mathrm{H}}}

\newcommand{\rmJ}{\ensuremath{\mathrm{J}}}
\newcommand{\rmK}{\ensuremath{\mathrm{K}}}

\newcommand{\rmR}{\ensuremath{\mathrm{R}}}
\newcommand{\rmS}{\ensuremath{\mathrm{S}}}
\newcommand{\rmT}{\ensuremath{\mathrm{T}}}

\newcommand{\scrE}{\ensuremath{\mathscr{E}}}

\newcommand{\scrG}{\ensuremath{\mathscr{G}}}

\newcommand{\scrI}{\ensuremath{\mathscr{I}}}

\newcommand{\scrL}{\ensuremath{\mathscr{L}}}

\newcommand{\scrO}{\ensuremath{\mathscr{O}}}

\newcommand{\scrR}{\ensuremath{\mathscr{R}}}
\newcommand{\scrS}{\ensuremath{\mathscr{S}}}

\newcommand{\scrV}{\ensuremath{\mathscr{V}}}

\newcommand{\calA}{\ensuremath{\mathcal{A}}}

\newcommand{\calD}{\ensuremath{\mathcal{D}}}
\newcommand{\calE}{\ensuremath{\mathcal{E}}}
\newcommand{\calF}{\ensuremath{\mathcal{F}}}
\newcommand{\calG}{\ensuremath{\mathcal{G}}}

\newcommand{\calM}{\ensuremath{\mathcal{M}}}
\newcommand{\calN}{\ensuremath{\mathcal{N}}}
\newcommand{\calO}{\ensuremath{\mathcal{O}}}

\newcommand{\calV}{\ensuremath{\mathcal{V}}}

\newcommand{\calX}{\ensuremath{\mathcal{X}}}

\newcommand{\OKE}{\ensuremath{\mathcal{O}_{\mathbf{E}_\lambda}}}
\newcommand{\kl}{\ensuremath{k_\lambda}}

\newcommand{\rmSh}{\ensuremath{\mathrm{Sh}}}

\newcommand{\sshg}{\ensuremath{\mathscr{S}_{K}(G)}}
\newcommand{\sshG}{\ensuremath{\mathscr{S}_{K}(G_{\rmS,\rmT})}}
\newcommand{\sshGa}{\ensuremath{\mathscr{S}_{K}(G_{\rmS_\fka,\rmT_\fka})}}

\newcommand{\ignore}[1]{{}}
\DeclareSymbolFont{cyrletters}{OT2}{wncyr}{m}{n}
\usepackage{enumitem}

\setlist[enumerate]{leftmargin=*}
\setlist[itemize]{leftmargin=*}

\title{Motivic Cohomology of Quaternionic Shimura varieties and level raising}
\author{Rong Zhou}
\email{rzhou@ias.edu} \thanks{}
\address{Institute for Advanced Study, 1 Einstein Drive, Princeton, NJ 08540.}


\date{\today}

\begin{document}
	\begin{abstract}
		We study the motivic cohomology of the special fiber of quaternionic Shimura varieties at a prime of good reduction. We exhibit  classes in these motivic cohomology groups and use this to give an explicit geometric realization of level raising between Hilbert modular forms.  
		The main ingredient for our construction is a form of Ihara's Lemma for compact quaternionic Shimura surfaces which we prove by generalizing a method of Diamond--Taylor. Along the way we also verify the Hecke orbit conjecture for these quaternionic Shimura varieties which is a key input for our proof of Ihara's Lemma.
	\end{abstract}
\maketitle

\tableofcontents

\section{Introduction}
\ignore{\subsection{Motivic cohomology of schemes} Let $X$ be a scheme of finite type over a field $k$. For integers $p,q$, the motivic cohomology groups $\rmH_{\calM}^i(X,\bbZ(j))$ are an important algebraic invariant of $X$. Indeed when $i\geq 0$, the motivic cohomology group $\rmH_{\calM}^{2j}(X,\bbZ(j))$ recovers the usual Chow group of codimension $j$ cycles on $X$ modulo rational equivalence. Many of the most interesting problems number theory and algebraic geometry concern properties of motivic cohomology groups.

 Following work of many authors, it is known that the motivic cohomology groups of $X$ satisfy a lot of the formal properties of a cohomology theory. For example, there is a localization exact sequence for $Z\subset X$ a closed subscheme extending the exact localization sequence for ordinary Chow groups, and if $X$ is smooth and proper, there is a Poincar\'e duality isomorphism. However there are still many things about motivic cohomology we do not understand. For example it is not even known whether the motivic cohomology groups are finitely generated, or whether $\rmH_{\calM}^i(X,\bbZ(j))=0$ for $i<0$. In the cases when these motivic cohomology groups are expected to be non-zero, it is also a difficult problem to produce explicit classes in motivic cohomology. We refer to  \cite{Geisser} for a short survey of the properties, known and conjectural, concerning motivic cohomology groups. 
 
 The aim of this paper  is then two-fold. Firstly we would like to produce many examples where  motivic cohomology groups are non-zero by producing explicit classes in these groups. Moreover the situation we consider is quite a non-trivial one, where $X$ is the special fiber of a certain Shimura variety. Secondly we will use these  classes to exhibit a geometric realization of level raising  for Hilbert modular forms. Indeed this second aim leads to the new observation that the motivic cohomology of the special fiber of Shimura varieties can encode very rich arithmetic information.  We now give some more details about the construction.
 
 When $X$ is a smooth scheme over $k$, the motivic cohomology groups $\rmH_{\calM}^i(X,\bbZ(j))$ are identified with the higher Chow groups $\mathrm{Ch}^j(X,2j-i)$ defined by Bloch \cite{Bloch}. As we have already mentioned, when $2j-i=0$, $\mathrm{Ch}^j(X,0)$ is  non other than the usual Chow group; when $2j-i>0$ one can think of classes in $\mathrm{Ch}^j(X,2j-i)$ as being cycles together with a function on the cycle. We refer to \S\ref{sec: Higher Chow groups} for the precise definition of these groups. From this point of view, classes in motivic cohomology can be identified with higher Chow classes; this is the approach we will use. 
 
 Using the interpretation in terms of higher Chow groups, we can exhibit classes in motivic cohomology by constructing cycles with functions on them satisfying certain properties. However it is difficult to show that the image of the class in motivic cohomology is non-zero as there could be many relations between  classes. For example, for  usual Chow groups, this amounts to understanding rational equivalence between cycles.
 
  Now the motivic cohomology groups are equipped with cycle class maps
\begin{equation}
\label{eq: intro cycle class map}
\mathrm{cyc}_l:\mathrm{Ch}^j(X, 2j-i)\rightarrow\rmH^i_{\mathrm{\acute{e}t}}(X,\bbZ_l(j))\end{equation} where $l$ is a prime coprime to the characteristic of $k$ and $\rmH^i_{\mathrm{\acute{e}t}}(X,\bbZ_l(j))$ is the continuous \'etale cohomology group of \cite{Jann}. Then one way to show the  classes are non-zero is to compute  the image in \'etale cohomology. 
 
 We will consider a variant of this construction using torsion coefficients. Let $k_\lambda$ be a finite extension of $\bbF_l$. Then we may define motivic cohomology and higher Chow groups with $k_\lambda$ coefficients. We write $\rmH_{\calM}^i(X,k_\lambda(j))$ and $\mathrm{Ch}^i(X,2j-i,k_\lambda)$ for these groups, which are isomorphic as before. Then there is a cycle class map  \begin{equation}
 \mathrm{Ch}^j(X,2j-i,k_\lambda)\rightarrow\rmH^i_{\mathrm{\acute{e}t}}(X,k_\lambda(j)).\end{equation} 
The group on the right is the \textit{absolute} \'etale cohomology group (i.e. we do not base change $X$ to the algebraic closure of $k$). Voevodsky's proof of the motivic Bloch--Kato conjecture \cite{Voe1} implies that the map (\ref{eq: intro cycle class map mod l}) is an isomorphism  when $j\geq i$. However when $j<i$,  very little is known; this is the case we will consider.
}
 \subsection{Main Theorem} The aim of this paper is to study the motivic cohomology of the special fiber of certain quaternionic Shimura varieties. For a scheme of finite type over a field, its motivic cohomology groups are a generalization of the usual Chow groups, and the main new observation of this paper is that for certain Shimura varieties, these groups can encode very rich arithmetic information. 
 More precisely, we will show that the cycle class map from motivic cohomology to \'etale cohomology gives a geometric realization of level raising between Hilbert modular forms.
 
 We now state our main result. Let $F$ be a totally real field of even degree $[F:\bbQ]=g$ and $p$ a prime which is \textit{inert} in $F$. Let $B$ be a totally indefinite quaternion algebra over $F$ which is unramified at the unique prime $\fkp$ above $p$ and $G$ the associated reductive group over $\bbQ$. Let $K$ be a sufficiently small compact open subgroup of $G(\bbA_f)$ such that $K=K_pK^p$ where $K_p\subset G(\bbQ_p)=GL_2(F_{\fkp})$ is the standard hyperspecial maximal compact $GL_2(\calO_{F_{\fkp}})$ and $K^p\subset G(\bbA_f^p)$. Then there is a Shimura variety $\rmSh_K(G)$ defined over $\bbQ$; it extends to a smooth integral model $\underline{\rmSh}_K(G)$ over $\bbZ_{(p)}$. We let $\scrS_K(G)$ denote its special fiber over $\bbF_p$ and $\scrS_K(G)_{\bbF_{p^g}}$ its base change to $\bbF_{p^g}$.
 
  Fix an irreducible cuspidal automorphic representation $\Pi$ of $GL_2(F)$ of parallel weight 2 defined over a number field $\bfE$.
 Let $\rmR$ be a finite set of places of $F$ not containing $\fkp$ and away from which $\Pi$ is unramified and $K$ is hyperspecial. We also choose a prime $\lambda$ of $\calO_{{\mathbf{E}}}$ whose residue characteristic is coprime to $p$ and write $k_\lambda=\calO_{{\mathbf{E}}}/\lambda$. We write $\rmH_\calM^i(\scrS_K(G)_{\bbF_{p^g}},k_\lambda(j))$ for the motivic cohomology group with $k_\lambda$ coefficients defined in \cite{SuVo}. By \cite{Voe2}, we may identify this with the higher Chow group $\mathrm{Ch}^j(\scrS_K(G)_{\bbF_{p^g}},2j-i,k_\lambda)$ defined in \cite{Bloch}. When $2j=i$, this group is  just the usual Chow group of codimension $j$ cycles  modulo rational equivalence (with coefficients in $k_\lambda$). The group $\mathrm{Ch}^j(\scrS_K(G)_{\bbF_{p^g}},2j-i,k_\lambda)$ is equipped with the following  cycle class map to the absolute \'etale cohomology:
\begin{equation}\label{eq: intro cycle class map mod l}\mathrm{Ch}^j(\scrS_K(G)_{\bbF_{p^g}},2j-i,k_\lambda)\rightarrow \rmH_{\text{\'et}}^i(\scrS_K(G)_{\bbF_{p^g}},k_\lambda(j)).\end{equation}
 
  We let $\bfT_{\rmR}$ denote the abstract  Hecke algebra of $GL_2(F)$ away from $\rmR$; it is the $\bbZ$-algebra generated by elements $T_{\fkq}, S_{\fkq}$ where $\fkq$ runs over primes of $F$ away from $\rmR$. Then the Hecke eigenvalues of $\Pi$ induce a map $$\phi^{\Pi}_\lambda:\bfT_{\rmR}\rightarrow \calO_{{\mathbf{E}}}\rightarrow k_\lambda.$$
 We write $\fkm_{\rmR}:=\ker(\phi^{\Pi}_\lambda)$ a maximal ideal of $\bfT_{\rmR}$ and $\fkm\subset \bfT_{\rmR\cup\{\fkp\}}$ the preimage in $\bfT_{\rmR\cup\{\fkp\}}$.
 
 The Hecke algebra $\bfT_{\rmR\cup\{\fkp\}}$ acts on the \'etale cohomology $\rmH_{\mathrm{\acute{e}t}}^\bullet(\scrS_K(G)_{\Fpbar},k_\lambda(-))$  and  higher Chow groups $\mathrm{Ch}^j(\scrS_K(G)_{\bbF_{p^g}},2j-i,k_\lambda)$ of $\scrS_K(G)$. Upon making a large image assumption on the$\mod\lambda$ Galois representation associated to $\Pi$ (see Assumption \ref{ass: property of l}) and localizing at the maximal ideal $\fkm$,  there is an isomorphism $$\rmH_{\et}^{g+1}(\scrS_K(G)_{\bbF_{p^{g}}},k_\lambda(g/2+1)))_{\fkm}\cong\rmH^1(\bbF_{p^{g}},\rmH_{\et}^g(\scrS_K(G)_{\Fpbar},k_\lambda(g/2+1))_{\fkm}).$$
 The cycle class map then induces the \textit{Abel--Jacobi map}:\begin{equation}\label{eq: intro Abel--Jacobi}\mathrm{Ch}^{g/2+1}(\scrS_K(G)_{\bbF_{p^{g}}},1,k_\lambda)_{\fkm}\rightarrow \rmH^1(\bbF_{p^{g}},\rmH_{\et}^g(\scrS_K(G)_{\Fpbar},k_\lambda(g/2+1))_{\fkm}).\end{equation}

In \S\ref{sec: Motivic Cohomology and Level-raising}, we will define a subgroup $\mathrm{Ch}^{g/2+1}_{\mathrm{lr}}(\scrS_K(G)_{\bbF_{p^{g}}},1,k_\lambda)_{\fkm}$ of $\mathrm{Ch}^{g/2+1}(\scrS_K(G)_{\bbF_{p^{g}}},1,k_\lambda)_{\fkm}$ using the geometry of Goren--Oort cycles\footnote{In fact the cycles we consider  arise from the supersingular locus.} on $\scrS_K(G)_{\bbF_{p^{g}}}$ as studied in \cite{TX}, \cite{TX1} and \cite{LT}. As the notation suggests, this subgroup is related to level raising. The main Theorem of the paper is the following; we refer to \S\ref{sec: Motivic Cohomology and Level-raising} for the precise statement.
\begin{thm}\label{thm:intro main}Suppose  that $p$ is a $\lambda$-level raising prime in the sense of Definition \ref{def: level raising prime} and that Assumptions \ref{ass: property of l} and \ref{ass: Dim Jacquet Langlands} are satisfied; in particular $T_\fkp\equiv p^g+1\mod\fkm_{\rmR}$ and $S_\fkp\equiv 1\mod \fkm_{\rmR}$. Then the map 
	$$\mathrm{Ch}_{\mathrm{lr}}^{g/2+1}(\scrS_K(G)_{\bbF_{p^{g}}},1,k_\lambda)/{\fkm}\rightarrow \rmH^1(\bbF_{p^{g}},\rmH^g(\scrS_K(G)_{\Fpbar},k_\lambda(g/2+1))/\fkm)$$ induced by (\ref{eq: intro Abel--Jacobi}) is surjective.
\end{thm}

We note that as in \cite[Remark 4.2, 4.6]{LT}, if there exist rational primes inert in $F$, and $\Pi$ is not dihedral and not isomorphic to a twist by a character of any of its internal conjugates, then for all but finitely many $\lambda$, the set of primes $p$ which are $\lambda$-level raising primes has positive density.

In general it is difficult problem to produce non-zero classes in motivic cohomology. The key input to proving the surjectivity in Theorem \ref{thm:intro main} is a form of Ihara's Lemma which we prove by generalizing a method of Diamond--Taylor \cite{DT}; see the next subsection for more details.

We now give an example of the construction of $\mathrm{Ch}_{\mathrm{lr}}^{g/2+1}(\scrS_K(G)_{\bbF_{p^{g}}},1,k_\lambda)$ which makes clear the relationship with level raising. We assume $g=2$ so that $\dim\scrS_K(G)=2$. 

 We write $B'$ for the totally definite quaternion algebra which agrees with $B$ at all finite places. We fix an isomorphism $$B'\otimes_{\bbQ}\bbA_f\cong B\otimes_{\bbQ}\bbA_f$$ which allows us to consider $K$ as a compact open subgroup of $B'\otimes_{\bbQ}\bbA_f$.
We let $\calX'$ and $\calX'_0(\fkp)$ denote the discrete Shimura sets
$$\calX':=B'\backslash B'\otimes_{\bbQ}\bbA_f/K,\quad \calX'_0(\fkp):=B'\backslash B'\otimes_{\bbQ}\bbA_f/K_0(\fkp)$$
where the compact open subgroup $K_0(\fkp)\subset B'\otimes_{\bbQ}\bbA_f$ agrees with $K$ away from $\fkp$ and is the standard Iwahori subgroup of $GL_2(F_{\fkp})$ at $\fkp$. We let $$\pi_1,\pi_2:\calX'_0(\fkp)\rightarrow \calX'$$
denote the natural degeneracy maps so that the diagram $$\calX'\xleftarrow{\pi_1}\calX'_{0}(\fkp)\xrightarrow{\pi_2}\calX'$$ is the usual Hecke correspondence for $\calX'$. For any finite set $S$, we write $\Gamma(S,k_\lambda)$ for the abelian group of $k_\lambda$-valued functions on $S$.

We may think of $\scrS_K(G)$ as a moduli space of abelian varieties with multiplication by some maximal order $\calO_B$ in $B$. We let  $\scrS_K(G)^{\mathrm{ss}}$ be the locus where the underlying abelian variety is supersingular.
 Using the geometry of $\scrS_K(G)^{\mathrm{ss}}$ one can show that under the assumptions of Theorem \ref{thm:intro main}, $\mathrm{Ch}^{2}(\scrS_K(G)_{\bbF_{p^{2}}},1,k_\lambda)_{\fkm}$ admits a map from $$\rmK_{\fkm}:=\ker\left((\pi_{1*},\pi_{2*}):\Gamma(\calX'_{0}(\fkp),k_\lambda)\rightarrow \Gamma(\calX',k_\lambda)\right)_{\fkm}.$$ The construction  uses an interpretation of classes in $\mathrm{Ch}^{2}(\scrS_K(G)_{\bbF_{p^{2}}},1,k_\lambda)_{\fkm}$ as cycles together with a rational function on the cycle; see \S\ref{sec: Motivic Cohomology and Level-raising} for the details. Then $\mathrm{Ch}^{2}_{\mathrm{lr}}(\scrS_K(G)_{\bbF_{p^{2}}},1,k_\lambda)_{\fkm}$ is defined to be the image of $\rmK_{\fkm}$.  Theorem \ref{thm:intro main} in this case follows from the following stronger result:
\begin{thm}\label{thm:intro surface}Let $g=2$. Suppose that $p$ is a $\lambda$-level raising prime and that Assumption \ref{ass: property of l} is satisfied. Then the map \begin{equation}
	\label{eq: intro abel jacobi surface}
\rmK_\fkm\rightarrow \rmH^1(\bbF_{p^{2}},\rmH_{\et}^2(\scrS_K(G)_{\Fpbar},k_\lambda(2))_\fkm)	\end{equation} is surjective.
\end{thm}
The relationship with level-raising should now be clear. Indeed under the Jacquet--Langlands correspondence, $\bfT_{\rmR\cup\{\fkp\}}$ acts on left hand side of (\ref{eq: intro abel jacobi surface}) via the $\fkp^{\mathrm{new}}$ quotient in the sense of \cite{Ribet}, whereas it is well known that it acts via the $\fkp^{\mathrm{old}}$ quotient  on the right hand side. In this sense, the Abel--Jacobi map gives an explicit realization of the congruence between old and new forms.

It is known by the work of many authors that the motivic cohomology groups satisfy many of the formal properties of a cohomology theory. However there is much that is still not understood, we refer to \cite{Geisser} for a brief survey. We may use Theorem \ref{thm:intro surface} to show that in certain cases of Shimura surfaces, motivic cohomology coincides with \'etale cohomology, upon localizing at $\fkm$. 
\begin{thm}Let $g=2$. Suppose that $p$ is a $\lambda$-level raising prime and that Assumption \ref{ass: property of l} is satisfied. Then the cycle class map induces an isomorphism $$\mathrm{H}^3_{\calM}(\scrS_K(G)_{\bbF_{p^{2}}},k_\lambda(2))_\fkm\xrightarrow\sim \rmH_{\et}^3(\scrS_K(G)_{\bbF_{p^{2}}},k_\lambda(2))_\fkm.$$
\end{thm}
When $i\leq j$,  Voevodsky \cite{Voe1} has shown that $\mathrm{Ch}^j(X,2j-i,k_\lambda)$ is isomorphic to  $\rmH_{\et}^i(X,k_\lambda(j))$, for $X$ proper smooth over any base field. For $i>j$, not much seems to be known.

\begin{rem}When $g$ is odd, there is an Abel--Jacobi map \begin{equation}\label{eq: Liu-Tian}\mathrm{Ch}^{\frac{g-1}{2}}(\scrS_K(G)_{\bbF_{p^{2g}}},k_\lambda)_\fkm\rightarrow \rmH^1(\bbF_{p^{2g}},\rmH_{\et}^{g}(\scrS_K(G)_{\Fpbar},k_\lambda(\lfloor g/2\rfloor+1))_\fkm)\end{equation}
	In this case the supersingular locus  $\scrS_K(G)^{\mathrm{ss}}_{\bbF_{p^{g}}}$   is equidimensional of dimension $\frac{g-1}{2}$ and we may consider the subgroup  $\mathrm{Ch}_{\mathrm{lr}}^{\frac{g-1}{2}}(\scrS_K(G)_{\bbF_{p^{2g}}},k_\lambda)_\fkm$ generated by the irreducible components in
	 $\scrS_K(G)^{\mathrm{ss}}_{\bbF_{p^{2g}}}$. Then \cite[Theorem 1.3]{LT} have shown the surjectivity of (\ref{eq: Liu-Tian})  modulo $\fkm$ restricted to this subgroup. Thus our Theorem \ref{thm:intro main} can be thought of as the even dimensional analogue of the Theorem of Liu--Tian. The main new observation of this work is that we are able to produce certain classes in motivic cohomology, or higher Chow groups, as opposed to ordinary Chow groups. Its conceptual importance lies in the fact that we are able to obtain a geometric interpretation of \textit{even} dimensional Galois cohomology.
	 
	\end{rem}

\begin{rem}
		 In \cite{LT}, the geometric realization of level raising was a key ingredient in their proof of certain cases of the Bloch--Kato conjecture, see \cite[Theorem 5.7]{LT}. Our work should have applications to cases of this conjecture for non-central $L$-values; we aim to carry this out in a future work.
\end{rem}

\subsection{Proof of main result and Ihara's Lemma}
We now explain the proof of Theorem \ref{thm:intro main}. Our approach follows that of \cite{LT}, but there are many new difficulties in the even dimensional case. 

Firstly, using the construction of the group $\mathrm{Ch}_{\mathrm{lr}}^{g/2+1}(\scrS_K(G)_{\bbF_{p^{g}}},1,k_\lambda)$
and the intersection pairing between certain Goren--Oort strata proved  in \cite{TX1}, we reduce to proving the surjectivity statement in the case of  quaternionic Shimura surfaces, see Proposition \ref{prop: AJ surjective for surface}. The statement in this case follows from the following form of  Ihara's Lemma. These type of results first appeared in Ribet's ICM article \cite{Ribet1} for the case of modular curves and over the last thirty years they have seen many important arithmetic applications. Therefore our result in the case of surfaces should certainly be of independent interest.

For simplicity, we only state the result in the totally indefinite case; we refer to Theorem \ref{thm: Ihara's Lemma} for the more general statement. Thus we assume $g=2$ as in the example of the previous subsection.

\begin{thm}[Ihara's Lemma]\label{thm:intro Ihara}Under the  Assumption \ref{ass: surface}, the map $$\pi_1^*+\pi_2^*:\rmH_{\et}^2(\rmSh_K(G)_{\overline{\bbQ}},k_\lambda)^2_\fkm\rightarrow \rmH_{\et}^2(\rmSh_{K_0(\fkp)}(G)_{\overline{\bbQ}},k_\lambda)_\fkm$$
	is injective.
\end{thm}
Here $\rmSh_{K_0(\fkp)}(G)$ denotes the quaternionic Shimura surfaces with Iwahori level structure at $p$ and $\pi_1,\pi_2$ are the natural degeneracy maps. In fact the appropriate Abel--Jacobi map in this case can be related to the  map $\pi_1^*+\pi_2^*$  in the statement of Theorem \ref{thm:intro Ihara}; they are essentially dual to one another. To show the existence of this duality requires a careful analysis of the  global geometry of the mod $p$ fiber of the quaternionic Shimura surface \textit{with Iwahori level structure} at $p$. We note that in this case the Shimura surface has bad reduction at $p$. The main result which is Corollary \ref{cor:Iwahori level structure} is proved in an appendix and is analogous to the results of \cite{Stamm} in the case of Hilbert modular surfaces.

 We now describe our approach to Theorem \ref{thm:intro Ihara}. The result in the case of Hilbert modular varieties has been proved by Dimitrov \cite{Dim}. However his proof relied crucially on the existence of a $q$-expansion. Note that when $g>2$, even if one is interested in Theorem \ref{thm:intro main} for Hilbert modular varieties, the reduction to the case of surfaces will necessitate that we consider compact Shimura surfaces where a $q$-expansion is not available.
We therefore take another approach by generalizing a method of Diamond--Taylor who proved the result for Shimura curves \cite{DT}.

We first apply a crystalline comparison isomorphism to reduce the problem to proving  injectivity of a certain map between global sections of line bundles over the$\mod l$ reduction of the Shimura surface  (Proposition \ref{prop: Ihara coherent}). The property that  a non-zero section lies in the kernel implies that the divisor $D$ corresponding to this section is invariant under  $p$-power Hecke operators. In the case of Shimura curves, it's known that the image of an ordinary point under $p$-power Hecke operators is infinite; this constrains $D$ to be supported on the supersingular locus. Since $p$-power Hecke operators act transitively on supersingular points, the support contains the supersingular locus and this is enough to deduce a contradiction for degree reasons.

In the case of surfaces, we need a stronger result to constrain the support of the divisor $D$.  In section \S\ref{sec: Hecke orbit conjecture}, we prove a  version of the Hecke orbit conjecture of Chai--Oort \cite{ChOort} for the ordinary locus on quaternionic Shimura varieties. We assume $l$ is a prime where the compact open $K$ is hyperspecial and we write $\mathscr{S}^l_K(G)$ for the$\mod l$ reduction of the integral model $\underline{\mathrm{Sh}}^l_K(G)$ at a prime of the reflex field above $l$. We write $\mathscr{S}^l_K(G)^{\mathrm{ord}}$ for the locus where the universal abelian variety is ordinary.

\begin{thm}[Hecke orbit Conjecture]\label{thm:intro Hecke orbit} Let $x\in \mathscr{S}^l_K(G)^{\mathrm{ord}}(\overline{\bbF}_l)$. Then the prime-to-$l$ Hecke orbit is Zariski dense in $\mathscr{S}^l_K(G)$.
\end{thm}
In fact we prove this result in a more general situation; we refer to \S\ref{sec: 3 statement} for the statement.  Using the strong approximation theorem, we deduce that the $p$-power Hecke orbit of $x\in\mathscr{S}^l_K(G)^{\mathrm{ord}}(\overline{\bbF}_l)$ is Zariski dense in the connected component of $\mathscr{S}^l_K(G)_{\overline{\bbF}_l}$ containing it. This allows one to show that $D$ is supported on the complement of $\mathscr{S}^l_K(G)^{\mathrm{ord}}$. A computation involving intersection numbers of $D$ with certain cycles on $\mathscr{S}^l_K(G)$ then gives the desired contradiction.
\begin{rem}
	In \cite{LT}, the authors reduce their surjectivity result to a form of Ihara's Lemma for Shimura curves, for which the method of \cite{DT} is directly applicable. In our case, the most pertinent case is that of Shimura surfaces which, as explained above, is more delicate.
\end{rem}
We note that many of the quaternionic Shimura varieties we consider do not admit good moduli interpretations. Thus in order to obtain the geometric results we need, we study the geometry of certain auxiliary unitary Shimura varieties which are of PEL-type. Using \cite[\S2]{TX}, the results for unitary Shimura varieties transfer easily to the quaternionic side. The moduli interpretation for the unitary Shimura varieties allow us to adapt many proofs in the  case of Hilbert modular varieties to the quaternionic case.

\subsection{Outline of paper}In \S2 we begin with some basics on Shimura varieties and define the quaternionic Shimura varieties of interest. We recall the construction of the auxiliary unitary Shimura varieties of PEL-type as in \cite[\S3]{TX}, and recall the description of  Goren--Oort cycles obtained in \cite{TX1}. In \S3 we prove the Hecke orbit conjecture for quaternionic Shimura varieties. We deduce our results from the corresponding statement for the auxiliary unitary Shimura varieties. Using the moduli interpretation, the proof  in the unitary case follows the strategy  of \cite{Chai} who proved the result for Hilbert modular varieties. A key input here is Moonen's generalization of Serre--Tate theory for ordinary abelian varieties \cite{Mo}. In \S4 we study the intersection pairing of cycles on the$\mod l$ reduction of Shimura surfaces and use this to prove Theorem \ref{thm:intro Ihara}. Finally in \S5, we recall the definition of  motivic cohomology groups and higher Chow groups, paying extra attention in the most pertinent case of surfaces, and we construct the level raising subgroup $\mathrm{Ch}^{g/2+1}_{\mathrm{lr}}(\scrS_K(G)_{\bbF_{p^{g}}},1,k_\lambda)_{\fkm}$. We then prove Theorem \ref{thm:intro main} using the strategy outlined above. In the Appendix we describe the bad reduction of quaternionic Shimura surfaces with Iwahori level structure.

\textit{Acknowledgments:} The author would like to thank Tony Feng, Bao Le Hung, Chao Li, Ananth Shankar, Richard Taylor, Liang Xiao and Xinwen Zhu for useful discussions and comments about this work. Above all the author would  like to thank Akshay Venkatesh for his suggestion that there could be interesting arithmetic information contained in the motivic cohomology of Shimura varieties and for many hours of enlightening and enjoyable discussions. The author was partially supported by NSF grant No. DMS-1638352 through membership at the Institute for Advanced Study.

\subsection{Notations}
\begin{itemize}
	
	\item If $F$ is a number field we write $\calO_F$ for its ring of integers. If $v$ is a  place of $F$, we write $F_{v}$ for the completion of $F$ at $v$ and if $v$ is finite, we write $k_v$ for its residue field at $v$.
	\item If $F$ is a local field we write $\calO_F$ for its ring of integers.
	\item For any field $F$, we write $\overline{F}$ for a fixed algebraic closure of $F$.
	\item We write  $\bbA$ for the  ring of adeles and $\bbA_f$ the ring of finite adeles. If $p$ is a prime, $\bbA_f^p$ denotes the finite adeles with trivial $p$-component.
	\item If $R\rightarrow S $ is a map of algebras and $X$ is an $R$-scheme, we write $X_S$ for the base change of $X$ to $S$.
\item If $X$ is a scheme, we write $\scrO_X$ for its structure sheaf.	We write $\rmH^\bullet(X,-)$ for the \'etale cohomology of $X$.  For any closed subscheme $Y\subset X$, we write $\rmH^\bullet_Y(X,-)$ for the \'etale cohomology supported on $Y$.

\end{itemize}

\section{Geometry of quaternionic Shimura varieties and Goren--Oort strata}\label{sec: Section1}
In this section we recall the results concerning the geometry of quaternionic Shimura varieties and Goren--Oort cycles following \cite{TX} and \cite{TX1} that we will need. 

\subsection{Basics on quaternionic Shimura varieties}\label{sec:basics}\ignore{Let $G$ be a reductive group over $\Q$ and $X$ a conjugacy class of homomorphisms
$$h:\mathbb{S}:=\text{Res}_{\mathbb{C}/\mathbb{R}}\mathbb{G}_m\rightarrow G_\mathbb{R}$$
such that $(G,X)$ is a Shimura datum in the sense of \cite[2.1]{De}. In other words, we require the data to satisfy the following conditions:

(SV1) For any $h\in X$, the induced Hodge structure on $\text{Lie}G_{\C}$ is of type $(1,-1),(0,0),(-1,1)$.

(SV2) For any $h\in X$, the adjoint action $h(i)$ on $G_{\ad,\mathbb{R}}$ is a Cartan involution.

(SV3) $G_{\ad}$ has no $\mathbb{Q}$-factor $H$ such that $H(\mathbb{R})$ is compact.

Let $c$ be the complex conjugation. Then $\text{Res}_{\C/\mathbb{R}}(\C)\cong(\C\otimes_{\mathbb{R}}\C)^\times\cong \C^\times \times c^*(\C^\times)$ and we write $\mu_h$ for the cocharacter given by $$\C^\times\rightarrow \C^\times\times c^*(\C^\times)\xrightarrow h G(\C).$$
We set $w_h:=\mu_h^{-1}(\mu_h^{c})^{-1}$.

Let $K_p\subset G(\Q_p)$ and $K^p\subset G(\A_f^p)$ be compact open subgroups and write $K:=K_pK^p$. Then for $K^p$ sufficiently small
$$\text{Sh}_{K}(G,X)_{\C}=G(\Q)\backslash X\times G(\A_f)/K$$
has the structure of an algebraic variety over $\C$. This has a model over the reflex field $E(G,X)$, which is a number field and is the field of definition of the conjugacy class of $\mu_h$; it is naturally a subfield of $\bbC$.

We will also need to consider data $(G,X)$ as above which satisfy only (SV1), (SV2) but not (SV3). Such a datum will be called a {\it weak Shimura datum}. In fact all the weak Shimura data we consider will satisfy the following extra condition:

(SV3') $G_{\ad}(\mathbb{R})$ is compact.

By \cite[p.277]{Bor}, this implies the image of $h$ lands in the center $Z_{\mathbb{R}}$ of $G_{\mathbb{R}}$.}
 Let $F$ be a totally real field with $[F:\Q]=g$ such that $p$ is \textit{unramified} in $F$. We are mainly interested in  the case when $p$ is inert in $F$; this is the case considered in \cite{TX} and \cite{LT}. However, we will sometimes need to consider the reduction mod $l$ of these Shimura varieties, so we will keep the more general assumption for now. We write $\Sigma_p$ (resp. $\Sigma_\infty$) for the set of $p$-adic  places (resp. infinite places). We fix once and for all an isomorphism $\iota_p:\mathbb{C}\cong \overline{\Q}_p$, which we will use to identify $\Sigma_\infty$ with the set of $p$-adic embeddings of $F$. For $\mathfrak{p}\in\Sigma_p$ we let $g_\fkp:=[F_\fkp:\mathbb{Q}_p]$ and $\Sigma_{\infty/\fkp}$ the set of $p$-adic embeddings $\tau\in\Sigma_{\infty}$ which induce $\fkp$. As $p$ is unramified in $F$, the $p$-Frobenius $\sigma$ induces an action on $\Sigma_{\infty/\fkp}$.

We fix a totally indefinite quaternion algebra $B$ over $F$ which is split at all the places above $p$.  Let $\mathrm{S}\subset\Sigma_{\infty}\cup\Sigma_p$ be a subset of even cardinality. We set $\rmS_\infty=\rmS\cap\Sigma_{\infty}$ and for each $\fkp$ we set  $\rmS_{\infty/\fkp}=\rmS\cap\Sigma_{\infty/\fkp}$.  We will make the assumption that $\fkp\in\rmS$ only if $\rmS_{\infty/\fkp}=\Sigma_\infty$.


We write $B_{\mathrm{S}}$ for the quaternion algebra over $F$ whose ramification set is precisely the union of $\mathrm{S}$ and the places in $F$ over which $B$ ramifies. For each $\fkl$ a place of $F$ away from the ramification set for $B_{\rmS}$, we fix an isomorphism $B_{\rmS}\otimes_F F_\fkl\cong GL_2(F_\fkl).$ We define $G_{\mathrm{S}}$ to be the reductive group over $\Q$ such that for any $\Q$-algebra $R$ we have
$$G_{\rmS}(R)=(B_{\rmS}\otimes_{\Q}R)^\times.$$
When $\rmS=\emptyset$, we simply write $G$ for the above group. For $v\notin \rmS$ we have an isomorphism $G(\mathbb{Q}_v)\cong G_\rmS(\mathbb{Q}_v)$. Hence we may fix an isomorphism $$G(\mathbb{A}_f^p)\cong G_\rmS(\mathbb{A}_f^p).$$ 

Let $\rmT\subset \rmS_\infty$ and $\rmT_\fkp:=\rmS_{\infty/\fkp}\cap \rmT$. 
 We consider the following homomorphism:
$$h_{\rmS,\rmT}:\mathbb{S}(\bbR)\cong \mathbb{C}^\times\longrightarrow B_{\rmS}(\mathbb{R})\cong GL_2(\mathbb{R})^{\Sigma_{\infty}-\rmS_\infty}\times\mathbb{H}^{\rmT_\infty}\times \mathbb{H}^{\rmS_\infty-\rmT_\infty}$$ $$
x+yi\longrightarrow \left((x+yi)^{\Sigma_\infty-\rmS{\infty}}, (x^2+y^2)^{\rmT_\infty}, 1^{\rmS_\infty-\rmT_\infty}\right).$$
Then $G_{\rmS}$ and the conjugacy class of $h_{\rmS,\rmT}$ forms a (weak) Shimura datum in the sense of \cite[\S2.2]{TX}. We let $\bfE_{\rmS,\rmT}$ denote the reflex field which is the subfield of the Galois closure $\tilde{F}$ of $F$ in $\bbC$ fixed by the subgroup of $\Gal(\tilde{F}/\bbQ)$ stabilizing $\rmS_{\infty}$ and $\rmT$. We let $v$ be the $p$-adic place of $\bfE_{\rmS,\rmT}$ induced by the embedding $\bfE_{\rmS,\rmT}\hookrightarrow \bbC\cong \overline{\bbQ}_{p}$. We  define the compact open subgroup $K_p:=\prod_{\fkp\in\Sigma_p}K_{\fkp}\subset G_{\rmS}(\bbQ_p)$, where

$\bullet$ $K_\fkp=GL_2(\calO_{F_\fkp})$ if $\fkp\notin \rmS$.

$\bullet$ $K_\fkp=\calO_{B_\fkp}$ the unique maximal compact of $B_{\rmS}\otimes_F F_{\fkp}$ if $\fkp\in\rmS$.

For a sufficiently small compact open subgroup $K^p\subset G(\bbA_f^p)$, we write $K=K_pK^p$ and let $\mathrm{Sh}_K(G_{\rmS,\rmT})$ denote the Shimura variety associated to the above data. We use the notation of \cite{TX} so that $\rmT$ determines the Deligne homomorphism. It is an algebraic variety over $\bfE_{\rmS,\rmT}$ whose complex points are given by
$$\text{Sh}_K(G_{\rmS,\rmT})(\bbC)=G_{\rmS}(\bbQ)\backslash (\fkh^\pm)^{\Sigma_{\infty}-\rmS_{\infty}}\times G_\rmS(\bbA_f)/K$$
where $\fkh^\pm$ is the union of the complex upper and lower half planes. We note that the algebraic variety $\text{Sh}_K(G_{\rmS,\rmT})_{\overline{\bbQ}}$ is independent of $\rmT$. However, different choices of $\rmT$ will give rise to different $\bfE_{\rmS,\rmT}$  varieties, see for example \cite[p9]{TX1}. We also point out the abuse of notation here, where the compact open subgroup $K$ implicitly depends on the choice of $\rmS$. When $\rmS_{\infty}=\Sigma_{\infty}$, $\mathrm{Sh}_K(G_{\rmS,\rmT})(\overline{\bbQ})$ is a discrete set and the  action of $\Gal(\overline{\bbQ}/\bfE_{\rmS,\rmT})$ can be described explicitly as in \cite[\S2.1]{LT}.

We also set $$\Sh_{K_p}(G_{\rmS,\rmT}):=\lim\limits_{\leftarrow K^p}\Sh_{K}(G_{\rmS,\rmT})$$ and we write $\Sh_{K_p}^\circ(G_{\rmS,\rmT})$ for the neutral connected component of $\Sh_{K_p}(G_{\rmS,\rmT})_{\overline{\bbQ}}$. This is the  component containing the image of the point $$(i^{\Sigma_{\infty}-\rmS_{\infty}},1)\in(\fkh^\pm)^{\Sigma_{\infty}-\rmS_{\infty}}\times G_\rmS(\bbA_f).$$

\subsection{Unitary Shimura varieties}\label{sec:unitary}
In this section we define an auxiliary unitary Shimura variety which is of PEL-type in order to define integral models for $\mathrm{Sh}_{K}(G_{\rmS,\rmT})$.

Let $E/F$ be a CM-extension such that the following two conditions are satisfied:

(1) $E/F$ is inert at every place that $B_\rmS$ is ramified

(2) For $\fkp\in\Sigma_p$, $E/F$ is split at $\fkp$ if $\fkp\notin \rmS$ and $\Sigma_{\infty}-\rmS_{\infty/\fkp}$  is even, and is inert  if $\fkp\notin \rmS$ and $\Sigma_{\infty}-\rmS_{\infty/\fkp}$  is odd or if $\fkp\in\rmS$.

Let $\Sigma_{E,\infty}$ denote the set of complex embeddings of $E$, which we identify with the set of $p$-adic embeddings via $\iota_p$. For $\tilde{\tau}\in \Sigma_{E,\infty}$, we let $\tilde{\tau}^c$ denote its complex conjugate. For $\fkp\in \Sigma_p$, let $\fkp\in\Sigma_{E,\infty/\fkp}$ denote the subset of $p$-adic places of $E$ inducing $\fkp$. Similarly for $\fkq$ a $p$-adic place of $E$, we let $\Sigma_{E,\infty/\fkq}$ denote the set of $p$-adic places of $E$ inducing $\fkq$.

We choose a subset $\tilde{\rmS}_\infty\subset \Sigma_{E,\infty}$ satisfying the  property that for each $\fkp\in\Sigma_p$, the natural restriction map $\Sigma_{E,\infty/\fkp}\rightarrow \Sigma_{\infty/\fkp}$ induces a bijection $\tilde{\rmS}_{\infty/\fkp}\xrightarrow{\sim} \rmS_{\infty/\fkp}$, where $\tilde{\rmS}_{\infty/\fkp}=\tilde{\rmS}_{\infty}\cap\Sigma_{E,\infty/\fkp}$


For each $\tilde{\tau}\in \Sigma_{E,\infty}$, the choice of $\tilde{\rmS}_\infty$ determines a collection of numbers $s_{\tilde{\tau}}\in\{0,1,2\}$  given by:
\begin{equation}\label{eq:s tau}
s_{\tilde{\tau}}=\begin{cases} 0 & \text{if $\tilde{\tau}\in\tilde{\rmS}$}\\
	2 & \text{if $\tilde{\tau}^c\in\tilde{\rmS}$}\\
	1 &\text{otherwise}
\end{cases}
\end{equation}

Let $\tilde{\rmS}=(\rmS,\tilde{\rmS}_\infty)$ and $T_{E}:=\mathrm{Res}_{E/\bbQ}(\bbG_m)$. We let $K_{E,p}\subset T_E(\bbQ_p)$ denote the compact open  $(\calO_E\otimes_{\Z}\bbZ_p)^\times$. We define the homomorphism$$h_{E,\tilde{\rmS},\rmT}:\mathbb{S}(\bbR)\cong \mathbb{C}^\times\longrightarrow T_E(\mathbb{R})\cong \prod_{\tau\in\Sigma_{\infty}}(E\otimes_{F,\tau}\bbR)^\times \cong (\bbC^\times)^{\rmS_{\infty}-\rmT}\times (\bbC^\times)^{\rmT}\times(\bbC^\times)^{\Sigma_{\infty}-\rmS_{\infty}} $$ $$
z=x+yi\longrightarrow \left((\overline{z},\cdots,\overline{z}), (z^{-1},\cdots, z^{-1}), (1,\cdots,1)\right)$$
Here for $\tau\in\rmS_\infty$, we identify $(E\otimes_{F,\tau}\bbR)^\times$ with $\bbC^\times$  via the embedding $\tilde{\tau}:E\rightarrow \bbC$, where $\tilde{\tau}\in \tilde{\rmS}_{\infty}$ is the unique lift of $\tau$.
We use the above data to define a Shimura datum for a unitary similitude group which will give rise to a moduli interpretation of the unitary Shimura variety. 

 Let $D_{\rmS}:=B_{\rmS}\otimes E$, which is isomorphic to $\mathrm{Mat}_2(E)$ by our assumptions on $E$. We let $b\mapsto \overline{b}$ denote the involution on $D_{\rmS}$ defined by the product of canonical involution on $B_{\rmS}$ and complex conjugation on $E/F$. Fix a totally imaginary element $\alpha\in E$ such that $\alpha$ is a $\fkp$-adic unit for every place above $p$. Choose an element $\delta\in D_{\rmS}$ such that $\overline{\delta}=\delta$ as in \cite[Lemma 3.8]{TX}. We define a new involution of $D_\rmS$ by $b\mapsto b^*=\delta^{-1}\overline{b}\delta$.

We consider $W:=D_{\rmS}$ as a right $D_{\rmS}$-module of rank 1. It is equipped with the following pairing \begin{equation}
\label{eq:pairing}
\psi:W\times W\rightarrow \bbQ,\ \ \psi(x,y)=\mathrm{Tr}_{E/\bbQ}(\mathrm{Tr}^\circ_{D_\rmS/E}(\alpha x\delta y^*))\end{equation}where $\mathrm{Tr}^\circ_{D_\rmS/E}$ is the reduced trace. It is easy to see this pairing is alternating and non-degenerate. Moreover it satisfies the following property
$$\psi(bx,y)=\psi(x,b^*y), \ b\in D_{\rmS}^\times.$$
The unitary similitude group is defined to be $$G'_{\tilde{\rmS}}(\bbQ):=\{g\in GL_{D_{\rmS}}(W)| \psi(xg,yg)=c(g)\psi(x,y), \text{ for some } c(g)\in \bbQ^\times\}$$
which arises as the $\bbQ$-points of a reductive group $G'_{\tilde{\rmS}}$ over $\bbQ$. We may also describe this group in the following way. Since $GL_{D_{\rmS}}(W)\cong D_{\rmS}^\times,$ we have $g\in D_{\rmS}^\times$ lies in $G'_{\tilde{\rmS}}(\bbQ)$ if and only if $$\mathrm{Tr}_{E/\bbQ}(\mathrm{Tr}^\circ_{D_\rmS/E}(\alpha xg\delta g^*y^*))=c(g)\mathrm{Tr}_{E/\bbQ}(\mathrm{Tr}^\circ_{D_\rmS/E}(\alpha x\delta y^*)), \ \forall x,y\in D_{\rmS}.$$
This identity is equivalent to  $g\delta g^*=c(g)\delta$, i.e. $g\overline{g}=c(g)\in \bbQ^\times$. Thus $$G'_{\tilde{\rmS}}(\Q)=\{g=(b,t)\in B_{\rmS}^\times\times^{F^\times} E^\times| \nu_{\rmS}(b)\mathrm{Nm}_{E/F}(t)\in\bbQ^\times\}$$
where $\nu_{\rmS}:B_\rmS\rightarrow F$ is the reduced norm. Here $B_{\rmS}^\times\times^{F^\times} E^\times$ denotes the quotient of $B_{\rmS}^\times\times E^\times$ by the central embedding $F^\times\rightarrow B_{\rmS}^\times\times E^\times$ given by $x\mapsto (x,x^{-1})$.

Let $T_F$ denote the torus $\mathrm{Res}_{F/\bbQ}\bbG_m$. We define the group $$G''_{\tilde{\rmS}}:=G_{\tilde{\rmS}}\times^{T_F} T_E.$$Applying the above to points valued in a $\bbQ$-algebra, we see that $G'_{\tilde{\rmS}}$ is identified with the subgroup of $G_{\tilde{\rmS}}''$ corresponding to the preimage of $\bbG_m\subset T_F$ under the map $$N:G_{\tilde{\rmS}}''\rightarrow T_F, \ \ (b,t)\mapsto \nu_{\rmS}(b)\mathrm{Nm}_{E/F}(t).$$

We now let $h''_{\tilde{\rmS}}:\bbS\rightarrow G''_{\tilde{\rmS}}$ denote the morphism induced by $(h_{{\rmS},\rmT},h_{E,\tilde{\rmS},\rmT})$; it is independent of $\rmT$. The image of $h''_{{\tilde{\rmS}}}$ lies in $G'_{\tilde{\rmS}}$ and we let $h'_{\tilde{\rmS}}$ denote the induced map. Let $K_p''\subset G''_{\tilde{\rmS}}(\bbQ_p)$ the compact open  subgroup given by the image of $K_p\times K_{E,p}$ and $K'_p=G'_{\tilde{\rmS}}(\bbQ_p)\cap K_p''$. For sufficiently small compact open subgroups $K''^{p}\subset G_{\tilde{\rmS}}''(\bbA_f^p)$ and $K'^{p}\subset G'_{\tilde{\rmS}}(\bbA_f^p)$, we obtain Shimura varieties $\mathrm{Sh}_{K''}(G''_{\tilde{\rmS}}), \mathrm{Sh}_{K'}(G'_{\tilde{\rmS}})$ where $K''=K''_pK''^p$ and $K'=K'_pK'^p$. We also set $$\mathrm{Sh}_{K'_p}(G'_{\tilde{\rmS}}):=\varprojlim_{K'^p}\mathrm{Sh}_{K'}(G'_{\tilde{\rmS}}),\ \ \mathrm{Sh}_{K_p''}(G''_{\tilde{\rmS}}):=\varprojlim_{K''^p}\mathrm{Sh}_{K''}(G''_{\tilde{\rmS}})$$

Let $\bfE_{\tilde{\rmS}}$ denote the common reflex field of these Shimura varieties, which is a subfield of $\bbC$.  The isomorphism $\iota_p:\overline{\bbQ}_p\cong \bbC$ induces a $p$-adic place $\tilde{v}$ of $\bfE_{\tilde{\rmS}}$. We let $\mathrm{Sh}_{K''_p}(G''_{\tilde{S}})^\circ$ (resp. $\mathrm{Sh}_{K'_p}(G'_{\tilde{\rmS}})^\circ$) denote the neutral geometric connected component of $\mathrm{Sh}_{K''_p}(G''_{\tilde{\rmS}})_{\overline{\bbQ}}$ (resp. $\mathrm{Sh}_{K'_p}(G'_{\tilde{\rmS}})_{\overline{\bbQ}}$). Then both $\mathrm{Sh}_{K''_p}(G''_{\tilde{\rmS}}))_{\overline{\bbQ}_p}^\circ$ and $\mathrm{Sh}_{K'_p}(G'_{\tilde{\rmS}})_{\overline{\bbQ}_p}^\circ$  can be descended to $\bbQ_p^{\mathrm{ur}}$.

We have the following diagram of groups $$G_{\rmS}\leftarrow G_{\rmS}\times T_E\rightarrow G''_{\tilde{\rmS}}\leftarrow G'_{\tilde{S}}$$ compatible with Deligne homomorphisms and such that the induced maps on the derived and adjoint groups are isomorphisms. By \cite[Corollary 2.16]{TX} this induces  isomorphisms
$$\mathrm{Sh}_{K_p}(G_{\rmS})_{\overline{\bbQ}_p}^\circ\xleftarrow{\sim}\mathrm{Sh}_{K''_p}(G''_{\tilde{\rmS}})_{\overline{\bbQ}_p}^\circ\xrightarrow{\sim}\mathrm{Sh}_{K'_p}(G'_{\tilde{\rmS}})_{\overline{\bbQ}_p}^\circ$$
of neutral geometric connected components. Since Shimura varieties may be constructed from its neutral connected component and the action of Hecke and Galois, we may transfer integral models from one group to the other. See \cite[\S2.11]{TX1} for the details.
\subsection{Moduli interpretation for unitary Shimura varieties and integral models}\label{sec:moduli interpretation}
The Shimura variety $\mathrm{Sh}_{K'}(G'_{\tilde{\rmS}})$ is a Shimura variety of PEL-type and thus admits an integral model as a moduli space. Recall the $D_\rmS$-module  $W$ together with the non-degenerate alternating form $\psi$ introduced in the last subsection. We also fix some integral PEL data. Let $\calO_{D_{\rmS}}\subset D_{\rmS}$ be an order which is maximal at $p$ and $\Lambda\subset W$ an $\calO_{D_{\rmS}}$-lattice such that $\psi(\Lambda,\Lambda)\subset\Z$ and $\Lambda\otimes \bbZ_p$ is self-dual. Let $\rmK'^p$ be a sufficiently small compact open subgroup of $G'_{\tilde{S}}$ stabilizing $\Lambda\otimes\widehat{\bbZ}^p$.
We consider the moduli functor $\underline{\mathrm{Sh}}_{K'}(G'_{\tilde{\rmS}})$ that associates to each $\calO_{\bfE_{\tilde{\rmS},\tilde{v}}}$-scheme $S$ the set of isomorphism classes of triples $(A,\iota,\lambda,\epsilon_{K'^p})$ where:

$\bullet$ $A$  is an abelian scheme  over $S$ of dimension $4[F:\bbQ]$.

$\bullet$ $\iota:\calO_{D_{\rmS}}\hookrightarrow\mathrm{End}_S(A)$ is an embedding such that the induced action of $a\in\calO_{D_{\rmS}}$ acting on $\mathrm{Lie}(A/S)$ satisfies
\begin{equation}\label{eq:char poly 1}\det(T -\iota(a)| \mathrm{Lie}(A/S)) = \prod_{\tilde{\tau}\in\Sigma_{E,\infty}}(T-\tilde{\tau}(a))^{2s_{\tilde{\tau}}}\end{equation}

$\bullet$ $\lambda:A\rightarrow A^\vee$ is a  polarization such that:

-The Rosati involution on $\mathrm{End}_S(A)$ defined  by $\lambda$ induces the involution $b\mapsto b^*$ on $\calO_{D_{\rmS}}$.

-If $\fkp\notin \rmS$, $\Lambda$ induces an isomorphism of $p$-divisible groups $A[\fkp^\infty]\cong A^\vee[\fkp^\infty]$.

-If $\fkp\in \rmS$, $(\ker\lambda)[\fkp^\infty]$ is a finite flat group scheme contained in $A[\fkp]$ of rank $p^{4g_{\fkp}}$ and the cokernel of the map $\lambda_*:\rmH_1^{\mathrm{dR}}(A/S)\rightarrow \rmH_1^{\mathrm{dR}}(A^\vee/S)$ is a locally free module of rank 2 over $\calO_S\otimes_{\Z_p}\calO_E/\fkp$. Here $\rmH_1^{\mathrm{dR}}(A/S)$ denotes the relative de Rham homology.

$\bullet$ $\epsilon_{K'^p}$ is a $K'^p$-level structure, i.e. a  $K'^p$-orbit of isomorphisms
$$\epsilon_{K'^p}: \Lambda\otimes_{\Z}\widehat{\Z}^p\cong \widehat{T}^pA$$ which respects the action of $\calO_{D_{\rmS}}$ on both sides and preserves the pairings on both sides.  Here $\widehat{T}^pA:=\lim\limits_{\leftarrow p\nmid n}A[n]$ denotes the prime-to-$p$ Tate module of $A$.

By \cite[Theorem 3.14]{TX}, the moduli problem $\underline{\mathrm{Sh}}_{K'}(G'_{\tilde{\rmS}})$ is representable by a smooth quasi-projective scheme over $\calO_{\mathbf{E}_{\tilde{\rmS},\tilde{v}}}$, and its generic fiber is identified with $\mathrm{Sh}_{K'}(G'_{\tilde{\rmS}})_{\bfE_{\rmS,\tilde{v}}}$. Moreover it is an integral canonical model for $\mathrm{Sh}_{K'}(G'_{\tilde{\rmS}})$ in the sense of \cite[\s2.4]{TX}, see also \cite{Milne}. We write $$\underline{\mathrm{Sh}}_{K'_p}(G'_{\tilde{\rmS}}):=\lim\limits_{\leftarrow K'^p}\underline{\mathrm{Sh}}_{K'_pK'^p}(G'_{\tilde{\rmS}}).$$Taking the closure of $\mathrm{Sh}_{K'_p}(G'_{\tilde{\rmS}})^\circ$ in $\underline{\mathrm{Sh}}_{K'_p}\otimes \bbZ_p^{\mathrm{ur}}$ and using Deligne's recipe to transfer across to $G_{\rmS}$ we obtain an integral canonical model $\underline{\mathrm{Sh}}_{K}(G_{\rmS,\rmT})$ for $\mathrm{Sh}_{K}(G_{\rmS,\rmT})$; see \cite[\S2.11]{TX}. We write $\mathscr{S}_{K}(G_{\rmS,\rmT})$ (resp. $\mathscr{S}_{K'}(G'_{\tilde{\rmS}})$) for the special fiber of $\underline{\mathrm{Sh}}_K(G_{\rmS,\rmT})$  (resp. $\underline{\mathrm{Sh}}_{K'}(G'_{\tilde{\rmS}}$)) over $k_v$ (resp. $k_{\tilde{v}}$).  We let $k_0$ be any finite field containing all the residue fields of the $p$-adic places of $E$.

We fix an isomorphism $\calO_{D_{\rmS,p}}:=\calO_{D_{\rmS}}\otimes_{\bbZ}\bbZ_p \cong M_{2\times 2}(\calO_E\otimes_{\bbZ}\bbZ_p)$ and let $e$ denote the idempotent of $M_2(\calO_{E}\otimes_{\bbZ}\bbZ_p)$ given by $$e=\left(\begin{matrix}
	1 & 0 \\
	0 & 0
\end{matrix}\right).$$   Suppose $M$ is a module with an action of $\calO_{D_{\rmS,p}}$. We write $M^\circ$ for the sub $\calO_{E}\otimes_{\bbZ}\bbZ_p$-module $eM\subset M$.

Now suppose $k$ is a perfect field of characteristic $p$ containing all residue fields at the $p$-adic places of $E$. Let $(A,\iota,\lambda,\epsilon_{K'^p})$ be the data corresponding to an $S$-point of $\rmSh_{K'}(G_{\tilde{\rmS}})$ where $S$ is a $k_{\tilde{v}}$-scheme. Then we have an exact sequence:
\[\xymatrix{0 \ar[r]& \omega_{A^\vee}^\circ\ar[r]& \rmH_1^{\mathrm{dR}}(A/S)^\circ\ar[r]& \mathrm{Lie}(A)^\circ \ar[r]& 0 }.\]
Here $\omega_{A^\vee}$ is the  module of invariant differential forms for the dual abelian variety $A^\vee$ and $\rmH_1^{\mathrm{dR}}(A/S)$ is the first relative de Rham homology; see \cite[4.1]{TX}. For $\tilde{\tau}\in\Sigma_{E,\infty}$, we use a subscript $\tilde{\tau}$ to denote the subspace over which $\calO_{E,p}:=\calO_E\otimes_{\bbZ}\bbZ_p$ acts via $\tilde{\tau}$. Then the above induces an exact sequence
\begin{equation}\label{eq3}
\xymatrix{0 \ar[r]& \omega_{A,\tilde{\tau}}^\circ\ar[r]& \rmH_1^{\mathrm{dr}}(A/S)_{\tilde{\tau}}^\circ\ar[r]& \mathrm{Lie}(A)_{\tilde{\tau}}^\circ \ar[r]& 0 }\end{equation}
and the dimensions of these three factors are $2-s_{\tilde{\tau}}$, $2$ and $s_{\tilde{\tau}}$ respectively. In particular, for $\tau\in\Sigma_{\infty}-\rmS_\infty$, and $\tilde{\tau}$ a lift of $\tau$, $\omega_{\tilde{\tau}}^\circ$ is a line bundle over $S$.

 For simplicity, we will write $X'$ for $\scrS_{K'}(G'_{\tilde{\rmS}})_{k_0}$  and $(\calA',\iota,\lambda,\epsilon_{K'^p})$ the universal abelian variety over ${X'}$. Recall  the Kodaira--Spencer map $$\text{KS}_{(\calA',\iota,\lambda,\epsilon_{K'^p})} : \omega_{\calA'}\otimes_{\scrO_{X'}}\omega_{\calA'^\vee}\rightarrow\Omega^1_{X'/k_{0}}$$
which induces a map $$\text{KS}^\circ_{(\calA',\iota,\lambda,\epsilon_{K'^p})} : \omega_A^\circ\otimes_{\scrO_{X'}}\omega_{\calA'^\vee}^\circ\rightarrow\Omega^1_{{X'}/k_{0}}.$$

\begin{prop}\label{prop: KS}
The  map $\text{KS}^\circ_{({\calA'},\iota,\lambda,\epsilon_{K'^p})}$ factors as $$\omega^\circ_{\calA'}\otimes_{\scrO_{X'}}\omega_{{\calA'}^\vee}^\circ\rightarrow\omega^\circ_{\calA'}\otimes_{\scrO_{X'}\otimes\calO_{E_{\fkq}}}\omega^\circ_{{\calA'}^\vee}\rightarrow \Omega^{1}_{{X'}/k_{0}}.$$ 
and induces an isomorphism $\omega^\circ_{\calA'}\otimes_{\scrO_{X'}\otimes\calO_{E_{\fkq}}}\omega^\circ_{{\calA'}^\vee}\cong \Omega^{1}_{{X'}/k_{0}}.$
\end{prop}
\begin{proof}
This follows from  \cite[Proposition 2.3.4.1]{Lan} applied to the universal abelian variety over $X'$.
\end{proof}

\begin{cor}\label{cor: Kodaira_Spencer unitary}
The map $KS_{({\calA'},\iota,\lambda,\epsilon_{K'^p})}^\circ$ induces an isomorphism of line bundles $$\bigotimes_{\tau\in\Sigma_{\infty}-\rmS_\infty}(\omega^\circ_{{\calA'}^\vee,\tilde{\tau}})\otimes_{\scrO_{X'}} (\omega^\circ_{{\calA'}^\vee,\tilde{\tau}^c})\cong \Omega_{{X'}/k_{0}}^{|\Sigma_\infty-\rmS_{\infty}|}$$
where $\tilde{\tau}$ is a lift of $\tau$.
\end{cor}
\begin{proof} We have an isomorphism $$\omega_{\calA'}^\circ\otimes_{\scrO_{X'}\otimes\calO_{E_{\fkq}}}\omega^\circ_{{\calA'}^\vee}\cong \bigoplus_{\tau\in\Sigma_{\infty}}\omega^\circ_{{\calA'},\tilde{\tau}}\otimes\omega^\circ_{{\calA'}^\vee,\tilde{\tau}}.$$
	Now for $\tau\in\rmS_{\infty}$, either $\omega^\circ_{{\calA'},\tilde{\tau}}$ or $\omega^\circ_{{\calA'}^\vee,\tilde{\tau}}$ has rank $0$. 
	
	For $\tau\in \Sigma_{\infty}-\rmS_\infty$, the polarization induces an isomorphism $\omega_{{\calA'},\tilde{\tau}}\cong\omega_{{\calA'}^\vee,\tilde{\tau}^c}$ compatible with the action of $\calO_{D_\rmS}$ hence induces an isomorphism $\omega^\circ_{{\calA'},\tilde{\tau}}\cong\omega^\circ_{{\calA'}^\vee,\tilde{\tau}^c}$. Thus we may rewrite the isomorphism of Proposition \ref{prop: KS} as $$\bigoplus_{\tau\in\Sigma_\infty-\rmS_{\infty}}(\omega^\circ_{{\calA'}^\vee,\tilde{\tau}})\otimes_{\scrO_{X'}} (\omega^\circ_{{\calA'}^\vee,\tilde{\tau}^c})\cong\Omega_{{X'}/\calO_{k_{0}}}^1.$$
Taking the top exterior power gives the result.
\end{proof}

Now let $X$ denote the special fiber $\scrS_K(G_{\rmS,\rmT})_{k_0}$. Using \cite[Corollary 2.13]{TX}, we may transfer the vector bundles $\omega_{\calA'^\vee,\tilde{\tau}}^\circ$ from $X'$ to $X$. For each $\tau\in \Sigma_{\infty}-\rmS_{\infty}$, we write $\omega_{\tau}$ to be the line bundle on $X$ coming from $\omega_{\calA'^\vee,\tilde{\tau}}^\circ$ on $X'$ for some choice of lift $\tilde{\tau}$ of $\tau$ which we now fix. Then $\omega_{\tau}$ is independent of the choice of lift $\tilde{\tau}$ up to a torsion element in the Picard group of $X$ by \cite[Lemma 6.2]{TX}. For any line bundle $\scrL$ on $X$, we write $[\scrL]$ for the image of $\scrL$ in the rational Picard group. The next Corollary follows immediately from Corollary \ref{cor: Kodaira_Spencer unitary}.

\begin{cor}\label{cor:Kodaira Spencer quaternionic}
	There exists an isomorphism $$\left[\bigotimes_{\tau\in\Sigma_{\infty}-\rmS_\infty}\omega_\tau^2\right]\cong \left[\Omega_{{X}/k_0}^{|\Sigma_\infty-\rmS_{\infty}|}\right]$$ in the rational Picard group of $X$.
\end{cor}

\subsection{Goren-Oort divisors and Goren-Oort cycles}\label{sec: Goren-Oort}
We now let $k_0$ be the smallest subfield of $\Fpbar$ containing the residue fields of the $p$-adic places of $E$. Then $k_0=\bbF_{p^h}$ where $h$ is the least common multiple of $\{1+g_{\fkp}-2\lfloor g_{\fkp}/2\rfloor|\fkp\in\Sigma_{\infty}\}$. In this section we  will recall the description of  Goren--Oort divisors and Goren--Oort cycles in $\sshG_{k_0}$ obtained in \cite{TX} and \cite{TX1}.  As before $\rmS\subset \Sigma_{p}\cup\Sigma_{\infty}$ is a set of even cardinality. For each $\tau\in \Sigma_{\infty}-S_\infty$, \cite{TX} defines a Goren-Oort divisor $\sshG_{k_0,\tau}\subset\sshG_{k_0}$ by transferring over a certain divisor on an auxiliary unitary Shimura variety.  We briefly recall its construction. 

Let $(A,\iota,\lambda,\epsilon_{K'^p})$ be an $S$-point of $\scrS_{K'}(G'_{\tilde{\rmS}})$ for $S$ a $k_{0}$-scheme. We define the {\it essential Verschiebung} $$V_{\mathrm{es},\tilde{\tau}}:\rmH_1^{\mathrm{dR}}(A/S)_{\tilde{\tau}}^\circ\rightarrow \rmH_1^{\mathrm{dR}}(A^{(p)}/S)^\circ_{\tilde{\tau}}\cong \rmH_1^{\mathrm{dR}}(A/S)^{\circ,(p)}_{\sigma^{-1}\tilde{\tau}}$$ to be the usual Verschiebung if $s_{\sigma^{-1}\tilde{\tau}}=0$ or $1$, and to be the inverse of the usual Frobenius if $\sigma^{-1}(s_{\tilde{\tau}})=2$. Note that $V_{\mathrm{es},\tilde{\tau}}$ preserves the exact sequence (\ref{eq3}). For every integer $n\geq 1$ we define $$V_{\mathrm{es}}^n:\rmH_1^{\mathrm{dR}}(A/S)^{\circ}_{\tilde{\tau}}\rightarrow\rmH_1^{\mathrm{dR}}(A^{(p^n)}/S)^\circ_{\tilde{\tau}}\cong \rmH_1^{\mathrm{dR}}(A/S)^{\circ,(p^n)}_{\sigma^{-n}(\tilde{\tau})}$$
to be the $n^{\mathrm{th}}$-iteration of the essential Verschiebung.

For $\tau\in\Sigma_{\infty}-\rmS_{\infty}$, we define the integer $n_\tau$ to be the smallest integer such that $\sigma^{-n_\tau}\tau\in\Sigma_\infty-\rmS_{\infty}$. Now for each $\tilde{\tau}\in\Sigma_{E,\infty}$ with $s_{\tilde{\tau}}=1$, the restriction of $V_{\text{es}}^{n_\tau}$ to $\omega^{\circ}_{A^\vee/S}$ defines a map
$$h_{\tilde{\tau}}(A):\omega^{\circ}_{A^\vee,\tilde{\tau}}\rightarrow\omega^{\circ,(p^{n_\tau})}_{A^\vee,\sigma^{-n_\tau}\tilde{\tau}}\cong(\omega^\circ_{A^\vee,\sigma^{-n_\tau}\tilde{\tau}})^{\otimes p^{n_\tau}}.$$ Applying this to  the universal abelian variety $\mathcal{A}'$ over $\mathscr{S}_{K'}(G'_{\tilde{S}})_{k_0}$, we obtain a global section $$h_{\tilde{\tau}}\in \Gamma(\scrS_{K'}(G'_{\tilde{S}},)_{k_0},(\omega^\circ_{\calA'^\vee,\sigma^{-n_\tau}\tilde{\tau}})^{\otimes p^{n_\tau}}\otimes (\omega^{\circ}_{\calA'^\vee,\tilde{\tau}})^{\otimes{-1}}).$$ We call this  the $\tau^{\mathrm{th}}$ partial Hasse-invariant. We let $\scrS_{K'}(G'_{\tilde{\rmS}})_{k_0,\tilde{\tau}}\subset \scrS_{K'}(G'_{\tilde{\rmS}})_{k_0}$ denote the vanishing locus of $h_{\tilde{\tau}}$. We let $\scrS_K(G_{\rmS,\rmT})_{k_0,\tilde{\tau}}$ be the corresponding divisor on $\scrS_K(G_{\rmS,
\rmT})_{k_0}$. By \cite[Lemma 4.5]{TX}, $\scrS_{K'}(G'_{\tilde{\rmS}})_{k_0,\tilde{\tau}}$, and hence $\scrS_K(G_{\rmS,\rmT})_{k_0,\tilde{\tau}}$, is independent of the lifting $\tilde{\tau}$ of $\tau$. We may thus write $\scrS_{K'}(G'_{\tilde{\rmS}})_{k_0,\tau}$ and $\scrS_K(G_{\rmS,\rmT})_{k_0,\tau}$ for these divisors respectively. The divisor $\scrS_K(G_{\rmS,\rmT})_{k_0,\tau}$ is known as the {\it Goren-Oort divisor corresponding to $\tau$}. It can also be described as the vanishing of a certain section of the line bundle of $\omega_{\sigma^{-n_\tau}\tau}^{\otimes (p^{n_\tau})}\otimes\omega_\tau^{-1}$ on $\scrS_K(G_{\rmS,\rmT})_{k_0}$.

For $J\subset \Sigma_{\infty}-S_\infty$, we set $$\sshG_{k_0,J}:=\bigcap_{\tau\in J}\sshG_{k_0,\tau}.$$

The closed subvarieties $\sshG_{k_0,J}$ as $J$ runs over subsets of $\Sigma_{\infty}-\rmS_{\infty}$ give the {\it Goren-Oort stratification} of $\sshG$. 

To state the main structure theorem of \cite{TX}, we introduce the following notation. For $\tau\in \Sigma_{\infty}-\rmS_{\infty}$, we write $\tau^-$ for $\sigma^{-n_\tau}(\tau)$.  We write $\rmT_\tau=\rmT\cup\{\tau\}$ and define $$\rmS_{\tau}=\begin{cases}\rmS_{\infty/\fkp}\cap\{\tau,\tau^-\}&\text{if $\rmS_{\infty/
		\fkp}\cup\{\tau\}\neq\Sigma_{\infty/\fkp}$}\\
	\rmS\cup\{\tau,\fkp\}&\text{if $\rmS_{\infty/\fkp}\cup\{\tau\}=\Sigma_{\infty/\fkp}$}
\end{cases}$$

\begin{thm}[{\cite[Theorem 5.2]{TX}, \cite[Proposition 5.9]{LT}}]\label{thm:Goren-Oort strata}There is a morphism  of $\Fpbar$-varieties $$\pi_{\tau}:\sshG_{\Fpbar,\tau}\rightarrow \scrS_K(G_{\rmS_{\tau},\rmT_{\tau}})_{\Fpbar}$$ equivariant for the prime-to-$p$ Hecke action, such that
	
(1) if $\rmS_{\infty/
		\fkp}\cup\{\tau\}\neq\Sigma_{\infty/\fkp}$, $\pi_{\tau}$  is a $\bbP^1$-fibration which descends to a morphism $\sshG_{k_0,\tau}\rightarrow \scrS_K(G_{\rmS_{\tau},\rmT_{\tau}})_{k_0}$ of $k_0$-varieties.
	
	(2) if $\rmS_{\infty/
		\fkp}\cup\{\tau\}=\Sigma_{\infty/\fkp}$, $\pi_\tau$ is an isomorphism.
	
\end{thm}
Note that in case (2) $\scrS_K(G_{\rmS_{\tau},\rmT_{\tau}})_{\Fpbar}$ is discrete, and the level structure at $\fkp$ needs to be modified.

We now recall the framework for the construction of Goren--Oort cycles. These cycles are parametrized by certain combinatorial objects called periodic semi-meanders whose definition we now recall. We refer to \cite[\S3.4]{LT} for more details.

For $\fkp$ a place of $F$ above $p$, let $d_\fkp(\rmS):=g_\fkp-|S_{\infty/\fkp}|$. We consider the cylinder $C:=\{x^2+y^2=1\}$ in 3-dimensional space. Let $\Sigma_{\infty/\fkp}=\{\tau_0,...,\tau_{g_\fkp-1}\}$ where $\sigma(\tau_j)=\tau_{j+1}$ for $j\in\mathbb{Z}/g_\fkp\Z$; we use the $\tau_j$ to label the points on the $xy$-plane by identifying $\tau_j$ with the point $(\text{cos}\frac{2\pi j}{g_\fkp},\text{sin}\frac{2\pi j}{g_\fkp})$. If $\tau_j\in \rmS_{\infty/\fkp}$ we put a $+$ at the point $\tau_j$, otherwise we put a node. We call the series of nodes and $+$'s the {\it band} associated to $S_{\infty/\fkp}$.

A periodic semi-meander for $S_{\infty/\fkp}$ is a series of arcs and semi-lines on $C$ (an arc connects two nodes and a semi-line connects a node and $+\infty$) satisfying the following properties:
 \begin{enumerate}
\item All arcs and semi-lines lie above the band.

\item Every node is the end-point of an arc of semi-line.

\item There are no intersection points among the arcs and semi-lines.
\end{enumerate}
We identify two periodic semi-meanders if one can be continuously deformed into the other. We write $r$ for the number of arcs in a periodic semi-meander and $d_{\fkp}(\rmS)-r$ is the defect. We write $\fkB(\rmS_{\infty/\fkp},r)$ for the set of periodic semi-meanders with $r$ arcs.

For any $\fka\in\fkB(\rmS_{\infty/\fkp},r)$ we define the sets $S_\fka$ and $T_\fka$ to be \begin{equation}\rmS_\fka:=\rmS\cup\{\tau\in\Sigma_\infty| \text{$\tau$ is an end-point of an arc in $\fka$}\}\end{equation} \begin{equation}\label{eq:def of T semi-meander}\rmT_\fka:=\rmT\cup\{\tau\in\Sigma_{\infty}| \text{$\tau$ is the right end-point of an arc in $\fka$}\}\end{equation}

To any periodic semi-meander $\fka\in \fkB(\rmS_{\infty/\fkp},r)$,  \cite[\S3.5]{LT} constructs the {\it Goren-Oort} cycle  $$Z_{\rmS,\rmT}(\fka)\hookrightarrow \sshG_{k_0}.$$
 corresponding to $\fka$ using the method of \cite[\S3.7]{TX1}. The construction is by induction on $r$ and the resulting cycle is an $r^{\mathrm{th}}$-iterated $\bbP^1$-bundle over the Shimura variety $\scrS_K(G_{\rmS_\fka,\rmT_{\fka}})$. When $r=0$,  $Z_{\rmS,\rmT}(\fka)$ is defined to be $\scrS_{K}(G_{{\rmS}})_{k_0}$. For $r\geq 1$, we say an arc $\delta$ in $\fka$ is basic if it does not lie below any other arc. Choose such a basic arc $\delta$ and write $\tau^+$ and $\tau^-$ for its right and left endpoints. Consider the Goren-Oort divisor $\scrS_K(G_{\rmS,\rmT})_{k_0,\tau}$ together with the fibration $\pi_{\tau}:\scrS_K(G_{\rmS,\rmT})_{k_0,\tau}\rightarrow \scrS_K(G_{\rmS_\tau,\rmT_\tau})_{k_0}$. Let $\fka_\delta\in\fkB(\rmS_{\tau,\infty/\fkp},r-1)$ denote the periodic semi-meander given by removing the arc $\delta$ from $\fka$ and replacing the endpoints with $+$ signs. By induction hypothesis, we have the cycle $Z_{\rmS_\tau,\rmT_\tau}(\fka_\delta)\subset\scrS_K(G_{\rmS_{\tau},\rmT_{\tau}})_{k_0}$ which is an $(r-1)^{\mathrm{th}}$-iterated $\bbP^1$-bundle over $\scrS_K(G_{\rmS_{\fka},\rmT_{\fka}})_{k_0}$. The Goren-Oort cycle $Z_{\rmS,\rmT}(\fka)$ is defined to be the preimage of $Z_{\rmS_\tau,\rmT_\tau}(\fka_\delta)$ in $\scrS_K(G_{\rmS,\rmT})_{k_0}$ under the projection $\pi_{\tau}$. The following Proposition is clear from the construction and Theorem \ref{thm:Goren-Oort strata}.
\begin{prop}\label{prop: Goren-Oort fibration}  $Z_{\rmS,\rmT}(\fka)$ is an $r^{\mathrm{th}}$-iterated $\mathbb{P}^1$-fibration over the Shimura variety $\sshGa_{k_0}$.
Moreover the inclusion $Z_{\rmS,\rmT}(\fka)\rightarrow\sshG_{k_0}$ is equivariant for the prime-to-$p$ Hecke correspondences. \end{prop}

We write $\pi_{\fka}:Z_{\rmS,\rmT}(\fka)\rightarrow \sshGa_{k_0}$ for the projection map.

\subsection{Goren--Oort cycles and Shimura surfaces}\label{sec:Goren--Oort dim 2} We now give a more detailed description of the Goren--Oort cycles which will be used in the construction of the level raising subgroup in motivic cohomology. 

For this we will impose the extra assumptions that $[F:\bbQ]=g$ is even and that $p$ is inert in $F$; then $\fkp$ will denote the unique prime above $p$. In this case we may take $k_0=\bbF_p^g$. We consider the set $\fkB(\emptyset,g/2-1)$ of periodic semi-meanders which is easily seen to have ${g\choose g/2-1}$ elements.

Fix $\fka\in\fkB(\emptyset,g/2-1))$. Then we have constructed the Goren--Oort cycle $Z_\emptyset(\fka)\subset \scrS_K(G)_{\bbF_{p^g}}$ and the projection $$\pi_{\fka}: Z_\emptyset(\fka)\rightarrow \scrS_K(G_{\emptyset_{\fka},\rmT_{\fka}})_{\bbF_{p^g}}$$ which is ($g/2-1$)-iterated $\bbP^1$-fibration; here $\rmT_{\fka}$ is defined as in (\ref{eq:def of T semi-meander}) with $\rmT=\emptyset$. Since $\scrS_K(G_{\emptyset_{\fka},\rmT_{\fka}})_{\bbF_{p^g}}$ is the special fiber of a Shimura surface, the Goren--Oort cycle $Z_\emptyset(\fka)$ has dimension $g/2+1$. 

We write $\emptyset_\fka=\{\tau_i,\tau_j\}$. Then we have the Goren--Oort divisors $\scrS_K(G_{\emptyset_{\fka},\rmT_{\fka}})_{\bbF_{p^g},\tau_i}$ 
and $\scrS_K(G_{\emptyset_{\fka},\rmT_{\fka}})_{\bbF_{p^g},\tau_j}$ of the Shimura surface $\scrS_K(G_{\emptyset_{\fka},\rmT_{\fka}})_{\bbF_{p^g}}$ which are $\bbP^1$-bundles over the discrete Shimura sets  $\scrS_K(G_{\emptyset_{\fka,\tau_i},\rmT_{\fka,\tau_i}})_{\bbF_{p^g}}$ and $\scrS_K(G_{\emptyset_{\fka,\tau_j},\rmT_{\fka,\tau_j}})_{\bbF_{p^g}}$ respectively.
The $\Fpbar$-points of these Shimura sets may be identified (upon fixing a base-point) with $$G_{\Sigma_\infty}(\bbQ)\backslash G_{\Sigma_{\infty}}(\bbA_f)/K.$$

We let $K_0(\fkp)\subset G_{\Sigma_{\infty}}(\bbA_f)$ denote the compact open subgroup $K_{0,\fkp}K^p$ where $K_{0,\fkp}$ is the standard Iwahori subgroup of $GL_2(F_\fkp)$ consisting of matrices in $GL_2(\calO_{F_{\fkp}})$ which reduce to the upper triangular matrices mod $\fkp$. In \cite[2.17]{TX}, there is a construction of the integral model for the discrete Shimura set $\Sh_{K_0(\fkp)}(G_{\emptyset_{\fka,\tau_i},\rmT_{\fka,\tau_i}})$; we can  explicitly describe the $\Fpbar$-points of the special fiber $\scrS_{K_0(\fkp)}(G_{\emptyset_{\fka,\tau_i},\rmT_{\fka,\tau_i}})_{\bbF_{p^g}}$ and the action of $\Gal(\Fpbar/\bbF_{p^g})$ as follows. We equip the discrete set $G_{\Sigma_\infty}(\bbQ)\backslash G_{\Sigma_{\infty}}(\bbA_f)/K_0(\fkp)$ with the $\Gal(\Fpbar/\bbF_{p^g})$-action where the arithmetic $p^g$-Frobenius $\sigma_{p^g}$ acts via multiplication by the central element $\underline{p}^{-g/2}\in F\otimes_{\bbQ}\A_f^\times\subset G_{\Sigma_{\infty}}(\bbA_f)$. Here $\underline{p}^{-g/2}$ is the idele which is $p^{-g/2}$ at the $p$-adic place $\fkp$ and $1$ elsewhere. Then there is a bijection $$\scrS_{K_0(\fkp)}(G_{\emptyset_{\fka,\tau_i},\rmT_{\fka,\tau_i}})_{\bbF_{p^g}}(\Fpbar)\cong G_{\Sigma_\infty}(\bbQ)\backslash G_{\Sigma_{\infty}}(\bbA_f)/K_0(\fkp).$$
compatible with the Galois action. 

The action of $\sigma_{p^g}$ on $\scrS_{K}(G_{\emptyset_{\fka,\tau_i},\rmT_{\fka,\tau_i}})_{\bbF_{p^g}}(\Fpbar)$ is defined in the same way. We write $n_{K_\infty}$ for the order of $\sigma_{p^g}$ acting on $\scrS_{K}(G_{\emptyset_{\fka,\tau_i},\rmT_{\fka,\tau_i}})_{\bbF_{p^g}}(\Fpbar)$. It is easily seen to be independent of the choice of $\fka$ and $\tau_i$.

The following proposition is contained in \cite[Proposition 2.32]{TX1}.

\begin{prop}\label{prop: Goren--Oort 2 dim}
(1) There is an isomorphism $$\scrS_K(G_{\emptyset_{\fka},\rmT_{\fka}})_{\bbF_{p^g},\{\tau_i,\tau_j\}}\cong \scrS_{K_0(\fkp)}(G_{\emptyset_{\fka,\tau_i},\rmT_{\fka,\tau_i}})_{\bbF_{p^g}}$$ of $\bbF_{p^g}$-varieties.

(2) There is an isomorphism $$\eta_{\tau_j}:\scrS_K(G_{\emptyset_{\fka,\tau_j},\rmT_{\fka,\tau_j}})_{\bbF_{p^g}}\xrightarrow{\sim}\scrS_K(G_{\emptyset_{\fka,\tau_i},\rmT_{\fka,\tau_i}})_{\bbF_{p^g}}$$ of $\bbF_{p^g}$-varieties  such that the induced diagram
\[\xymatrix{& \scrS_{K_0(\fkp)}(G_{\emptyset_{\fka,\tau_i},\rmT_{\fka,\tau_i}})_{\bbF_{p^g}}\ar[ld]_{\pi_{\tau_i}}\ar[rd]^{\pi_{\tau_i}}& &\\
\scrS_K(G_{\emptyset_{\fka,\tau_i},\rmT_{\fka,\tau_i}})_{\bbF_{p^g}}& & \scrS_K(G_{\emptyset_{\fka,\tau_j},\rmT_{\fka,\tau_j}})_{\bbF_{p^g}} \ar[r]^{\eta_{\tau_j}}&\scrS_K(G_{\emptyset_{\fka,\tau_i},\rmT_{\fka,\tau_i}})_{\bbF_{p^g}}}\]
is identified with the base change to $\bbF_{p^g}$ of the Hecke correspondence for  $\scrS_K(G_{\emptyset_{\fka,\tau_i},\rmT_{\fka,\tau_i}})$.
\end{prop}

\section{Hecke orbit conjecture}\label{sec: Hecke orbit conjecture}
In this section we prove a version of the Hecke orbit conjecture for the ordinary locus of quaternionic Shimura varieties. The desired result follows from the corresponding statement for the auxiliary unitary Shimura varieties of PEL-type constructed in the previous section. The result for these Shimura varieties can then be deduced using the method of \cite{Chai}. \ignore{We note that it seems that many of the statements we use can be proved in the more general setting of Hodge-type Shimura varieties; however we do not choose to pursue this in the current paper. }

\subsection{Statement of Hecke orbit conjecture}\label{sec: 3 statement}We keep the notation introduced in \S2.
We let $\underline{\Sh}_{K}(G_{\rmS,\rmT})$ be the integral model  for the  quaternionic Shimura variety ${\Sh}_{K}(G_{\rmS,\rmT}).$ We will assume in this section that $\rmS\subsetneq \Sigma_{\infty}$ so that the compact open subgroup $K=K_pK^p$ is hyperspecial at $p$. Since the case of Hilbert modular varieties has been proved in \cite{Chai}, we also  assume ${\Sh}_{K}(G_{\rmS,\rmT})$, and hence $\underline{\Sh}_{K}(G_{\rmS,\rmT})$, is compact, or in other words that the quaternion algebra $B_{\rmS}$ is not totally split. By our assumption on $\rmS$, $\underline{\Sh}_{K}(G_{\rmS,\rmT})$ is not discrete, in other words it does not arise from a weak Shimura datum. For ease of notation we will write $\calX:=\scrS_{K}(G_{\rmS,\rmT})_{\overline{\bbF}_p}$ for the geometric special fiber of $\underline{\Sh}_{K}(G_{\rmS,\rmT})$.

We also write ${\Sh}_{K'}(G'_{\tilde{\rmS}})$ for the  unitary Shimura variety of PEL-type associated to ${\Sh}_{K}(G_{\rmS,\rmT})$ by the choice of an imaginary quadratic field $E/F$ and a subset $\tilde{\rmS}_\infty\subset\Sigma_{E,\infty}$ satisfying the conditions in \S2.2. We similarly write $\calX'=\scrS_{K'}(G'_{\tilde{\rmS}})_{\Fpbar}$ for the geometric special fiber of the integral model  $\underline{\Sh}_{K'}(G'_{\tilde{\rmS}})$.

Our formulation of the Hecke orbit conjecture for $\calX'$ differs slightly from the statement in \cite{Chai}; we need a slightly stronger statement in order to transfer the result to the quaternionic side. More precisely we will consider the orbit under Hecke correspondences coming from the derived group. We let $G'_{\mathrm{der}}$ denote the derived group of $G'_{\tilde{\rmS}}$. Since $G'_{\tilde{\rmS}}$ and $G_{\rmS}$ have the same derived group, $G'_{\rmS}$ is the reductive group corresponding to $(B_{\rmS}^{\times})^{\nu_{\rmS}=1}$; in particular  $G'_{\mathrm{der}}$ is simply connected. We write $T'_{\mathrm{ab}}$ for the quotient of $G'_{\tilde{\rmS}}$ by its derived group; it is isomorphic to the subtorus $\mathrm{Nm}_{E/F}^{-1}(\bbG_m)\subset T_{E}$. We write $T'_{\mathrm{ab}}(\bbR)^\dagger=\mathrm{Im}(Z'(\bbR)\rightarrow T'_{\mathrm{ab}}(\bbR))$, where $Z'$ is the center of $G'_{\tilde{\rmS}}$. We let $T'_{\mathrm{ab}}(\bbQ)^\dagger=T'_{\mathrm{ab}}(\bbR)\cap T'_{\mathrm{ab}}(\bbQ)$.

 We write $\nu':G_{\tilde{\rmS}}'\rightarrow T'_{\mathrm{ab}}$ for the quotient map. This induces a bijection \begin{equation}
 \label{eq: conn comp unitary}
\pi_0(\rmSh_{K'}(G'_{\tilde{\rmS}}))\cong T_{\mathrm{ab}}'(\bbQ)^{\dagger}\backslash  T'_{\mathrm{ab}}(\bbA_f)/\nu'(K') \end{equation} where the left hand side is the set of geometric connected components. 
Since $\underline{\Sh}_{K'}(G'_{\tilde{\rmS}})$ is smooth, there is  a bijection $$\pi_0(\calX')\xrightarrow{\sim}T_{\mathrm{ab}}'(\bbQ)^{\dagger}\backslash  T'_{\mathrm{ab}}(\bbA_f)/\nu'(K')$$ compatible with specialization from (\ref{eq: conn comp unitary}).
For an element $\fkc'$ on the right hand side, we write $\calX'^{\fkc'}$ for the corresponding connected component of $\calX'$.

	

 We consider the inverse limit scheme $$\underline{\Sh}_{K_p'}(G'_{\tilde{\rmS}}):=\lim_{\leftarrow K'^p}\underline{\Sh}_{K'_pK'^p}(G'_{\tilde{\rmS}})$$
which has an action by the group $G'_{\tilde{\rmS}}(\bbA_f^p)$. Let $x\in\underline{\Sh}_{K'}(G'_{\tilde{\rmS}})(\Fpbar)=\calX'(\Fpbar)$ and $\tilde{x}\in\underline{\Sh}_{K'_p}(G'_{\tilde{\rmS}})(\Fpbar)$   a lift of $x$. We write $\tilde{Y}'^p(\tilde{x})$ for the orbit of $\tilde{x}$ under the group $G'_{\mathrm{der}}(\bbA_f^p)$. Similarly, for $l$ a prime coprime to $p$, we write $\tilde{Y}'_l(\tilde{x})$ for the orbit of $\tilde{x}$ under the group $G_{\mathrm{der}}'(\bbQ_l)$.

\begin{defn}The reduced prime-to-$p$ Hecke orbit $Y'^p(x)$ of $x$ is defined to be the image of $\tilde{Y}'^p(\tilde{x})$ in $\calX'$. Similarly the reduced $l$-power Hecke orbit  $Y'_l(x)$ is defined to be the image of $\tilde{Y}'_l(\tilde{x})$ in $\calX'$. We write $Z'^p(x)$ and $Z'_l(x)$ for the closures of $Y'^p(x)$ and $Y'_l(x)$ respectively.
\end{defn}

It is clear from the definition that $Y'^p(x)$ and $Y'_l(x)$ are independent of the choice of $\tilde{x}$ and that $Y'_l(x)\subset Y'^p(x)$.  When we consider $Y'_l(x)$ we will also assume $K'$ factors as $K'_lK'^l$, where $K'_l\in G'_{\tilde{\rmS}}(\bbQ_l)$ and $K'^l\subset G'_{\tilde{\rmS}}(\bbA_f^l)$ are compact open subgroups.

By the main result of \cite{Ham2}, there is a stratification of $\calX'$ parametrized by the set $B(G'_{\tilde{\rmS},\bbQ_p},\{\mu\})$ where $\{\mu\}$ is the  Hodge cocharacter associated to the Shimura datum. This is a certain subset of the set of $\sigma$-conjugacy classes of $G'_{\tilde{\rmS}}(L)$, where $L$ is the completion of the maximal unramified extension of $\bbQ_p$; we refer to \S\ref{sec: Appendix B(G)} for the definition. The set $B(G'_{\tilde{\rmS},\bbQ_p},\{\mu\})$ is equipped with  a partial order. It  contains a unique maximal element $[b]^{\mathrm{ord}}$ and a unique minimal element $[b]^{\mathrm{ss}}$. We write $\calX'^{\mathrm{ord}}$ and $\calX'^{\mathrm{ss}}$  for the corresponding strata. 
\begin{rem}
In the literature, $\calX'^{\mathrm{ord}}$ is usually known as the $\mu$-ordinary locus. The stratum $\calX'^{\mathrm{ss}}$ coincides with the locus where the universal abelian variety is supersingular.  

\end{rem}
The main result of this section is the following. 

\begin{thm}\label{thm:Hecke orbit unitary}
Let $\fkc'\in T_{\mathrm{ab}}'(\bbQ)^{\dagger}\backslash  T'_{\mathrm{ab}}(\bbA_f)/\nu'(K')$. Then for any  $x\in \calX'^{\fkc'}\cap \calX'^{\mathrm{ord}}(\Fpbar)$ we have  $$Z'^p(x)=\calX'^{\fkc'}.$$
\end{thm}
\begin{rem}\label{rem: Hecke orbit unitary}(1) The inclusion $Z'^p(x)\subset\calX'^{\fkc'}$ is clear since $G'_{\mathrm{der}}=\ker \nu'$.
	
	(2) Since $G'_{\tilde{\rmS}}(\bbA_f^p)$ acts transitively on the set of connected components of $\underline{\Sh}_{K'_p}(G'_{\tilde{\rmS}})$ (see for example \cite[Lemma 2.2.5]{Ki2}), this Theorem implies the  prime-to-$p$ Hecke orbit of any $x\in \calX'^{\mathrm{ord}}(\Fpbar)$ is Zariski dense in $\calX'$. Here the  prime-to-$p$ Hecke orbit is defined by replacing the $G'_{\mathrm{der}}(\bbA^p_f)$-orbit with the $G'_{\tilde{\rmS}}(\bbA_f^p)$-orbit in the definition of $\tilde{Y}'^p(\tilde{x})$.
\end{rem}

The rest of this section is devoted to the proof of this Theorem. We briefly outline the strategy following \cite{Chai}. The first step is to analyze the formal completion $\widehat{Z}'^p(x)$ of $Z'^p(x)$ at a smooth ordinary point. The stability of $Z'^p(x)$ under Hecke correspondences constrains the possibilities for this formal completion. Indeed, using Moonen's generalization of the Serre--Tate coordinates \cite{Mo} in this context, one can show that the formal completion of $\calX'$ at an ordinary point has a formal group structure; then $\widehat{Z}'^p(x)$ is actually a formal subgroup. The second step is to show that $Z'^p(x)$ contains a supersingular point; this uses the quasi-affineness of the Ekedahl--Oort stratification on $\calX'$ proved in \cite{VW}. Finally we  analyze the formal completion of $Z'^p(x)$ at a supersingular point $s\in \calX'^{\mathrm{ss}}\cap Z'^p(x)$; we use the fact that $s$ is the specialization of an ordinary point $x'$ and that $\calA_s$ has a large endomorphism group to show that the formal completion of $Z'^p(x)$ at $x'$ is the whole formal neighborhood in $\calX'$. This implies the result.

A key ingredient of the  proof of  the Hecke orbit conjecture for Hilbert modular varieties in \cite{Chai} was an explicit description of the structure of isogeny classes associated to an abelian variety with extra structure. This required an understanding of the endomorphism  algebra of these abelian varieties, see \cite[Lemma 6]{Chai}. We will instead use a group theoretic approach using results of \cite{Ki3} which avoids some of  the case-by-case analysis of loc. cit.

Let $k$ be an algebraically closed field of characteristic $p$. Let $x\in \calX'(k)$  and $(\calA_x,\iota,\lambda,\epsilon_{K'^p})$ the associated quadruple.  We write $I_x$ for the reductive group over $\bbQ$ such that $I_x(\bbQ):=\mathrm{Aut}^0_{(\calO_{D_{\rmS}},\lambda)}(\calA_x)$ the group of quasi-isogenies of $\calA_x$ respecting the $\calO_{D_{\rmS}}$-structure and preserving the polarization $\lambda$ up to a scalar in $\bbQ^\times$. We would like to apply the results of \cite[\S2]{Ki3} to our situation. Suppose that $x\in \calX'(\Fpbar)$. We write $I_x^{\mathrm{Kis}}$ for the group denoted $I$ in \cite[\S2.1]{Ki3}. Then it is easy to see that $I_x^{\mathrm{Kis}}\cong I_x$; indeed in \cite{Ki3} we may take the tensors $s_{\alpha}\in W_{\bbZ_{(p)}}^\otimes$ to be the classes corresponding to the endomorphisms and polarization. Then the condition defining $I_x^{\mathrm{Kis}}$ in \cite[\S2.1]{Ki3} precisely corresponds to the condition defining $I_x$. 

By \cite[Corollary 2.3.5]{Ki3}, there is a subgroup $I_0\subset G'_{\tilde{\rmS}}$ equipped with an inner twisting $I_0\otimes_{\bbQ}\overline{\bbQ}\cong I_x\otimes_{\bbQ}\overline{\bbQ}$. The map $\nu':G'_{\tilde{\rmS}}\rightarrow T'_{\mathrm{ab}}$ determines a map $\nu'_x:I_x\rightarrow T'_{\mathrm{ab}}$ and we let $I_x^1\subset I_x$ denote the kernel of $\nu'_x$. For all primes $l\neq p$, there is a natural map $I_x(\bbQ)\rightarrow G'_{\tilde{\rmS}}(\bbQ_l)$ given by the action on the $l$-adic Tate-module.  The following proposition can be deduced in a standard way as in \cite[p448]{Chai}.

\begin{prop}\label{prop: isog classes}
	Suppose $x\in \calX'(\Fpbar)$. There are bijections $$Y'^p(x)\cong I^1_x(\bbQ)\cap \rmK'_p\backslash  G'_{\mathrm{der}}(\bbA_f^p)K'^p/K'^p$$
	$$Y'_l(x)\cong (I^1_x(\bbQ)\cap K'^l)\backslash G'_{\mathrm{der}}(\bbQ_l)K'_l/K'_l.$$
\end{prop}
\begin{proof}
Let $\tilde{x}\in\underline{\mathrm{Sh}}_{K'_p}(G'_{\tilde{\rmS}})(\Fpbar)$ be a lift of $x$. Then by the argument of \cite[p448]{Chai}, the stabilizer of $\tilde{x}$ in $G'_{\tilde{\rmS}}(\bbA_f^p)$  is identified with $I_x(\bbQ)\cap \rmK'_p$ and hence the  $G'_{\tilde{\rmS}}(\bbA_f^p)$-orbit of $\tilde{x}$ can be identified with $I_x(\bbQ)\cap \rmK'_p\backslash  G'_{\tilde{\rmS}}(\bbA_f^p)$, cf. also \cite[Proposition 2.1.3]{Ki3}. Since $G'_{\der}(\bbA_f^p)\cap I_x(\bbQ)=I_x^1(\bbQ)$, the $G'_{\mathrm{der}}(\bbA_f^p)$-orbit of $\tilde{x}$ is identified with $I^1_x(\bbQ)\cap \rmK'_p\backslash  G'_{\mathrm{der}}(\bbA_f^p)$ and the description of $Y'^p(x)$ follows by taking the image in $\calX'$. The description of $Y'_l(x)$ is proved in a similar way.
\end{proof}

\subsection{Serre--Tate theory following \cite{Mo} and analysis at ordinary points}	In this subsection we study the formal neighbourhood of a point in $\calX'^{\mathrm{ord}}$. We begin with some preliminaries concerning $p$-divisible groups with an action of the ring of integers $\calO$ of a finite unramified extension of $\bbQ_p$. Recall we have the integral PEL-datum $(\calO_{D_{\rmS}}\subset D_{\rmS},*,\Lambda\subset W,\psi)$; we write $\calD$ for the base change of this datum to $\bbZ_p$.

\begin{defn}\label{def:D-structure}
	Let $S$ be a $\calO_{\mathbf{E}_{\tilde{\rmS},\tilde{v}}}$-scheme.	A $p$-divisible group with $\calD$-structure is a triple $({\mathscr{G}},\lambda,\iota)$ where:
	
	$\bullet$ ${\mathscr{G}}$ is a $p$-divisible group over $S$ of height $\dim W$.
	
	$\bullet$ $\iota:\calO_{D_{\rmS},p}\rightarrow \text{End}({\mathscr{G}})$ is a homomorphism satisfying \begin{equation}\label{eq:char poly}\det(T -\iota(a)| \mathrm{Lie}(\scrG/S)) = \prod_{\tilde{\tau}\in\Sigma_{E,\infty}}(T-\tilde{\tau}(a))^{2s_{\tilde{\tau}}}.\end{equation}
	
	$\bullet$ $\lambda:{\mathscr{G}}\rightarrow {\mathscr{G}}^\vee$ is a polarization  such that $$\iota(a)=\lambda^{-1}\circ\iota(a^*)^\vee\circ\lambda. $$
\end{defn}
It is easy to see that for any $x\in \underline{\Sh}_K'(G'_{\tilde{\rmS}})(S)$, the associated $p$-divisible group $\calA_x[p^\infty]$ is a $p$-divisible group with $\calD$-structure.

Suppose $p$ factors as $\fkp_1\dotsc\fkp_r$ in $F$; we have a decomposition  $F\otimes\bbQ_p\cong F_{\fkp_1}\times\cdots \times F_{\fkp_r}$. This induces a decomposition $$D_{\rmS,p}\cong D_{\rmS,\fkp_1}\times \cdots\times D_{\rmS,\fkp_r}.$$
The datum $\calD$ decomposes as a product of data $\calD_i:=(\calO_{D_{\rmS,\fkp_i}}\subset D_{\rmS,\fkp_i},*,\Lambda_i\subset W_i,\psi_i,)$, where $W_i$ is the subspace of $W\otimes \bbQ_p$ on which $D_{\rmS,p}$ acts via $D_{\rmS,\fkp_i}$ and $\psi_i$ and $*$  are the restrictions of $\psi\otimes\bbQ_p$ and $*$ to $W_i$. Similarly to Definition \ref{def:D-structure}, we may define the notion of $p$-divisible group $\calD_i$-structure.
By Corollary \cite[Corollary 4.5 (1)]{Ham2}, there is an equivalence of categories 
\begin{equation}\label{eq:pdiv red 1}\{\text{$p$-divisible groups with $\calD$-structure}\}\cong \prod_i\{\text{$p$-divisible groups with $\calD_i$-structure}\}\end{equation}
which takes isogenies to isogenies. Moreover this equivalence gives a decomposition $$\scrG_x\cong\prod_{i=1}^r\scrG_{x,i}$$ for $\scrG_x$ a $p$-divisible group with $\calD$-structure. The $\scrG_{x,i}$ are equipped with polarizations $\lambda_i$ and $\calO_{D_{\rmS},\fkp_i}$-actions and the isomorphism  identifies $\lambda$ with the product polarization and $\iota$ with the product of the actions of $\calO_{D_{\rmS},\fkp_i}$.

For any $x\in \calX'(\Fpbar)$, we write $J_x$ for the group of automorphism of $\scrG_x$ in the isogeny category preserving the $\calO_{D_{\rmS,p}}$-action and the polarization up to a $\bbQ_p^\times$ scalar. Then $J_x$ is a reductive group over $\bbQ_p$ which is an inner form of a Levi subgroup of $G'_{\tilde{\rmS},\bbQ_p}$ and there is a natural embedding $I_x\otimes\bbQ_p\hookrightarrow J_x$. Similarly to the definition of $I_x^1$, the map $G'_{\tilde{\rmS}}\rightarrow T_{\mathrm{ab}}$ induces a natural map $J_x\rightarrow T_{\mathrm{ab},\bbQ_p}$ and we write $J_x^1$ for its kernel. The $J_x^1$ breaks up into a product $J_{x,i}^1$ corresponding to the decomposition $\scrG_x\cong \prod_{i=1}^r\scrG_{x,i}$.

Now suppose $k$ is an algebraically closed field of characteristic $p$ and  $x\in \calX'^{\mathrm{ord}}(k)$. We write $\widehat{U}_x$ for the formal neighborhood of $\calX'$ at the point $x$. By the Serre--Tate theorem,
$\widehat{U}_x$ is identified with the characteristic $p$ deformation space of $\calA_x[p^\infty]$; see \cite[Proposition 2.9]{Ham2} for example. In other words, let $\mathrm{Def}_{\scrG_x}$ be the functor on local Artinian rings $R/k$ with residue field $k$ satisfying:
$$\mathrm{Def}_{\scrG_{x},\lambda}(R)=\left\{(\tilde{\scrG},\tilde{\lambda},\iota)/R\text{ a $p$-divisible group with $\calD$-structure,\  $\theta:\tilde{\scrG}\otimes_R{k}\xrightarrow\sim \scrG_x$}\right\}/\sim.$$ Here $\theta$ is an isomorphism of $p$-divisible groups with $\calD$ structure (i.e. preserves actions and polarizations) and $\sim$ is the equivalence relation identifying isomorphic $p$-divisible groups with $\calD$-structure. Then $\mathrm{Def}_{\scrG_x,\lambda}$ is representable by a formal scheme $\mathrm{Spf}\ R$, and $\mathrm{Spf}\ R\cong \widehat{U}_x$. The equivalence of categories (\ref{eq:pdiv red 1}) implies that there is an isomorphism $$\mathrm{Def}_{\scrG_x,\lambda}\cong \prod_{i=1}^r\mathrm{Def}_{\scrG_{x,i},\lambda_i}$$
where $\mathrm{Def}_{\scrG_{x,i},\lambda_i}$  is the deformation space of the $p$-divisible group $\scrG_{x,i}$ with $\calD_i$-structure
\begin{equation}\label{defn: def space i}
\mathrm{Def}_{\scrG_{x,i},\lambda_i}(R)=\left\{(\tilde{\scrG}_i,\tilde{\lambda}_i,\iota_i)/R \text{ a $p$-divisible group with $\calD_i$-structure,\  $\theta_i:\tilde{\scrG}_{i}\otimes_Rk\xrightarrow\sim \scrG_{x,i}$}\right\}/\sim.
\end{equation}
\begin{thm}[{\cite[Theorem p226, Example 3.3.2]{Mo}}]For $i=1,\dotsc,r$, $\mathrm{Def}_{\scrG_{x,i},\lambda_i}$ is a formal $p$-divisible group. Hence $\mathrm{Def}_{\scrG_x,\lambda}$ is a formal $p$-divisible group.

\end{thm}
The results in \cite{Mo} in fact give an explicit description of the group structure of these deformation spaces which we now recall. To do this we introduce some notations. Let $\calO$ be an unramified extension of $\bbZ_p$ of degree $g$; see \S\ref{sec: Appendix O-structure}. Recall a $p$-divisible group with $\calO$-structure over a scheme $S$ is a $p$-divisible group $\scrG/S$ together with homomorphism of $\bbZ_p$-algebras $\iota:\calO\rightarrow \mathrm{End}_S(\scrG)$. The $\calO$-height $\mathrm{ht}_\calO(\scrG)$ of $\scrG$ satisfies $$g\mathrm{ht}_\calO(\scrG)=\mathrm{ht}(\scrG).$$ We identify the set of embeddings $\calO\rightarrow W(k)$ with $\scrI:=\{1,\dotsc,g\}$ so that $\sigma(m)=m+1$. Let $d$ be an integer and $\fkf:\scrI\rightarrow \{0,\dotsc,d\}$. For $m\in\scrI$ we let $M_m$ be a $W(k)$-module of rank $d$ with basis $(e_{m,n})_{n=1,\dotsc,d}$. We equip $M_m$ with a Frobenius $\varphi$ given by
\begin{equation}
\varphi(e_{m,n})=\begin{cases}
e_{m+1,n} & \text{if $n\leq d-\fkf(m)$}\\
pe_{m+1,n} &\text{ if $n> d-\fkf(m)$}
\end{cases}.
\end{equation}
This gives $M$ the structure of a Dieudonn\'e module and hence corresponds to a $p$-divisible group $\scrG(d,\fkf)$. The action of $\calO$ on $M$ given by letting $\calO$ act on $M_m$ via the embedding $m$ induces an action of $\calO$ on $\scrG(d,\fkf)$. Then $\scrG(d,\fkf)$ is a $p$-divisible group with $\calO$-structure of $\calO$-height $d$.

Let $i\in\{1,\dotsc,r\}$ and $\calO':=\calO_{E_{\fkp_i}}$; we write $g:=[F_{\fkp_i}:\bbQ_p]$. By Morita equivalence, there exists a polarized $p$-divisible group $(\scrG^\circ,\lambda^\circ)$ with $\calO'$-structure of $\calO'$-height 2 such that $(\scrG^{\circ})^2=\scrG_{x,i}$. The polarization $\lambda^\circ$ induces an involution $*$ on $\calO'$. Then $(\calO',*)$ is one of the following two forms, see \cite[3.1.2]{Mo}:

Case (AL): $\fkp_i$ is splits as $\fkq_i\overline{\fkq}_i$ in $E$ and $\calO'=W(\bbF_{p^g})\times W(\bbF_{p^g})$, where  the involution $*$ is given by $(a,b)\mapsto (b,a)$.

Case (AU): $\fkp$ is inert in $E$ and $\calO'=W(\bbF_{p^{2g}})$. The  involution $*$ is given by the $p^g$-Frobenius.

In the case (AL), we write $\calO=\calO_{E_{\fkq_i}}\cong W(\bbF_{p^g})$ and there is a $p$-divisible group $\scrG'$ with $\calO$-structure of $\calO$-height 2 such that $\scrG^\circ\cong \scrG'\times \scrG'^{\vee}$. We identify the set of embeddings $\calO\rightarrow W(k)$ with $\scrI:=\{1,\dotsc,g\}$ as above, and we also identify this with $\Sigma_{E,\infty/\fkq_i}$. Then $\scrG'\cong\scrG(2,\fks)$ where $\fks:\scrI\rightarrow \{0,1,2\}$ is defined by $\fks(m)=s_m$; here $s_m$ is the number defined in (\ref{eq:s tau}).

We define a function $\fkd:\scrI\rightarrow \{0,1\}$ by setting $$\fkd(m)=\begin{cases}
1 & \text{if $s_m=1$}\\
0 & \text{otherwise}
\end{cases}.$$

Then by \cite[Theorem p4]{Mo} there is an isomorphism $$\mathrm{Def}_{\scrG_{x,i},\lambda_i}\cong \mathrm{Def}_{\scrG'}\cong\scrG(1,\fkd)^{\mathrm{for}},$$ where we write $\scrG(1,\fkd)^{\mathrm{for}}$ for the formal group associated to $\scrG(1,\fkd)$ and the space $\mathrm{Def}_{\scrG'}$ is the deformation space of the $p$-divisible group with $\calO$-structure (without polarization). The first isomorphism here is via \cite[Corollary 4.5 (2) I]{Ham2}. Then $\mathrm{Def}_{\scrG_{x,i}}$ is a $p$-divisible group with $\calO$-structure of $\calO$-height 1.

In the case (AU), we write $\calO:=\calO_{F_{\fkp_i}}\cong W(\bbF_{p^{g}})$; recall $\calO':=\calO_{E_{\fkp_i}}\cong W(\bbF_{p^{2g}})$. Then $\scrG^\circ\cong  \scrG(2,\fks)$ where $\fks:\{1,\dotsc,2g\}\rightarrow \{0,1,2\}$ is defined by $\fks(m)=s_m$ as above. Here we have identified $\{1,\dotsc,2g\}$ with $\Sigma_{E,\infty/\fkp_i}$. There exists a $p$-divisible group with $\calO'$-structure $\scrG'$ of $\calO'$-height 1 and an isomorphism $\scrG^\circ\cong \scrG'\times\scrG'^{\vee}$. The polarization $\lambda^\circ$ on $\scrG^\circ$ switches the two factors.

We define a $p$-divisible group $\tilde{\scrG}:=\scrG(1,\fkd)$ with $\calO'$-structure where $\fkd$ is defined by  $$\fkd(m)=\begin{cases}
1 & \text{if $s_m=1$}\\
0 & \text{otherwise}
\end{cases}$$
as above.  In this case, the polarization $\lambda^\circ$ induces an endomorphism of  $\tilde{\scrG}$; then there is an isomorphism $$\mathrm{Def}_{\scrG_{x,i},\lambda_i}\cong\mathrm{Def}_{\scrG^\circ,\lambda^\circ}\cong(\tilde{\scrG}^{\lambda^\circ})^{\mathrm{for}}.$$
In this case the deformation space $\mathrm{Def}_{\scrG_{x,i},\lambda_i}$ is no longer stable under the action of $\calO'$. However it is stable under $\calO$ and indeed it is a $p$-divisible group with $\calO$-structure of $\calO$-height 1. 

Using the definition of slopes for a $p$-divisible group with $\calO$-structure as in \cite[\S1.2.5]{Mo}, it is easily checked  that $\scrG'$ in case (AL) and $\scrG^\circ$ in case (AU) has one slope  if  $\Sigma_{\infty/\fkp_i}-\rmS_{\infty/\fkp_i}=\emptyset$ and two slopes otherwise. More we have that  $ \mathrm{Def}_{\scrG_{x,i},\lambda_i}$ has dimension $|\Sigma_{\infty/\fkp_i}-\rmS_{\infty/\fkp_i}|$ in both cases (AL) and (AU); in particular $\mathrm{Def}_{\scrG_{x,i},\lambda_i}$ is trivial if $\Sigma_{\infty/\fkp_i}-\rmS_{\infty/\fkp_i}=\emptyset$.

We also need a description of the endomorphism algebra of $\scrG_{x,i}$ and its action on the deformation space.  For $i\in\{1,\dotsc,r\}$, we write $\rmJ_{x,i}:=\mathrm{Aut}_{(\calO',\lambda^\circ_i)}(\scrG^\circ_{x,i})$ for the group of automorphisms of $\scrG_{x,i}^\circ$ respecting the $\calO'$ action and polarization. The group $\rmJ_{x,i}$ acts on $\mathrm{Def}_{\scrG_{x,i}}$ via modification of the isomorphism $\theta_i$ in (\ref{defn: def space i}).

In case (AL), $\rmJ_{x,i}$ can be identified with the group of automorphisms of $\scrG'$ respecting the $\calO_{F_{\fkp_i}}$-structure.  If $\scrG'$ has one slope, in other words if $s_m\neq 1$ for any $m\in\Sigma_{E,\infty/\fkq_i}$, then $\rmJ_{x,i}\cong M_2(\calO_{\fkp_i})$.  In this case the deformation space is trivial and $\rmJ_{x,i}$ acts trivially on $\mathrm{Def}_{\scrG_{x,i},\lambda_i}$.
If $\scrG'$ has more than one slope, which occurs if $\rmS_{\infty/\fkp_i}\neq\Sigma_{\infty/\fkp_i}$, then  $\rmJ_{x,i}\cong (\calO_{F_{\fkp_i}}^\times)^2$. In this case $\rmJ_{x,i}$ acts on  $\mathrm{Def}_{\scrG_{x,i},\lambda_i}$ via $(a,b)\mapsto \iota(ab^{-1})$; here $\iota$ gives the $\calO_{F_{\fkp_i}}$-structure on $\mathrm{Def}_{\scrG_{x,i}}$. 

In case (AU), if $\scrG^\circ$ has one slope, then $\rmJ_{x,i}$ is identified with the  subgroup in $GL_2(\calO_{E_{\fkp_i}})$ preserving the Hermitian form  given by the matrix $\left(\begin{matrix}0 & 1 \\ -1 & 0\end{matrix}\right)$. In this case the deformation space is trivial and $\rmJ_{x,i}$ acts trivially on $\mathrm{Def}_{\scrG_{x,i}}$. If $\scrG^\circ$ has two slopes, $\rmJ_{x,i}$ is identified with $\calO_{E_{\fkp_i}}^\times$. Then $\rmJ_{x,i}$ acts on  $\mathrm{Def}_{\scrG_{x,i}}$ via $a\mapsto \iota(\mathrm{Nm}_{E_{\fkp_i}/F_{\fkp_i}}(a))$.

It follows from these descriptions that $\mathrm{Def}_{\scrG_x,\lambda}$ is equipped with an action of $(\calO_{F}\otimes\bbZ_p)^\times$.

Now assume $x'\in \calX'^{\mathrm{ord}}(\Fpbar)$. We let $x\in Z'^p(x')(\Fpbar)$ be an ordinary point which lies in the smooth locus of $Z'^p(x')$ and we write $\widehat{Z}_x'^p(x')\subset \mathrm{Def}_{\scrG_x}$ for the formal completion of $Z'^p(x')$ at $x$. The analysis of the actions of $\rmJ_{x,i}$ on $\mathrm{Def}_{\scrG_{x,i}}$  leads to the following proposition.
\begin{prop}\label{prop:formal nbd stable}
 $\widehat{Z}_x'^p(x')\subset \mathrm{Def}_{\scrG_x}$ is stable under the action of an open subgroup of $(\calO_F\otimes\bbZ_p)^\times$.
\end{prop}

\begin{proof} 
	
Recall we have the groups $J_x$ and $J_x^1$ defined above.	Let $T\subset I_x$ denote a maximal torus and $T^1$ its intersection with $I_x^1$. By \cite[Corollary 2.1.7]{Ki3}, $I_x$ and  $J_x$ have the same rank, and hence $I_x^1$ and $J_x^1$ have the same rank. It follows that $T^1_{\bbQ_p}$ is a maximal torus of $J^1_x$.
	
	Now $T^1(\bbQ_p)$ breaks up into a product $\prod_{i=1}^r T_i(\bbQ_p)$ where $T_i(\bbQ_p)\subset J^1_{x,i}$. Let $i\in\{1,\dotsc,r\}$ be an element such that $\scrG_{x,i}'$ in case (AL) (resp. $\scrG_{x,i}^\circ$ in case (AU)) has more than one slope.  It can be checked in the two cases there is the following description of $T_i^1(\bbQ_p)$:
	
	Case (AL): $T_i^1(\bbQ_p)\subset \mathrm{Aut}_{\calO}^0(\scrG_{x,i}')\cong (F_{\fkp_i}^\times)^2$ is the  subgroup of elements of the form $(a,b)$ with $ab=1$. Here  $\scrG_{x,i}'$ is as in the discussion above and $\mathrm{Aut}_{\calO}^0(\scrG_{x,i}')$ is the group of automorphisms in the isogeny category preserving the $\calO$-action.
	
	Case (AU): $T_i^1(\bbQ_p)\subset \mathrm{Aut}_{\calO',\lambda^\circ}^0(\scrG_{x,i}^\circ)\cong E_{\fkp_i}^\times$ is the subgroup consisting of elements $a\in E_{\fkp_i}^\times$ with $a\overline{a}^{-1}=1$. Here $\mathrm{Aut}_{\calO',\lambda^\circ}^0(\scrG_{x,i}^\circ)$ denotes the group of automorphisms in the isogeny category preserving the $\calO'$-actions and the polarization.
	
By weak approximation for $T^1$, see \cite[Corollary 3.5]{San}, $T^1(\bbQ)\cap \rmJ_x$ is dense in $T^1(\bbQ_p)\cap \rmJ_x$. Since the action of $T^1(\bbQ)\cap \rmJ_x$ on $\mathrm{Def}_{\scrG_x,\lambda}$ preserves $\widehat{Z}'^p(x')$ so does the action of $T^1(\bbQ_p)\cap \rmJ_x$. By the description of $T^1(\bbQ_p)$ in the previous paragraph, the action of $T^1(\bbQ_p)\cap\rmJ_x^1$ on the deformation space can be described in the following way. We assume $i\in\{1,\dotsc,r\}$ is such that $\scrG_{x,i}^\circ$ has two slopes. Then in either case (AU) or (AL),
	$T_i^1(\bbQ_p)\cap \rmJ_{x,i}\cong\calO_{F_{\fkp_i}}^\times$ acts on $\mathrm{Def}_{\scrG_{x,i},\lambda_i}$ via $a\mapsto \iota(a^2)$. Since the image of $a\mapsto a^2$ in $\calO_{F_{\fkp_i}}^\times$ is open, this proves the proposition.

\end{proof}

We may relabel the primes $\fkp_i$ so that $\scrG_{x,i}'$ in case (AL) (resp. $\scrG_{x,i}^\circ$ in case (AU)) has two slopes for $i=1,\dotsc, a$ and one slope for $i=a+1,\dotsc,r$. Then $\mathrm{Def}_{\scrG_{x,i},\lambda_i}$ is non-trivial for $i=1,\dotsc,a$ and trivial for $i=a+1,\dotsc,r$, and we have an isomorphism:
$$\mathrm{Def}_{\scrG_x,\lambda}\cong\prod_{i=1}^a\mathrm{Def}_{\scrG_{x,i},\lambda_i}.$$
\begin{cor}\label{cor: rigidity}
	There exists a subset $\omega\subset \{1,\dotsc,a\}$ such that $$\widehat{Z}_x'^p(x')=\prod_{i\in\omega}\mathrm{Def}_{\scrG_{x,i},\lambda_i}.$$
\end{cor}
\begin{proof}
By \cite [Proposition 4.2]{Chai2}, $\widehat{Z}_x'^p(x')$ is a formal $p$-divisible subgroup of $\mathrm{Def}_{\scrG_{x,i},\lambda_i}$. Indeed it can be checked  from the explicit description of the action of $T^1(\bbQ_p)\cap\rmJ_x$ on $\mathrm{Def}_{\scrG_x}$ that the conditions stated in  \cite[Proposition 4.2]{Chai2} are satisfied. By Proposition \ref{prop:formal nbd stable}, $\widehat{Z}_x'^p(x')$ is a formal $p$-divisible subgroup stable under the action of an open subgroup of $(\calO_F\otimes\bbZ_p)^\times$. The only such subgroups are the ones described.
\end{proof}

This Proposition can be globalized to the following. Recall the tangent sheaf $T_{\calX'}$ is equipped with an action of $\calO_F\otimes \bbZ_p$  by  Proposition \ref{prop: KS}. We write $T_{\calX'}(\fkp_i)$ for the subspace of  $T_{\calX'}$ on which $\calO_F\otimes \bbZ_p$ acts via the projection to $\calO_{F_{\fkp_i}}$. The next corollary follows immediately from \ref{cor: rigidity} and faithfully flat descent.
\begin{cor}\label{cor:rigidity globalized}Let $W$ be an irreducible component of the smooth locus $Z'^p(x')^{\mathrm{sm}}$ of $Z'^p(x')$. Then there exists a subset $\omega\subset\{1,\dotsc,a\}$ such that the tangent sheaf  $T_W$ of $W$ is given by $$\bigoplus_{i\in\omega} T_{\calX'}(\fkp_i)\otimes_{\scrO_{\calX'}} \scrO_W\subset T_{\calX'}\otimes_{\scrO_{\calX'}}\scrO_W.$$
\end{cor} 

\subsection{Proof of the main Theorem}
Recall for $x\in \calX'(\Fpbar)$, $Y^p(x)$ and $Y_l(x)$ are the reduced prime-to-$p$ and $l$-power Hecke orbits respectively, and $Z'^p(x)$ and $Z'_l(x)$ are their respective closures.
\begin{prop}\label{prop:orbit contains supersingular}
	$Z'_l(x)\cap \calX'^{\mathrm{ss}}\neq\emptyset$.
\end{prop}
\begin{proof} Since $\calX'$ is proper over $\Fpbar$, $Z'_l(x)$ is proper. By the main Theorem of \cite{VW}, there is a stratification on $\calX'$ whose strata are quasi-affine and stable under prime-to-$p$ Hecke correspondences. The strata  $\calX'_w$ are parametrized by a certain subset ${^JW}$ of the Weyl group $W$ of $G'_{\tilde{\rmS}}$. The closure relations are given by a partial ordering on ${^JW}$.
	
	We let $w\in W^J$ be a minimal element for this partial order such that $\calX'_w\cap Z'_l(x)\neq\emptyset.$ Then  $\calX'_w\cap Z'_l(x)$ is closed in $\calX'$ by the minimality of $w$ and the properness of $Z'_l(x)$, hence $\calX'_w\cap Z'_l(x)$ is proper. Since $\calX'_w\cap Z'_l(x)$ is also quasi-affine, it follows that $\calX'_w\cap Z'_l(x)$ is 0-dimensional.
	By Lemma \ref{lem:supersingular 0-dim} below, $\calX'_w\cap Z'_l(x)\subset \calX'^{\mathrm{ss}}\neq \emptyset$.
	
\end{proof}

\begin{lem}\label{lem:supersingular 0-dim}Let $Y\subset \calX'$ be a $0$-dimensional subscheme of $\calX'$ which is stable under reduced $l$-power Hecke correspondences. Then $Y\subset \calX'^{\mathrm{ss}}$.
	
	Conversely if $x\in \calX'^{\mathrm{ss}}(\Fpbar)$, then $Y^p(x)$ is finite and hence $Z'^p(x)$ is $0$-dimensional.
\end{lem}
\begin{proof}
First assume $x\in \calX'^{\mathrm{ss}}(\Fpbar)$. Then $I_x\otimes\bbQ_l\cong G'_{\tilde{\rmS}}\otimes\bbQ_l$ for all $l\neq p$ by \cite[7.2.14]{XZ} and hence $I_x^1\otimes\bbQ_l\cong G'_{\mathrm{der}}\otimes\bbQ_l$. Therefore $Y^p(x)\cong I^1_x(\bbQ)\cap \rmK'_p\backslash  G'_{\mathrm{der}}(\bbA_f^p)K'^p/K'^p$ is finite by the finiteness of class groups.

Now we assume $Y$ is finite and $x\in Y$; then $Y_l(x)$ is finite. It follows that $I_x^1(\bbQ_l)\backslash G'_{\mathrm{der}}(\bbQ_l)$ is compact. The same argument as in \cite[Corollary 2.1.7]{Ki3} shows that $$I_x^1\otimes \bbQ_l\cong G'_{\mathrm{der}}\otimes\bbQ_l$$ and hence $$I_x\otimes \bbQ_l\cong G'_{\tilde{\rmS}}\otimes\bbQ_l$$ since $I_x$ and $G'_{\tilde{\rmS}}$ have the same ranks by \cite[Corollary 2.1.7]{Ki3}. But $I_x\otimes \bbQ_p$ is a subgroup of $J_x$ which is an inner form of a Levi subgroup of $G'_{\tilde{\rmS}}\otimes\bbQ_p$.  Since $I_x$ and $G'_{\tilde{\rmS}}$ have the same dimension, $J_x$ is an inner form of $G'_{\tilde{\rmS}}$. Therefore the Newton cocharacter defined as in \S\ref{sec: Appendix B(G)} is central and hence $x$ is a supersingular point.
\end{proof}

We will need the following two results on the structure of $I^1_x$ for supersingular points. 

\begin{prop}\label{prop:Tate supersingular} Suppose $x\in\calX'^{\mathrm{ss}}(\Fpbar)$.
	
	(1) The natural map $I_x^1\otimes_{\bbQ}\bbQ_p\rightarrow J_x^1$ is an isomorphism. 
	
	(2) $I_x^1$ is an inner form of $G'_{\mathrm{der}}$. In particular $I_x^1$ is simply connected.
\end{prop}
\begin{proof}
 By \cite[Lemma 7.2.14]{XZ} $I_x$ is an inner form of $G$ and $I_x\otimes \bbQ_p\rightarrow J_x$ is an isomorphism. Therefore $I_x^1$ is an inner form of $G'_{\mathrm{der}}$ and  $I^1_x\otimes \bbQ_p\rightarrow J^1_x$ is an isomorphism.

\end{proof}

We can now complete the proof of the main Theorem.

\begin{proof}[Proof of Theorem \ref{thm:Hecke orbit unitary}]We proceed as in \cite[Proposition 7]{Chai}; we will  assume $Z'^p(x)$ is not equal to $\calX'^{\fkc'}$ and deduce a contradication. By Proposition \ref{prop:orbit contains supersingular}, $Z'^p(x)$ contains a supersingular point $s\in\calX'^{\mathrm{ss}}(\Fpbar)$.  We write $\widehat{Z}_s'^p(x)$ for the completion of $Z'^p(x)$ at the point $s$.
	
	Let $\calA_s$ be the abelian variety with $\calO_{D_{\rmS}}$-multiplication associated to $s$. The corresponding $p$-divisible group $(\scrG_s,\lambda,\iota)$  with $\calD$-structure breaks up into a product $\scrG_{s,1}\times\cdots\times\scrG_{s,r}$ of $p$-divisible groups with $\calD_i$-structure
 and there is a decomposition of the deformation space $$\mathrm{Def}_{\scrG_{s},\lambda}\cong \mathrm{Def}_{\scrG_{s,1},\lambda_1}\times\cdots\times \mathrm{Def}_{\scrG_{s,r},\lambda_r}.$$
We write $\rmJ_{s}$ for the group  of  automorphisms of $\scrG_{s}$ respecting the $\calO_{D_{\rmS}}$ action and polarization, and $\rmJ_s^1$ its intersection with $J_x^1$. Then $\rmJ_s^1$ is a compact open subgroup of $J_s^1$ and there is a product decomposition
$$\rmJ^1_s\cong \prod_{i=1}^r\rmJ^1_{s,i}.$$

Since $I_s^1$ is simply connected and $I_s^1\otimes_\bbQ\bbQ_p\cong J_s^1$ by Proposition \ref{prop:Tate supersingular}, the weak approximation theorem implies that $I_s^1(\bbQ)\cap \rmJ_s^1$ is dense in $\rmJ_s^1$. Since $I_s^1(\bbQ)\cap \rmJ_s^1$ preserves the subspace $\widehat{Z}'^p_s(x)\subset \mathrm{Def}_{\scrG_x,\lambda}$, it is preserved by the group $\rmJ_s^1$.

There exists an irreducible component $W$ of the smooth locus of $Z'^p(x)$ such that $s$ lies in the closure of $W$. By Proposition \ref{cor:rigidity globalized}, the tangent sheaf $T_W$ is equal to a product $$\bigoplus_{i\in\omega} T_{\calX'}(\fkp_i)\otimes_{\scrO_{\calX'}}\scrO_W$$ where $\omega\subset \{1,\dotsc,a\}$ is a proper subset since we assumed $Z'^p(x)\neq \calX'^{\fkc'}$. We fix a  choice of  $j\in\{1,\dotsc,a\}- \omega$.

Let $\xi:\Spec k[[t]]\rightarrow \mathrm{Def}_{\scrG_s,\lambda}$  be a formal curve which generically lies in $W$ and with special fiber $s$, and we write $\xi_i:\Spec k[[t]]\rightarrow\mathrm{Def}_{\scrG_{s,i},\lambda_i}$ for the $\fkp_i$-component of the map $\xi$. We write $\eta$ for the generic fiber of $\xi$; then $\eta$ is an ordinary point. It follows that $J^1_{\eta,j}$ is a torus by the description in \ref{prop:formal nbd stable}. Therefore, there exist elements $u_j^{(n)}\in \rmJ_{s,j}$ such that $u_j^{(n)}\xi\neq \xi$ but $(u_j^{(n)}\xi)$ is congruent to $\xi$ modulo arbitrarily high powers of $t$ as $n\rightarrow \infty.$ For $i\neq j$, the $\fkp_i$-component of $u_{j}^{(n)}\xi$ is equal to $\xi_i$. This implies that the tangent space to $W$ at the point $\eta:\Spec k((t))\rightarrow W$ contains non-zero $\fkp_j$-component, which is a contradiction.
\end{proof}

We may use this to deduce the following corollary

\begin{cor}\label{cor: l power Hecke orbit}
	Let $x\in \calX'^{\fkc'}\cap\calX'^{\mathrm{ord}}(\Fpbar)$. Then $Z'_l(x)=\calX'^{\fkc'}$.
\end{cor}
\begin{proof}
	By Theorem \ref{thm:Hecke orbit unitary}, it suffices to show that $Z_l'(x)$ is stable under reduced prime-to-$p$ Hecke correspondences; we use the same argument as in \cite[Theorem 1]{Chai}.
	
	By Proposition \ref{prop:orbit contains supersingular}, $Z'_l(x)$ contains a supersingular point $s\in\calX'^{\mathrm{ss}}(\Fpbar)$. Since $I_s^1$ is simply connected, the strong approximation theorem implies $I_s^1(\bbQ)I_s^1(\bbQ_l)$ is dense in $I^1_s(\bbA_f^p)\cong G'_{\mathrm{der}}(\bbA_f^p)$; in particular $I^1_s(\bbQ)\cap K^{l'}$ is dense in $K'_p$. Since $I_s^1(\bbQ)\cap K'^{l}$ is contained in the stabilizer of  $Z'_l(x)$ at $s$, the completion $\widehat{Z}'_{l,s}(x)$ at $s$ is 	stable under the action of  $K_p$. 
	
	Now strong approximation also implies that $Y'^p(s)\subset Y'_l(s)$ by the description of the isogeny classes in Proposition \ref{prop: isog classes}. For any reduced prime-to-$p$ Hecke correspondence $\gamma$ we write $\gamma(Z'_l(x))$ for the image of $Z'_l(x)$ under the Hecke correspondence $\gamma$. Then the formal completion of $Z'_l(x)\cup\gamma(Z'_l(x))$ at any supersingular point is equal to the formal completion of $Z'_l(x)$. It follows that $Z'_l(x)\cup\gamma(Z'_l(x))$ is equal to $Z'
	_l(x)\cup W$ where $W$ is a closed subscheme of $\calX'$ stable under reduced $l$-power Hecke correspondences and which does not contain supersingular points. This is a contradiction to Proposition \ref{prop:orbit contains supersingular}.
	
\end{proof}

\subsection{The result in the quaternionic case}\label{sec: Hecke orbit quaternionic}We now explain how Theorem \ref{thm:Hecke orbit unitary} may be used to deduce the Hecke-orbit conjecture for quaternionic Shimura varieties. The open subscheme $\calX'^{\mathrm{ord}}$ transfers to an open subscheme $\calX^{\mathrm{ord}}$ of $\calX$, see for example \cite{ShenZhang}\footnote{It is not a priori clear that the definitions of ordinary locus is independent of the choice of auxiliary PEL data. For the applications, it is not necessary to know this.}. We write $G_{\mathrm{der}}$ for the derived group of $G'_{\rmS}$. Then $G_{\mathrm{der}}\cong G'_{\mathrm{der}}$.

For $x\in \calX(\Fpbar)$, we may define the reduced prime-to-$p$ Hecke orbit $Y^p(x)$ and reduced $l$-power Hecke orbit $Y_l(x)$ in the same way as for $\calX'$. Namely, we let $\tilde{x}$ denote a lift of $x$ to $\underline{\Sh}_{K_p}(G_{\rmS,\rmT})(\Fpbar)$ and we write $\tilde{Y}^p(x)$ (resp. $\tilde{Y}_l(x)$) for the $G_{\mathrm{der}}(\bbA_f^p)$-orbit (resp $G_{\mathrm{der}}(\bbQ_l)$-orbit) of $\tilde{x}$. Then $Y^p(x)$ (resp. $Y_l(x)$) is defined to be the image of $\tilde{Y}^p(x)$ (resp. $\tilde{Y}_l(x)$) in $\calX$. We write $Z^p(x)$ and $Z_l(x)$ for the closures of $Y^p(x)$ and $Y_l(x)$ respectively.

Recall $\nu_{\rmS}:G_{\mathrm{S}}\rightarrow T_F$ denotes the reduced norm. Then $\nu_{\rmS}$ induces an isomorphism  $$\pi_0(\calX)\cong T_F(\bbQ)^{\dagger}\backslash  T_F(\bbA_f)/\nu_{\rmS}(K)$$
where $T_F(\bbQ)^{\dagger}$ is defined as in \ref{sec: 3 statement}. For $\fkc$ in the right hand side we write $\calX^\fkc$ for the corresonding connected component

Let $l$ be a prime coprime to $p$. The following result follows immediately from Theorem \ref{thm:Hecke orbit unitary} by using the fact that the unitary Shimura variety and quaternionic Shimura variety have isomorphic geometric connected components (upon taking the inverse limit over the prime-to-$p$ level) and the fact that $G'_{\mathrm{der}}\cong G_{\mathrm{der}}$.

\begin{thm}\label{thm:Hecke-Orbit quat}Let $\fkc\in T_F(\bbQ)^{\dagger}\backslash  T_F(\bbA_f)/\nu_{\rmS}(K)$ and $x\in  \calX^\fkc\cap \calX^{\mathrm{ord}}(\Fpbar)$. Then
	
	(1) $Z^p(x)=\calX^\fkc$.
	
	(2)  $Z_l(x)=\calX^\fkc$.

\end{thm}

\begin{rem} As in Remark \ref{rem: Hecke orbit unitary}, one can use this to deduce that the prime-to-$p$ Hecke orbit of an ordinary point is dense in $\calX$.
	\end{rem}

\section{Ihara's Lemma for Shimura surfaces}
In this section we prove a version of Ihara's lemma for certain quaternionic Shimura surfaces. The argument combines the arguments of \cite{Dim} and \cite{DT}.

\subsection{Statement of Ihara's Lemma}\label{sec:Ihara}Let $\Pi$ be an irreducible cuspidal automorphic representation of $GL_2(\A_F )$ of parallel weight 2 defined over a number field $\mathbf{E}$. We write $\calO_{\mathbf{E}}$ for the ring of integers of ${\mathbf{E}}$ and $k_\lambda$ for the residue field of $\calO_{\mathbf{E}}$ at a prime $\lambda$. Let $\rmR$ be a finite set of places of $F$ away from which $\Pi$ is unramified and $K$ is hyperspecial; in particular $\rmR$ contains the ramification set of $B$ . Let $\mathbf{T}_\rmR$ denote the (abstract) Hecke algebra away from $\rmR$; i.e. the polynomial ring over $\Z$ generated by $T_{\fkq},S_{\fkq}$ where $\fkq$ is a prime away from $\rmR$. The representation $\Pi$ determines a homomorphism $$\phi_{\rmR}^\Pi:\mathbf{T}_{\rmR}\rightarrow \Ok_{\mathbf{E}}$$ via the  Hecke eigenvalue of $\Pi$. For every prime $\lambda$ of ${\mathbf{E}}$ there is an attached  Galois representation
$$\rho_{\Pi,\lambda}:\text{Gal}(\overline{F}/F)\rightarrow GL_2(\Ok_{{\mathbf{E}}_\lambda})$$
unramified outside of $\rmR\cup\rmR_\lambda$ where $\rmR_\lambda$ is the set of primes of $F$ having the same residue characteristic as $\lambda$.  It is characterized by the property that for $\fkq\notin \rmR\cup\rmR_\lambda$, the characteristic polynomial of $\text{Frob}_\fkq$ is given by
$$X^2-\phi^\Pi_\rmR(T_\fkq)X+\text{Nm}_{F/\Q}(\fkq)\phi_{\rmR}^\Pi (S_\fkq).$$
We write $\overline{\rho}_{\Pi,\lambda}$ for the reduction  of $\rho_{\Pi,\lambda} \mod \lambda$.

We now assume $p$ is a prime which is inert in $F$ and we let $\fkp$ denote the unique prime of $F$ above $p$. Let $\fkm:=\fkm_{\Pi,\lambda}\subset \bfT_{\rmR\cup\{\fkp\}}$ denote the preimage of the ideal $(\lambda)\subset\Ok_{\mathbf{E}}$ under the map $\phi_{\rmR\cup\{\fkp\}}^\Pi$; it is a maximal ideal of $\mathbf{T}_{\rmR\cup\{\fkp\}}$. For any $\mathbf{T}_{\rmR\cup\fkp}$-module $M$, we write $M_{\fkm}$ for the localization of $M$ at the ideal $\fkm$. We  introduce the following assumptions:
\begin{ass}\label{ass: property of l} Let $l$ be the underlying prime of $\lambda$.
	
	(1) $l$ is coprime to $\rmR$, $\text{disc}(F)$, and the cardinality of $F^\times\backslash \bbA^\times_{F,f}/\bbA^\times_{F,f}\cap K$, where we consider  $\bbA^\times_{F,f}:=F\otimes_{\bbQ}\bbA_f$ as  the center of $B\otimes_{\bbQ}\bbA_f$.
	
	(2) $l\geq g+2$.
	
	(3) The representation  $\overline{\rho}_{\Pi,\lambda}$ satisfies the condition $\mathbf{LI}_{\text{Ind}\overline{\rho}_{\Pi,\lambda}}$ in \cite[Proposition 0.1]{Dim05}.
	
		(4) $B$ splits at all places above $l$ and the compact open subgroup $K$ factors as $K=K_lK^l$ where $K_l\subset GL_2(\bbQ_l)$ is a hyperspecial subgroup.
\end{ass}

\begin{prop}\label{prop: middle cohomology non-zero}Suppose Assumption \ref{ass: property of l} is satisfied. Then
	
	(1) $\rmH^j(\mathrm{Sh}_K(G_{\rmS,\rmT})_{\overline{\bbQ}},\OKE)_\fkm=0$ unless $j=g_{\rmS}:=|\Sigma_{\infty}-\rmS|$.
	
	(2) $\rmH^{g_{\rmS}}(\mathrm{Sh}_K(G_{\rmS,\rmT})_{\overline{\bbQ}},\calO_{{\mathbf{E}}_{\lambda}})_\fkm$ is a free $\OKE$-module.
\end{prop}

\begin{proof}The case of Hilbert modular varieties is proved in \cite[Theorem 6.6]{Dim}. We thus assume $\mathrm{Sh}_K(G_{\rmS,\rmT})$ is compact.
	
	For (1), it suffices to show that $\rmH^j(\text{Sh}_K(G_{\rmS,\rmT})_{\overline{\bbQ}},\kl)/\fkm=0$ for $j=0,\dotsc,g_{\rmS}-1$ by Nakayama's Lemma and Poincar\'e duality. We let $\tilde{F}$ denote the Galois closure of $F$ and $\calG_{\tilde{F}}$ its absolute Galois group. For any irreducible representation $\overline{\rho}'$ of $\calG_{\tilde{F}}$ which appears as a subquotient of  $\rmH^j(\text{Sh}_K(G_{\rmS,\rmT})_{\overline{\bbQ}},\kl)/\fkm$, the same argument as in \cite[Theorem 3.21]{Liu1} shows that  $g_{\rmS}$ appears as a Fontaine--Laffaille weight of $\overline{\rho}'$; note here we need to use Assumption \ref{ass: property of l} (2) in order to apply this theory. By Assumption \ref{ass: property of l} (4), $\mathrm{Sh}_K(G_{\rmS,\rmT})_{\bbF_{l^h}}$ admits a proper smooth  model over $W(\bbF_{l^h})$, where $h$ is the residue degree of a prime in $\tilde{F}$ above $l$. Therefore, by Faltings' Comparison Theorem \cite{Fa}, we know that $g_{\rmS}$ cannot be a Fontaine--Laffaille weight for $\rmH^j(\text{Sh}_K(G_{\rmS,\rmT})_{\overline{\bbQ}},\kl)/\fkm$, $j<g_{\rmS}$.  This proves (1).
	
	For (2) the proof is the same as \cite[Theorem 6.6]{Dim05}.
\end{proof}

For the rest of this section we make the following assumption.
\begin{ass}\label{ass: surface} 
	$|\Sigma_{\infty}-\rmS|=2$ and $B_{\rmS}\neq GL_2(F)$. \end{ass}
In particular this condition implies that the associated Shimura varieties are compact. We write $\Sigma_{\infty}-\rmS=\{\tau_1,\tau_2\}$ (note we do not use the convention in \S\ref{sec: Goren-Oort} so that $\sigma(\tau_1)$ is not necessarily equal to $\tau_2$).
The main theorem of this section is the following.

\begin{thm}\label{thm: Ihara's Lemma} Under Assumptions \ref{ass: property of l} and \ref{ass: surface}, the map $$\pi_1^*+\pi_2^*:\rmH^2(\rmSh_K(G_{\rmS,\rmT})_{\overline{\bbQ}},k_\lambda)^2_\fkm\rightarrow \rmH^2(\rmSh_{K_0(\fkp)}(G_{\rmS,\rmT})_{\overline{\bbQ}},k_\lambda)_\fkm$$
	is injective.
	\end{thm}
\begin{rem}In the case of Hilbert modular varieties, this theorem was proved by Dimitrov \cite{Dim}. In the case of Shimura curves over $\bbQ$, it  is due to Diamond and Taylor.
\end{rem}

We outline the strategy for the proof of this Theorem. Following the idea of Diamond--Taylor, we use Faltings' comparison theorem \cite{Fa} to reduce to proving the injectivity of a certain map between global sections of line bundles over the$\mod l$ reduction of the Shimura surface  (Proposition
\ref{prop: Ihara coherent}). The property that such a global section lies in the kernel of the map (\ref{eq: degeneracy maps coherent}) implies that the divisor corresponding to this section is stable under certain Hecke correspondences. The Hecke orbit conjecture proved in the previous section shows that this divisor (if non-trivial) must be supported on the complement of the ordinary locus. We then compute the intersection pairing of this divisor with certain Goren--Oort divisors to deduce a contradiction.

\ignore{Recall we have assume $l$ is a prime which is unramified in $F$ and such that $K$ is hyperspecial at $l$. In this section we need to consider the integral model $X_l$ of $\text{Sh}_K(G_{\rmS,\rmT})$ over $\calO_l$. We write $\calX_l$ for the special fiber over $\bbF_{l^g}$. Fixing an isomorphism $\overline{\bbQ}_l$ with $\bbC$ we identify $\Sigma_\infty$ with $\Sigma_l$.  We consider the two following cases separately:

We first consider case (1). Then we have $$\calX_l^{ss}=\calX_{l,\tau_1}\cup\calX_{l,\tau_c}$$ where $\calX_{l,\tau_1}$ and $\calX_{l,\tau_c}$ are the Goren-Oort divisors corresponding to $\tau_1,\tau_2\in\Sigma_{\infty}-\rmS_{\infty}.$
We write $s_l(K)$ for the number of irreducible components of $\calX_{l,\tau_1}$ and $s_l(K_0(\fkl))$ for the number of superspecial points. Note that by the description of Goren-Oort strata in Section 2, we have $s_l(K)$ is the number of elements of the discrete Shimura set $$G_{\rmS\cup \{\tau_1,\tau_c\}}(\bbQ)\backslash G_{\rmS\cup \{\tau_1,\tau_c\}}(\bbA_f)/K$$ and $s_l(K_0(\fkl))$  is the number of points in $$G_{\rmS\cup \{\tau_1,\tau_c\}}(\bbQ)\backslash G_{\rmS\cup \{\tau_1,\tau_c\}}(\bbA_f)/K_0(\fkl)$$ where $K_0(\fkl)$ denotes the compact open with Iwahori level structure at $l$. In particular, $s_l(K)$ and $s_l(K_0(\fkl))$ we have the equality $$s_l(K_0(\fkl))=(l^g+1)s_l(K).$$

\begin{prop}
	We have  the equality $$\chi_{\rmSh_K(G_{\rmS,\rmT})}(0)= s_l(K)$$
\end{prop}
\begin{proof}
Recall we have a finite etale map $$\rmSh_{K_0(\fkl)}(G_{\rmS,\rmT})\rightarrow \rmSh_{K}(G_{\rmS,\rmT})$$ of degree $l^g+1$. Therefore $\chi_{\rmSh_{K_0(\fkl)}(G_{\rmS,\rmT})}(0)=(l^g+1)\chi_{\rmSh_K(G_{\rmS,\rmT})}(0)$.

Let $\calX_{0,l}$ denote the special fiber of the integral model of $\rmSh_{K_0(\fkl)}(G_{\rmS,\rmT})$ at a prime above $\fkl$. Then by flatness, we have $\chi_{\rmSh_K(G)}(0)=\chi_{\calX_{0,l}}(0)$.
	
We will calculate $\calX_{0,l}$ using Theorem \ref{thm:Iwahori-structure}.  We have a decomposition $$\calX_{0,l}=\calN_1\cup\calN_2 \cup \calN_3$$

where $\calN_1=\calF(\calX_l)$, $\calN_2=\calV(\calX_l)$ and $\calN_3$ is the supersingular locus. We have the following exact sequence of coherent sheaves on $$0\rightarrow \calO_{\calX_{0,l}}\rightarrow \calO_{\calN_1}\oplus \calO_{\calN_2}\oplus  \calO_{\calN_3}\rightarrow  \calO_{\calN_1\cap\calN_2} \oplus\calO_{\calN_1\cap\calN_3}\oplus \calO_{\calN_2\cap\calN_2}\oplus  \calO_{\calN_1\cap\calN_2\cap\calN_3}\rightarrow 0$$
where for a closed subscheme $Y\hookrightarrow \calX_{0,l}$ we consider $\calO_Y$ as a sheaf on $\calX_{0,l}$ via pushforward. We obtain the following:
\begin{equation}\label{Euler}\chi_{\calX_{0,l}}(0)=\chi_{\calN_1}(0)+\chi_{\calN_2}(0)+\chi_{\calN_3}(0)-\chi_{\calN_1\cap\calN_2}(0)-\chi_{\calN_1\cap\calN_3}-\chi_{\calN_2\cap\calN_3}(0)+\chi_{\calN_1\cap\calN_2\cap\calN_3}(0)\end{equation}
We now compute the terms on the right hand.

(1) $\chi_{\calN_1}(0)=\chi_{\calN_2}(0)=\chi_{\calX_l}(0)$ since $\calN_1$ and $\calN_2$ are both isomorphic to $\calX_l$.

(2) $\chi_{\calN_3}(0)=(1-l^g)s_l(K)$. Note that by Theorem \ref{thm:Iwahori-structure} (2), $\calN_3$ decomposes as $\calM_1\cup\calM_c$ where $M_i$ is a $\bbP^1$-bundle over $\calX_{l,\tau_i}$,  hence a $( \bbP^1)^2$-bundle over $$G_{\rmS\cup \{\tau_1,\tau_c\}}(\bbQ)\backslash G_{\rmS\cup \{\tau_1,\tau_c\}}(\bbA_f)/K$$ and $\calM_1\cap\calM_2$  is identified with $$\calX_{l,\tau_1}^{sp}\cong G_{\rmS\cup \{\tau_1,\tau_c\}}(\bbQ)\backslash G_{\rmS\cup \{\tau_1,\tau_c\}}(\bbA_f)/K_0(\fkl)$$
Therefore taking Euler characteristics of the exact sequence of sheaves on $\calN_3$ given by $$0\rightarrow \calO_{\calN_3}\rightarrow \calO_{\calM_1}\oplus\calO_{\calM_c}\rightarrow \calO_{\calM_1\cap\calM_c}\rightarrow 0$$
we obtain $$\chi_{\calN_3}(0)=2s_l(K)-s_l(K_0(\fkl))=(1-l^g)s_l(K)$$

(3) $\chi_{\calN_1\cap\calN_3}(0)=\chi_{\calN_2\cap\calN_3}(0)=(1-l^g)s_l(K)$. This follows from the fact that $\calN_1\cap\calN_3\cong \calN_2\cap\calN_3\cong X_l^{ss}$. Then applying the same argument as in (2) to $\calX_l^{ss}$ gives the result.

(4) $\chi_{\calN_1\cap\calN_2}(0)=(1+l^g)s_l(K)$ since $\calN_1\cap\calN_2$ is isomorphic to $$\calX^{sp}_l\cong G_{\rmS\cup \{\tau_1,\tau_c\}}(\bbQ)\backslash G_{\rmS\cup \{\tau_1,\tau_c\}}(\bbA_f)/K_0(\fkl)$$

(5) $\chi_{\calN_1\cap\calN_2\cap\calN_3}(0)=(1+l^g)s_l(K)$. This follows from (4) and the fact that $$\calN_1\cap\calN_2=\calN_1\cap\calN_2\cap\calN_3.$$

Substituting the above into (\ref{Euler}), we obtain:
\begin{align*}
(l^g+1)\chi_{\calX_l}(0)&=2\chi_{\calX_l}(0)+(1-l^g)s_l(K)-2(1-l^g)s_l(K)-(1+l^g)s_l(K)+(1+l^g)s_l(K)\\
&= 2\chi_{\calX_l}(0)-(1-l^g)s_l(K)\end{align*}
The proposition follows.
\end{proof}}
\subsection{Intersection numbers}
In this section we compute the intersection numbers of certain cycles on the special fiber of quaternionic Shimura surfaces. Recall we have assumed $l$ is a prime which is unramified in $F$ and such that $K$ is hyperspecial at $l$; in other words that $K=K_lK^l$ where $K_l\subset G_{\rmS}(\bbQ_l)\cong GL_2(F\otimes_\bbQ\bbQ_l)$ is hyperspecial. In this subsection we need to consider the$\mod l$ reductions of the quaternionic Shimura varieties constructed in section \ref{sec: Section1}. Fix an isomorphism $\iota_l:\overline{\bbQ}_l\xrightarrow\sim\bbC$ we identify $\Sigma_\infty$ with $\Sigma_l$.  We let $X_l$ denote the integral model of $\text{Sh}_K(G_{\rmS,\rmT})$ over $\calO_{\mathbf{E}_{\tilde{\rmS},\tilde{l}}}$ where $\tilde{l}$ is the prime of $\mathbf{E}_{\tilde{\rmS}}$ induced by $\iota_l$; we may use the construction in \S\ref{sec: Section1} upon replacing $p$ by $l$.  In this section we will only need to consider geometric special fibers, so we do not need to keep track of the fields of definitions as in \S\ref{sec: Section1}. In particular, the subset $\rmT$ will not play a role. We write $\calX_l$ for the special fiber of $X_l$ over $\overline{\bbF}_l$.

For a choice of auxiliary PEL data, the Newton stratification on the corresponding unitary Shimura variety defines a stratification on $\calX_l$, see \S\ref{sec: 3 statement}. We write $\calX_l^{\mathrm{ord}}$ for the ordinary locus of $\calX_l$ and we write $\calX_l^{\mathrm{n-ord}}$ for the complement of $\calX_l^{\mathrm{ord}}$ in $\calX_l$. We would like to understand the intersection numbers of certain cycles supported on $\calX_l^{\mathrm{n-ord}}$. This is possible since in the case of surfaces, we may give an explicit description of the non-ordinary locus in terms of Goren--Oort cycles. The relationship with Goren--Oort strata follows  easily from an examination of the Dieudonn\'e modules of the universal $p$-divisible group over the unitary Shimura variety. Since the calculations are completely analogous to the case of Hilbert modular varieties (see for example \cite{Stamm} and \cite{LT} for some cases of totally indefinite quaternion algebras), we omit the computations and just give the statements. We consider the two following cases separately:

Case (1): There exists a prime $\mathfrak{l}$ above $l$ such that $\tau_1,\tau_c\in\Sigma_{\fkl/\infty}$.

Case (2): There exists distinct primes $\fkl_1$, $\fkl_2$ above $l$ such that $\tau_{1}\in\Sigma_{\fkl_1/\infty}$ and $\tau_2\in\Sigma_{\fkl_2/\infty}$.

We first consider Case (1). Then we have $$\calX_l^{\mathrm{n-ord}}\cong\calX_l^{\mathrm{ss}}=\calX_{l,\tau_1}\cup\calX_{l,\tau_2}$$ where $\calX_{l,\tau_1}$ and $\calX_{l,\tau_2}$ are the Goren--Oort divisors over $\overline{\bbF}_l$ corresponding to $\tau_1,\tau_2\in\Sigma_{\infty}-\rmS_{\infty}.$ Here $\calX_l^{\mathrm{ss}}$ is defined as in \S\ref{sec: 3 statement}.

 By Proposition  \ref{thm:Goren-Oort strata}, $\calX_{l,\tau_1}$ and $\calX_{l,\tau_2}$ are $\bbP^1$-bundles over the discrete Shimura set  $$G_{\Sigma_\infty}(\bbQ)\backslash G_{\Sigma_{\infty}}(\bbA_f)/K$$ and by \cite[Proposition 2.32 (3)]{TX1} there is an isomorphism $$\calX_{l,\tau_1}\cap\calX_{l,\tau_2}\cong G_{\Sigma_\infty}(\bbQ)\backslash G_{\Sigma_\infty}(\bbA_f)/K_0(\fkl)$$ where $K_0(\fkl)\subset K$ denotes the compact open which agrees with $K$ away from $\fkl$ and with Iwahori level structure at $\fkl$. Here we consider these finite sets as discrete schemes over ${\overline{\bbF}}_l$. We write $s_l(K)$ for the cardinality of $G_{\Sigma_\infty}(\bbQ)\backslash G_{\Sigma_{\infty}}(\bbA_f)/K$ and $s_l(K_0(\fkl))$ for the cardinality of $G_{\Sigma_\infty}(\bbQ)\backslash G_{\Sigma_{\infty}}(\bbA_f)/K_0(\fkl)$.  Then $s_l(K)$ and $s_l(K_0(\fkl))$ are related by the equality $$s_l(K_0(\fkl))=(l^{g_{\fkl}}+1)s_l(K)$$ where $g_{\fkl}=[F_\fkl:\bbQ_p]$. 

 By \cite[Theorem 4.3]{TX1}\footnote{In \cite{TX1}, there is a running assumption that the prime $p$ ($l$ in our setting) is inert in $F$. This is required for the cohomological results to hold. However it is easily checked that the computation of the intersection pairing, which is purely geometric, holds if the prime is only assumed unramified.} the intersection matrix of $\calX_{l,\tau_1}$ and $\calX_{l,\tau_2}$ considered as divisors on $\calX_l$ is given by $$\left(\begin{matrix} -2l^{n_{\tau_1}}s_l(K) &s_l(K_0(\fkl))\\
 s_l(K_0(\fkl)) & -2l^{n_{\tau_2}}s_l(K)
 \end{matrix}\right)=s_l(K)\left(\begin{matrix} -2l^{n_{\tau_1}} & l^{g_\fkl}+1\\
 l^{g_\fkl}+1 & -2l^{n_{\tau_2}}
 \end{matrix}\right)$$
 We write $\omega_{\tau_1}$ (resp. $\omega_{\tau_2}$) for the line bundle on $\calX_l$ defined in \S\ref{sec:moduli interpretation} corresponding to a choice of lift $\tilde{\tau_1}$ (resp, $\tilde{\tau}_c$) in $\tilde{\rmS}_{\infty}$. Since $\calX_{l,\tau_1}$ (resp. $\calX_{l,\tau_2}$) is the vanishing locus of a section of $(\omega_{\tau_2})^{l^{n_{\tau_1}}}\otimes(\omega_{\tau_1})^{-1}$ (resp. $(\omega_{\tau_1})^{l^{n_{\tau_2}}}\otimes (\omega_{\tau_2})^{-1}$), applying a change of basis matrix we compute the intersection matrix of $\omega_{\tau_1}$ and $\omega_{\tau_2}$ to be:

 $$\left(\begin{matrix}0 &s_l(K)\\
 s_l(K) &0
 \end{matrix}\right)$$

 Now we consider Case (2). In this case we have the two Goren--Oort divisors $\calX_{l,\tau_1}$ and $\calX_{l,\tau_2}$ and  $$\calX_l^{\text{n-ord}}=\calX_{l,\tau_1}\cup\calX_{l,\tau_2}.$$
 
 By Theorem \ref{thm:Goren-Oort strata}, $\calX_{l,\tau_1}$ and $\calX_{l,\tau_2}$ are isomorphic to  $\mathscr{S}_{K}(G_{\rmS_{\tau_1},\rmT_{\tau_1}})_{\overline{\bbF}_l}$ and $\mathscr{S}_{K}(G_{\rmS_{\tau_2},\rmT_{\tau_2}})_{\overline{\bbF}_l}$ respectively; these are the special fibers of certain Shimura curves. Moreover, by \cite[Theorem 5.2]{TX}, $\calX_{l,\tau_1}$ and $\calX_{l,\tau_2}$ intersect transversally and there is an identification of $\calX_{l,\tau_1}\cap\calX_{l,\tau_2}$ with the discrete Shimura set $$G_{\rmS\cup\{\tau_1,\tau_2,\fkl_1,\fkl_2\}}(\bbQ)\backslash G_{\rmS\cup\{\tau_1,\tau_2,\fkl_1,\fkl_2\}}(\bbA_f)/K_{\fkl_1,\fkl_2}=G_{\Sigma_\infty\cup\{\fkl_1,\fkl_2\}}(\bbQ)\backslash G_{\Sigma\cup\{\fkl_1,\fkl_2\}}(\bbA_f)/K_{\fkl_1,\fkl_2}$$
 where $K_{\fkl_1,\fkl_2}$ is the compact open which agrees with $K$ away from $\fkl_1$ and $\fkl_2$ and is the unique maximal compact at the places $\fkl_1$ and $\fkl_2$. We write $s_l(K_{\fkl_1,\fkl_2})$ for the cardinality of this finite set. It follows that the intersection number $\calX_{l,\tau_1}.\calX_{l,\tau_2}$ is equal to $s_l(K_{\fkl_1,\fkl_2})$.

For a projective scheme $S$ of dimension $d$ over a field $k$ and $\calF$ a coherent sheaf on $S$, we write $$\chi_S(\calF):=\sum_{i=0}^d(-1)^i\dim_k\rmH^i(S,\calF)$$ for the Euler characteristic of $\calF$. We write $\chi_S(0)$ for the Euler characteristic of the structure sheaf. \begin{prop}\label{prop:Euler char, case (2)}
 \begin{equation}\label{genus}(l^{n_{\tau_1}}-1)\chi_{\calX_{l,\tau_1}}(0)=(l^{n_{\tau_2}}-1)\chi_{\calX_{l,\tau_2}}(0)=s_l(K_{\fkl_1,\fkl_2})\end{equation}
\end{prop}
\begin{proof}
	The proof is the same as \cite[Lemma 6]{DT}.
\end{proof}
Let $K_{\mathrm{can}}$ denote the canonical bundle on $\calX_{l}$. Since $\calX_{l,\tau_1}$  (resp. $\calX_{l,\tau_2}$) is the vanishing locus of a section of  $(\omega_{\tau_1})^{l^{n_{\tau_1}}-1}$ (resp. $(\omega_{\tau_2})^{l^{n_{\tau_2}}-1}$), it follows by the adjunction formula that
 \begin{align*}2\chi_{\calX_{l,\tau_1}}(0)&=\calX_{l,\tau_1}.(\calX_{l,\tau_1}+K_{\mathrm{can}})\\
&=(\omega_{\tau_1})^{l^{n_{\tau_1}}-1}.\left((\omega_{\tau_1})^{l^{n_{\tau_1}}+1}+(\omega_{\tau_2})^2\right)\\&=(\omega_{\tau_1})^{l^{n_{\tau_1}}-1}.\left((\omega_{\tau_1})^{l^{n_{\tau_1}}+1}\right)+\frac{2}{l^{n_{\tau_2}}-1}s_l(K_{\fkl_1,\fkl_2})
\end{align*}where for the second equality we have used Corollary \ref{cor:Kodaira Spencer quaternionic}.
Thus it follows by (\ref{genus}) and the fact that $l>4$ that $\omega_{\tau_1}.\omega_{\tau_1}=0$, and hence $\calX_{l,\tau_1}.\calX_{l,\tau_1}=0$. Similarly we have $\calX_{l,\tau_2}.\calX_{l,\tau_2}=0$. 

In summary, the intersection matrix of $\calX_{l,\tau_1},\calX_{l,\tau_2}$ is given by $$\left(\begin{matrix}
0 &s_l(K_{\fkl_1,\fkl_2})\\
s_l(K_{\fkl_1,\fkl_2})& 0
\end{matrix}\right)$$ and the intersection matrix of $\omega_1$, $\omega_{\tau_2}$ is $$\frac{1}{(l^{n_{\tau_1}}-1)(l^{n_{\tau_2}}-1)}\left(\begin{matrix}
0 &s_l(K_{\fkl_1,\fkl_2})\\
s_l(K_{\fkl_1,\fkl_2})& 0
\end{matrix}\right).$$

\subsection{Connected components}\label{sec:conn comp}We will need a variant of the above computations for the connected components of $\calX_l$; this is due to the fact that we will need to apply the strong approximation theorem in the proof of Proposition \ref{prop:ss-orbit} and this only holds for the derived group of $G_{\rmS}$. Recall $\nu_{\rmS}:G_\rmS\rightarrow T_F$  denotes the reduced norm and this induces an isomorphism  \begin{equation}
\label{eq: conn comp 1}
\pi_0(\calX_l)\cong T_F(\bbQ)^{\dagger}\backslash  T_F(\bbA_f)/\nu_{\rmS}(K).\end{equation}We write $\mathrm{Cl}_F(K)$ for the right hand side.
For an element $\fkc\in\mathrm{Cl}_F(K)$, we write $\calX_l^{\fkc}$ for the connected component of $\calX_l^{\fkc}$ corresponding to $\fkc$. We also write $\omega_{\tau_1}^\fkc, \omega_{\tau_2}^{\fkc}$ for the restriction of the lines bundles $\omega_{\tau_1}, \omega_{\tau_2}$ to $\calX_l^\fkc$. We would like to obtain the intersection matrix of the line bundles  $\omega_{\tau_1}^\fkc, \omega_{\tau_2}^{\fkc}$ on $\calX_l^\fkc$.

There are maps between the Shimura sets \begin{equation}\label{eq:conn comp 2}G_{\Sigma_\infty}(\bbQ)\backslash G_{\Sigma_\infty}(\bbA_f)/K\rightarrow T_F(\bbQ)^{\dagger}\backslash  T_F(\bbA_f)/\nu_{\Sigma_\infty}(K),\quad \text{in Case (1).}\end{equation}
\begin{equation}\label{eq: conn comp 3}G_{\Sigma_\infty\cup\{\fkl_1,\fkl_2\}}(\bbQ)\backslash G_{\Sigma_\infty\cup\{\fkl_1,\fkl_2\}}(\bbA_f)/K_{\fkl_1,\fkl_2}\rightarrow T_F(\bbQ)^{\dagger}\backslash  T_F(\bbA_f)/\nu_{\Sigma_\infty\cup\{\fkl_1,\fkl_2\}}(K_{\fkl_1,\fkl_2}),\quad \text{in Case (2).}\end{equation}

In each of the equations (\ref{eq: conn comp 1}), (\ref{eq:conn comp 2}) and (\ref{eq: conn comp 3}), $T_F(\bbQ)^\dagger$ is identified with the set of totally positive elements in $F$ and the images of $K$ and $K_{\fkl_1.\fkl_2}$ under the various reduced norms are all identified, therefore we may identify the right hand sides  of (\ref{eq:conn comp 2}) and (\ref{eq: conn comp 3}) with $\mathrm{Cl}_F(K)$. We thus write  $(G_{\Sigma_\infty}(\bbQ)\backslash G_{\Sigma_\infty}(\bbA_f)/K)^{\fkc}$
(resp. $(G_{\Sigma_\infty\cup\{\fkl_1,\fkl_2\}}(\bbQ)\backslash G_{\Sigma_\infty\cup\{\fkl_1,\fkl_2\}}(\bbA_f)/K_{\fkl_1,\fkl_2})^{\fkc}$) for the preimage of $\fkc$ under (\ref{eq:conn comp 2}) (resp. (\ref{eq: conn comp 3})), and we write $s_l(K)^{\fkc}$  (resp $s_l(K_{\fkl_1,\fkl_2})^{\fkc}$) for its cardinality. 

Suppose we are in Case (1). For $i=1,2$, we have an identification of the irreducible components  of $\calX_{l,\tau_i}$ with $G_{\Sigma_{\infty}}(\bbQ)\backslash G_{\Sigma_{\infty}}(\bbA_f)/K$. We may choose this identification compatibly with (\ref{eq: conn comp 1}). In other words, if we denote by $\calX_{l,\tau_i}^{\fkc}$  the union of irreducible components of $\calX_{l,\tau_i}$ corresponding to $(G_{\Sigma_{\infty}}(\bbQ)\backslash G_{\Sigma_{\infty}}(\bbA_f)/K)^{\fkc}$, then $\calX_{l,\tau_i}^\fkc=\calX_{l,\tau_i}\cap\calX_l^\fkc$ for $i=1,2$.

\begin{prop}\label{prop:int matrix case 1}In Case (1), the intersection matrix of
	 $\omega_{\tau_1}^\fkc,\omega_{\tau_1}^\fkc$ is given by $$\left(\begin{matrix}0 &s_l(K)^\fkc\\
	s_l(K)^\fkc &0
	\end{matrix}\right).$$

\end{prop}
\begin{proof} Using the fact that $\calX_{l,\tau_i}\cap \calX_l^{\fkc}\cong \calX_{l,\tau_i}^{\fkc}$ and \cite[Theorem 4.3]{TX1}, we find that in Case (1) the intersection matrix of $\calX_{l,\tau_1}$, $\calX_{l,\tau_2}$ is given by
	$$s_l(K)^\fkc\left(\begin{matrix} -2l^{n_{\tau_1}} & l^{g_\fkl}+1\\
	l^{g_\fkl}+1 & -2l^{n_{\tau_2}}
	\end{matrix}\right).$$
Applying a change of basis matrix gives the desired result.
\end{proof}
Now suppose we are in Case (2). For $i=1,2$ we may identify $$\pi_0(\calX_{l,\tau_i})\cong \mathrm{Cl}_F(K)$$ compatibly with the identification $ \pi_0(\calX_{l})\cong \mathrm{Cl}_F(K)$. In other words if we write $\calX_{l,\tau_i}^\fkc$ for the component of $\calX_{l\tau_i}$ corresponding to $\fkc$, we have $\calX_{l,\tau_i}^\fkc=\calX_{l}^\fkc\cap\calX_{l,\tau_i}$ for $i=1,2$.

Applying the argument of \cite[Lemma 6]{DT} to each connected component we obtain the following.
\begin{prop}\label{prop:genus} In Case (2), there is an equality
	$$(l^{n_{\tau_1}}-1)\chi_{\calX^\fkc_{l,\tau_1}}(0)=(l^{n_{\tau_2}}-1)\chi_{\calX^\fkc_{l,\tau_2}}(0)=s_l(K_{\fkl_1,\fkl_2})^\fkc$$
\end{prop}
We can use this to compute the intersection matrix of $\omega_{\tau_1}^\fkc, \omega_{\tau_2}^\fkc$ as in the previous subsection.

\begin{prop}\label{prop:int matrix case 2}In Case (2), the intersection matrix of $\omega_{\tau_1}^\fkc, \omega_{\tau_2}^\fkc$ is given by $$\frac{1}{(l^{n_{\tau_1}}-1)(l^{n_{\tau_2}}-1)}\left(\begin{matrix}
	0 &s_l(K_{\fkl_1,\fkl_2})^\fkc\\
	s_l(K_{\fkl_1,\fkl_2})^\fkc& 0
	\end{matrix}\right).$$ 
\end{prop}
\begin{proof}
	Using the fact that $\calX_l^\fkc$ and $\calX_l^\fkc$ intersect transversally and that there are $s_l(K_{\fkl_1,\fkl_2})^\fkc$ intersection points, we see that $\calX_{l,\tau_1}.\calX_{l,\tau_2}=s_l(K_{\fkl_1,\fkl_2})^\fkc$. Using the adjunction formula applied to each component as in the previous section and Proposition \ref{prop:genus} we find that $\calX^\fkc_{l,\tau_i}.\calX^\fkc_{l,\tau_i}=0$, and hence $\omega^\fkc_{\tau_i}.\omega^\fkc_{\tau_i}$. for $i=1,2$.
\end{proof}

We write $\rmSh_{K_0(\fkp)}(G_{\rmS,\rmT})$ for the Shimura variety  with standard Iwahori level at $p$. We let $\calX_{0}(\fkp)_l$ denote the special fiber of the integral model for $\rmSh_{K_0(\fkp)}(G_{\rmS,\rmT})$ over $\overline{\bbF}_l$. We note that, since we have only changed level away from $l$, $\calX_{0}(\fkp)_l$ is still smooth.

Recall there are  two  degeneracy maps $$\pi_1,\pi_2:\text{Sh}_{K_0(\fkp)}(G_{\rmS,\rmT})\rightarrow \text{Sh}_K(G_{\rmS,\rmT}).$$
 We also write $\pi_1,\pi_2:\calX_0(\fkp)_l\rightarrow\calX_l$ for the corresponding maps on the special fiber; both of these maps are \'etale.

\begin{defn}\label{def: Hecke orbit}
	Let $x,x' \in \calX_l(\overline{\bbF}_l)$. We write $x\sim x'$ if there exists $\tilde{x},\tilde{x}'\in \calX_{0}(\fkp)_l(\overline{\bbF}_l)$ such that $$\pi_1(\tilde{x})=x,  \pi_1(\tilde{x}')=x', \text{ and }\pi_2(\tilde{x})=\pi_2(\tilde{x}')$$ This defines an equivalence relation on $\calX_l(\overline{\bbF}_l)$.
\end{defn}
It is easy to see that for $x\in\calX_l(\overline{\bbF}_l)$, the $\sim$ equivalence class of $x$ may be identified with $Y_p(x)$ in the notation of \S\ref{sec: Hecke orbit quaternionic}. Note that the roles of $p$ and $l$ have been switched.

Recall we have the discrete Shimura sets $G_{\Sigma_{\infty}}(\bbQ)\backslash G_{\Sigma_\infty}(\bbA_f)/K$, $G_{\Sigma_{\infty}}(\bbQ)\backslash G_{\Sigma_\infty}(\bbA_f)/K_0(\fkp)$ and natural degeneracy maps $$\pi_1,\pi_2:G_{\Sigma_{\infty}}(\bbQ)\backslash G_{\Sigma_\infty}(\bbA_f)/K_0(\fkp)\rightarrow G_{\Sigma_{\infty}}(\bbQ)\backslash G_{\Sigma_\infty}(\bbA_f)/K.$$ Similarly to the above we may define an equivalence relation on this set by specifying $x\sim x'$ for $x,x'\in G_{\Sigma_{\infty}}(\bbQ)\backslash G_{\Sigma_\infty}(\bbA_f)/K$ if there exists $\tilde{x},\tilde{x}'\in G_{\Sigma_\infty}(\bbQ)\backslash G_{\Sigma_{\infty}}(\bbA_f)/K_0(\fkp)$ such that $$\pi_1(\tilde{x})=x,  \pi_1(\tilde{x}')=x', \text{ and }\pi_2(\tilde{x})=\pi_2(\tilde{x}').$$ We also write $Y_p(x)$ for the $\sim$ equivalence class of $x$.

\begin{prop}\label{prop:ss-orbit}For any $x\in (G_{\Sigma_{\infty}}(\bbQ)\backslash G_{\Sigma_\infty}(\bbA_f)/K)^{\fkc}$, we have $Y_p(x)=(G_{\Sigma_{\infty}}(\bbQ)\backslash G_{\Sigma_\infty}(\bbA_f)/K)^{\fkc}$.
	
\end{prop}

\begin{proof}Let $G^1_{\Sigma_\infty}$ the kernel of $\nu_{\Sigma_\infty}:G_{\Sigma_\infty}\rightarrow T_F$. Then $G^1_{\Sigma_\infty}$ is the derived group of $G_{\Sigma_\infty}$ and is simply connected. Let $g\in G_{\Sigma_{\infty}}(\bbA_f)$ be a  representative of $(G_{\Sigma_{\infty}}(\bbQ)\backslash G_{\Sigma_\infty}(\bbA_f)/K)^{\fkc}$. It is easy to see that  \begin{align*}(G_{\Sigma_{\infty}}(\bbQ)\backslash G_{\Sigma_\infty}(\bbA_f)/K)^{\fkc}&=G_{\Sigma_{\infty}}^1(\bbQ)\backslash G_{\Sigma_\infty}^1(\bbA_f)gK/K\\ &\cong G_{\Sigma_{\infty}}^1(\bbQ)\backslash G_{\Sigma_\infty}^1(\bbA_f)gKg^{-1}/gKg^{-1}; \end{align*}
	the isomorphism is given by right multiplication by $g^{-1}$.

Then $Y_p(x)$ is identified with $$G^1_{\Sigma_{\infty}}(\bbQ)\cap K^p\backslash G^1_{\Sigma_\infty}(\bbQ_p)gKg^{-1}/gKg^{-1}\subset G_{\Sigma_\infty}^1(\bbQ)\backslash G_{\Sigma_\infty}^1(\bbA_f)gKg^{-1}/gKg^{-1}.$$
	By the strong approximation theorem, $ G^1_{\Sigma_\infty}(\bbQ)G^1_{\Sigma_\infty}(\bbQ_p)$ is dense in $G^1_{\Sigma_\infty}(\bbA_f)$ and hence $$G^1_{\Sigma_{\infty}}(\bbQ)\cap K^p\backslash G^1_{\Sigma_\infty}(\bbQ_p)gKg^{-1}/gKg^{-1}= G_{\Sigma_\infty}^1(\bbQ)\backslash G_{\Sigma_\infty}^1(\bbA_f)gKg^{-1}/gKg^{-1}.$$
	
 
\end{proof}

\subsection{Ihara's Lemma}We now prove the main Theorem.

\begin{prop}\label{prop: Ihara coherent}
	The map \begin{equation}
	\label{eq: degeneracy maps coherent}
\pi_1^*+\pi_2^*:\rmH^0(\calX_l,\Omega^2_{\calX_l/\overline{\bbF}_l})^2\rightarrow \rmH^0(\calX_0(\fkp)_l,\Omega^2_{\calX_0(\fkp)_l/\overline{\bbF}_l})\end{equation}
	is injective.
\end{prop}
\begin{proof}Let $(f_1,f_2)$ be an element of the kernel. We show $f_1=0$; in fact we will show $f_1|_{\calX_l^{\fkc}}=0$ for all $\fkc$.  Assume for contradiction that $f_1|_{\calX_l^{\fkc}}\neq0. $ 
Suppose $f_1$ has  a zero  at a point $x\in \calX_l^{\fkc}(\Fpbar)$. Let $y,z\in \calX_l(\Fpbar)$ be such that $$\pi^{-1}_1(x)\cap \pi^{-1}_2(y)\neq\emptyset,\quad \pi^{-1}_1(z)\cap \pi^{-1}_2(y)\neq\emptyset.$$ Then since $\pi_1^*(f_1)=-\pi_2^*(f_2)$, $f_1$ has a  zero at $z$. This implies $f_1$ vanishes at every point of  $Y_p(x)$. If $f_1$ vanishes at an ordinary point of $\calX_{l}^\fkc$, then by Theorem \ref{thm:Hecke-Orbit quat}, $f_1|_{\calX_l^{\fkc}}=0$.

 Therefore $\mathrm{div}(f_1|_{\calX_{l}^{\fkc}})$ is supported on the complement of the ordinary locus.
	
	We first consider Case (1). Then $\text{div}(f_1)$ is supported on $\calX_{l}^{\mathrm{ss}}$.
	
	 Let us consider $D^{\fkc}:=\mathrm{div}(f_1)|_{\calX_l^{\fkc}}$ for some $\fkc\in\mathrm{Cl}_F(K)$. Then $D^\fkc$ is supported on $\calX_l^{\fkc}\cap\calX_l^{\mathrm{ss}}$. Recall we may identify $\calX_l^{\fkc}\cap\calX_l^{\mathrm{ss}}$  with two copies of$$(G_{\Sigma_{\infty}}(\bbQ)\backslash G_{\Sigma_\infty}(\bbA_f)/K)^{\fkc}$$  corresponding to $\tau_1$ and $\tau_2$. We let $D_{\tau_1}^{\fkc}$ (resp. $D_{\tau_2}^{\fkc}$) be the divisor corresponding to the sum of irreducible components in $\calX_{l,\tau_1}^\fkc=\calX_{l,\tau_1}\cap\calX_{l}^{\fkc}$ (resp. $\calX_{l,\tau_2}^\fkc=\calX_{l,\tau_2}\cap\calX_{l}^{\fkc}$). We claim there are non-negative integers $a,b$ such that
	 $$D^{\fkc}=aD_{\tau_1}^{\fkc}+bD_{\tau_2}^{\fkc}.$$
	 Indeed if $f_1$ vanishes on some irreducible component of $\calX_{l,\tau_1}$ then by Proposition \ref{prop:ss-orbit}, $f_1$ vanishes on every irreducible component of $\calX_{l,\tau_1}$. We let $h_{\tau_1}\in \rmH^0(\calX_l,(\omega_{\tau_1})^{l^{n_{\tau_1}}}\otimes \omega_{\tau_2}^{-1})$ be a section with  divisor $D_{\tau_1}^\fkc$. The section $h_{\tau_1}$ can be chosen to be compatible with prime-to-$l$ level structure and so $\pi_1^*(h_{\tau_1})=\pi_2^*(h_{\tau_1})$. Then $(f_1h_{\tau_1}^{-1},f_2h_{\tau_1}^{-1})\in\ker(\pi_1^*+\pi_2^*)$. Repeating the argument we see that there exists $a$ such that $f_1h_{\tau_1}^{-a}|_{\calX_{l}^{\fkc}}$ is non-vanishing on $\calX_{l,\tau_1}^{\fkc}$. We may use the same argument for $\calX_{l,\tau_2}^{\fkc}$ to obtain $b$ such that $f_1h_{\tau_2}^{-b}|_{\calX_{l}^{\fkc}}$ is non-vanishing on  $\calX_{l,\tau_1}^{\fkc}$. It follows that the zero-locus of $f_1h_{\tau_1}^{-a}h_{\tau_2}^{-b}|_{\calX_{l}^{\fkc}}$ has codimension 2 and hence is empty; this proves the claim.
	 
Now $f_1$ corresponds to a section of the line bundle $\Omega^2_{\calX_l/\overline{\bbF}_l}$, which is isomorphic to  $(\omega_{\tau_1})^2\otimes(\omega_{\tau_2})^2$ up to a torsion element of the Picard group by Proposition \ref{cor:Kodaira Spencer quaternionic}. Furthermore $D_{\tau_1}^{\fkc}$ (resp. $D_{\tau_2}^{\fkc}$) is the divisor corresponding to  a global section of $(\omega^\fkc_{\tau_2})^{l^{n_{\tau_1}}}\otimes(\omega^\fkc_{\tau_1})^{-1}$ (resp.  $(\omega^\fkc_{\tau_1})^{l^{n_{\tau_2}}}\otimes(\omega^\fkc_{\tau_2})^{-1}$). Therefore the line bundle
$$((\omega^\fkc_{\tau_2})^{l^{n_{\tau_1}}}\otimes(\omega^\fkc_{\tau_1})^{-1})^a\otimes((\omega^{\fkc}_{\tau_1})^{l^{n_{\tau_2}}}\otimes(\omega^{\fkc}_{\tau_2})^{-1})^b\otimes\Omega^2_{\calX_l/\overline{\bbF}_l}$$ has a non-vanishing section; hence it must be the trivial bundle. It follows that 
 $$((\omega^{\fkc}_{\tau_2})^{l^{n_{\tau_1}}}\otimes(\omega^{\fkc}_{\tau_1})^{-1})^a\otimes((\omega^{\fkc}_{\tau_1})^{l^{n_{\tau_2}}}\otimes(\omega^{\fkc}_{\tau_2})^{-1})^b\otimes (\omega^{\fkc}_{\tau_1})^{-2}\otimes(\omega^{\fkc}_{\tau_2})^{-2}\cong (\omega^{\fkc}_{\tau_2})^{l^{an_{\tau_1}}-b-2}\otimes (\omega^{\fkc}_{\tau_1})^{l^{bn_{\tau_2}}-a-2}$$
is torsion in the Picard group of $\calX_l$, hence pairs with any divisor to be $0$. Since $n_{\tau_1},n_{\tau_2}>0$, at least one of $l^{an_{\tau_1}}-b-2$, $l^{bn_{\tau_1}}-a-2$ is non-zero. 

Recall the intersection matrix for the line bundles $\omega^{\fkc}_{\tau_1}$ and $\omega^{\fkc}_{\tau_2}$ on $\calX_l^{\fkc}$ is given by $$\left(\begin{matrix} 0& {s_l(K)}^\fkc \\{s_l(K)}^\fkc & 0\end{matrix}\right).$$
Thus $(\omega^{\fkc}_{\tau_2})^{l^{an_{\tau_1}}-b-2}\otimes (\omega^{\fkc}_{\tau_1})^{l^{an_{\tau_1}}-b-2}$ pairs with either $\omega^{\fkc}_{\tau_1}$ or $\omega^{\fkc}_{\tau_2}$ to give a non-zero number. This is a contradiction

We now consider Case (2). Then $\text{div}(f_1)$ is supported on $\calX_{l,\tau_1}\cup\calX_{l,\tau_2}$. Let $D^{\fkc}:=\text{div}(f_1)|_{\calX_{l}^{\fkc}}$ for some $\fkc\in \text{Cl}_F(K)$, and we write $D_{\tau_1}^{\fkc_1}$ (resp. $D_{\tau_2}^{\fkc}$) denote the divisor $\calX_{l,\tau_1}\cap\calX_{l}^{\fkc}$.  Then  $$D^{\fkc}=aD_{\tau_1}^{\fkc}+bD_{\tau_2}^{\fkc}$$ for some integers $a,b$ since $D_{\tau_1}^\fkc$ and $D_{\tau_2}^\fkc$ are primitive divisors. Since $D_{\tau_1}^{\fkc}$ (resp. $D_{\tau_2}^{\fkc}$) is the vanishing locus of a section of $(\omega_{\tau_1})^{al^{n_{\tau_1}}-1}$ (resp. $(\omega_{\tau_2})^{bl^{n_{\tau_2}}-1}$), the same argument as in case (1)  shows that  $$(\omega_{\tau_1})^{al^{n_{\tau_1}}-3}\otimes(\omega_{\tau_2})^{bl^{n_{\tau_2}}-3}$$  is a torsion element in the Picard group of $\calX_l$. Since $l> 4$, if $a\neq 0$,  intersecting with $D_{\tau_2}$ gives a non-zero number; hence we obtain a contradiction. Similarly if $b\neq 0$.
\end{proof}

\begin{proof}[Proof of Theorem \ref{thm: Ihara's Lemma}]Let $$W:=\rmH^2(\rmSh_K(G_{\rmS,\rmT})_{\overline{\bbQ}},k_\lambda)^2_\fkm,\ \ W_0(\fkp):=\rmH^2(\rmSh_K(G_{\rmS,\rmT})_{\overline{\bbQ}},k_\lambda)^2_\fkm$$ 
	The image $\bfT_{\rmR\cup\{\fkp\}}^W$ of $\bfT_{\rmR\cup\{\fkp\}}$ in $\text{End}_{k_\lambda}(W)$ is a local Artinian ring and $\fkm^iW$ is a finite decreasing filtration of $W$. By the freeness result in Proposition \ref{prop: middle cohomology non-zero} (2), each graded piece $\fkm^iW/\fkm^{i+1}W$ is a quotient of two $\bfT_{\rmR\cup\{\fkp\}}[\calG_{\tilde{F}}]$-lattices in $\rmH^2(\rmSh_K(G_{\rmS,\rmT})_{\overline{\bbQ}},\calO_{{\mathbf{E}_\lambda}})_\fkm$. By Lemma \cite[Lemma 6.5]{Dim05} and the Eichler--Shimura congruence relation proved by Nekovar \cite[A.6]{Nek}, the irreducible subquotients of $W$  are all isomorphic as Galois representations of $\calG_{\tilde{F}}$. The same statement holds for the irreducible subquotients of $W_0(\fkp)$. Using Faltings' comparison theorem \cite{Fa}, we may therefore check the injectivity on the last graded pieces of the Fontaine--Laffaille filtration. By the degeneracy of the Hodge--de Rham spectral sequence proved in \cite{DI}, this follows from  the assertion of \ref{prop: Ihara coherent}.

	\end{proof}

\section{Abel--Jacobi map and geometric level-raising}\label{sec:lr}
In this section we will use the results from the previous sections to construct classes in the motivic cohomology of quaternionic Shimura varieties. We begin by recalling the definition of higher Chow groups and the associated cycle class maps.
\subsection{Higher Chow groups and $l$-adic cycle class maps.}\label{sec: Higher Chow groups} Let $X$ be a smooth variety over a field $k$ and let $\Delta^n$ denote the standard $n$-simplex $\Spec k[x_0,...,x_n]/(\sum_{i=0}^n x_i -1)$. For integers $n,r$, we define $z^r(X,n)$ to be the free abelian group generated by the integral closed subvarieties $Z$ of $X\times\Delta^n$  such that for any face $F\subset \Delta^n$ we have  $$\text{codim}_{X\times F}(Z\cap(X\times F))\geq r.$$ The groups $z^r(X,n)$ fit into a complex  \begin{equation}\label{eq: higher Chow complex}
...\rightarrow z^r(X,n)\rightarrow z^r(X,n-1)\rightarrow...\rightarrow z^r(X,0)\rightarrow 0\end{equation}
where the differential is given by taking the alternating sum of the induced face maps.
 The {\it higher Chow group} $\text{Ch}^r(X,n)$  is the defined to the be $n^{\mathrm{th}}$ homology of the above complex. It is easy to see that  $\text{Ch}^r(X,0)$ is the standard Chow group of codimension $r$ cycles on $X$. By \cite{Voe2}, we have an isomorphism $$\text{Ch}^j(X,2j-i)\cong \text{H}_{\calM}^i(X,\mathbb{Z}(j))$$
where $\text{H}_{\calM}^i(X,\mathbb{Z}(j))$ is the motivic cohomology group of \cite{SuVo}.

We may also consider a variant of this construction by introducing coefficients. Let $R$ be any ring and we let $\mathrm{Ch}^r(X,n,R)$ denote the $n^{\mathrm{th}}$ homology of the sequence (\ref{eq: higher Chow complex}) tensored with $R$. As before there is an isomorphism $\text{Ch}^j(X,2j-i,R)\cong \text{H}_{\calM}^i(X,R(j))$; see \cite[Corollary 2]{Voe2}. From now on, we will work with higher Chow groups.

Let $l$ be a prime which is invertible in $k$, then there is an $l$-adic cycle class map
$$\text{Ch}^j(X,2j-i)\rightarrow \text{H}_{\text{cont}}^i(X,\mathbb{Z}_l(j))$$
where $\text{H}_{\text{cont}}^i(X,\mathbb{Z}_l(j))$ is the continuous \'etale cohomology defined by \cite{Jann}. This coincides with the usual cycle class map when $i=2j$. 

Similarly if $k_\lambda$ is a finite extension of $\bbF_l$, there is a cycle class map $$\text{Ch}^j(X,2j-i,k_\lambda)\rightarrow \text{H}^i(X,k_\lambda(j)).$$

Let $Y\rightarrow X$ denote a fibration with connected smooth fibers of dimension $s$ and $R$ any coefficient ring. Then taking preimages of cycles under the projection $Y\times\Delta^n\rightarrow X\times\Delta^n$ induces a map \begin{equation}\label{eq: mot coh fibration}\text{Ch}^r(X,n,R)\rightarrow \text{Ch}^{r}(Y,n,R).\end{equation}

Finally let $Z\hookrightarrow X$ be a closed immersion where $Z$ is smooth and of codimension $t$ in $X$. Then pushforward of cycles along $Z\times\Delta^n\rightarrow X\times\Delta^n$ induces a map \begin{equation}\label{MotCoh-imm}
\text{Ch}^r(Z,n,R)\rightarrow \text{Ch}^{r+t}(X,n,R).\end{equation}

\subsection{Motivic cohomology of surfaces and dual graphs}\label{sec: dual-graph}We are particularly interested in the case of surfaces. In this case, the motivic cohomology $\rmH_{\calM}^3(X,\Z(2))\cong \mathrm{Ch}^2(X,1)$ is given by the homology of the following sequence (see  \cite{Scholl} for example):
\begin{equation}\label{MotCoh-surface}K_2(k(X))\xrightarrow{\partial} \oplus_{S\subset X} k(S)^\times \xrightarrow{\mathrm{div}} \oplus_{x\in X} \bbZ
\end{equation}
Here the middle sum runs over the set of irreducible curves $S\subset X$ and $k(X)$ (resp. $k(S)$) denotes the field of rational functions on  $X$ (resp. $S$). The term $K_2(k(X))$  denotes the second Milnor $K$-group of the rational function field $k(X)$ and the $S$-component of the  map $\partial$ is the tame symbol associated to the valuation $\text{ord}_S$. The map $\mathrm{div}$ sends a rational function $f$ on $S$ to its divisor $\text{div}(f)$.

There is a special case where we can understand a part of $\text{Ch}^2(X,1)$ in a purely combinatorial way. If $k'/k$ is a finite extension, we write $\bbP^1_{k'/k}$ for the projective line $\bbP^1_{k'}$ over $k'$ considered as a $k$-scheme. Therefore $\bbP^1_{k'/k}\otimes_k\overline{k}$ can be identified with $[k':k]$ copies of $\bbP^1_{\overline{k}}$ corresponding to the embeddings $k'\rightarrow\overline{k}$. Let $Y\subset X$ be a codimension 1 subvariety satisfying the following conditions:

(1) Each irreducible  component $S$ of $Y$ is isomorphic to $\bbP_{k_S/k}^1$ where $k_S$ is a finite extension of $k$.

(2) Any two irreducible components of $Y_{\overline{k}}$ intersect transversally and no three components have a common intersection point. 

The non-smooth points of $Y_{\overline{k}}$ are the intersection points of the components in $Y_{\overline{k}}$ and it is naturally a closed subscheme of $X$ defined over $k$. We note that by (1), the irreducible components of $Y_{\overline{k}}$ are isomorphic to $\bbP^1_{\overline{k}}$.

\begin{defn}\label{def:subgroup of higher Chow P1}
We define $\text{Ch}^2_Y(X,1)$ to be the subgroup of $\text{Ch}^2(X,1)$ supported on $Y$. In other words, it is generated by elements of the form $\bigoplus_{S\subset X}f_S$ where $f_S\in k(S)^\times$ is trivial unless $S\subset Y$. 
\end{defn}
 We now describe how ${\text{Ch}}^2_Y(X,1)$ can be interpreted in terms of the combinatorics of the configuration of $\bbP_{\overline{k}}^1$'s on  $Y_{\overline{k}}$. 
 
 \begin{defn} The dual graph $\fkG$ associated to $Y_{\overline{k}}$ is the unoriented graph defined by the following:

(i) The set of vertices $\scrV$ is identified with the set of irreducible components of $Y_{\overline{k}}$. For $i\in\scrV$, we let $S_i$ denote the corresponding irreducible component.

(ii) The set of edges $\mathscr{E}$ is identified with the set of intersection points of two divisors in $Y_{\overline{k}}$, where an edge $e$ connects $i,j\in\scrV$  if $e\in S_i\cap S_j$.
 \end{defn}
The Galois group $\Gal(\overline{k}/k)$ naturally acts on $\fkG$ and hence on 	the homology $\rmH_1(\fkG,\Z)$ of the graph $\fkG$. We will define a  map $$\Theta:\rmH_1(\fkG,\Z)^{\Gal(\overline{k}/k)}\rightarrow {\text{Ch}}^2_Y(X,1).$$

	To do this, we first calculate $\rmH_1(\fkG,\Z)$. Consider the bouquet of circles $\tilde{\fkG}$ given by contracting all the elements of $\scrV$. For each $e\in\scrE$ we choose an orientation of $e$, i.e. an ordering $(v^1(e),v^2(e))$ of the two vertices adjacent to $e$. This choice determines an isomorphism
	$$\rmH_1(\tilde{\fkG},\Z)\cong \Z^{\scrE}.$$ 
	We may then identify $\rmH_1(\fkG,\Z)$ as the subgroup of $\Z^\scrE$ corresponding to  the kernel of the map
	$$d:\Z^\scrE\rightarrow \Z^\scrV$$ given by $$e\mapsto v_1(e)-v_2(e).$$
		
	The map $\Theta$ can then be defined as follows. Fix a basis of the free $\Z$-module $\rmH_1(\fkG,\Z)^{\Gal(\overline{k}/k)}$ and let $m:=(m_e)_{e\in\scrE}\in \rmH_1(\fkG,\Z)^{\Gal(\overline{k}/k)}$ be an element of this basis. 
	Let $i\in\scrV$. We let $Z(i),P(i)\subset\scrV$  be the two subsets of $\scrV$ defined by
	$$Z(i)=\{e\in \scrV| v_1(e)=i\}$$
	$$P(i)=\{e\in\scrV|v_2(e)=i\}.$$
	By definition, $Z(i)$ and $P(i)$ are identified with subsets of points on $S_i$. 
\begin{lem}There exists an element $(f^m_S)_{S\subset Y}\in\oplus_{S\subset Y}k(S)^\times$ such that its image  $(f^m_i)_{i\in\scrV}\in\oplus_{i\in\scrV}\overline{k}(S_i)^\times$ satisfies  $$\mathrm{div}(f^m_i)=\sum_{e\in Z(i)}m_e(e)-\sum_{e\in P(i)}m_e(e).$$
	\end{lem}
\begin{proof}Let $S\cong \bbP_{k_S/k}^1$ be an irreducible component of $Y$. Let $\scrV_S\subset\scrV$ denote the set of irreducible components of $S_{\overline{k}}$. We consider the divisor $$\sum_{i\in \scrV_S}\sum_{e\in Z(i)}m_e(e)-\sum_{i\in \scrV_S}\sum_{e\in P(i)}m_e(e)$$ on $S_{\overline{k}}$. Since  $m\in\ker(d)$, this divisor has degree 0 on every component $S_i$. Moreover, it is Galois invariant, hence there exists a function $f_S^m\in k(S)^\times$ with this divisor (we note that $k(S)$ is isomorphic to the field of rational functions in one variable $k_S(x)$ over $k_S$). Repeating the argument for all components $S\subset Y$ proves the lemma.
\end{proof}

Let $(f^m_S)_{S\subset Y}$ be an element as in the lemma above. Then we have
\begin{align*}\sum_{S\subset Y}\mathrm{div}(f_S^m)&=\sum_{i\in\scrV}\text{div}(f_i^m)\\
&=\sum_{i\in\scrV}\sum_{e\in Z(i)}m_e(e)-\sum_{i\in\scrV}\sum_{e\in P(i)}m_e(e)\\
&=\sum_{e\in \scrE}m_e(e)-\sum_{e\in \scrE}m_e(e)\\
&=0.
\end{align*}
Thus by (\ref{MotCoh-surface}), the tuple $(f_S^m)_{S\subset Y}$ gives an element of $\text{Ch}_Y^2(X,1)$ and we define $\Theta(m)$ to be this element. Repeating for the other basis elements and extending linearly gives the definition of $\Theta$.

In fact the map $\Theta$ is ``essentially'' surjective in the following sense. Note that the choice of $f_S^m$ is well-defined up to multiplication by an element of $k(S)^\times$ whose divisor is 0, i.e. an element of $k_S^\times$. It follows that $\Theta$ induces a map
\begin{equation}\label{Theta}
\tilde{\Theta}:\left(\bigoplus_{S\subset Y}k_S^\times\right)\oplus \rmH_1(\fkG,\Z)^{\Gal(\overline{k}/k)}\rightarrow \text{Ch}^2_Y(X,1).\end{equation}

\begin{prop}\label{prop: level raising group surjects onto Chow group}
	The map $\tilde{\Theta}$ is surjective.
\end{prop}
\begin{proof}
	Let $(f_S)_{S\subset Y}\in \oplus_{S\subset Y}k(S)^\times$ be a representative of an element of $\text{Ch}^2_Y(X,1)$. We write $(f_i)_{i\in\scrV}$ for its image in $\oplus_{i\in\scrV}\overline{k}(S_i)$. Since $\sum_{S\subset Y}\text{div}(f_S)=\sum_{i\in\scrV}\text{div}(f_i)=0$, the divisor $\text{div}(f_i)$ can only be supported on the intersection of $S_i$ with some other $S_j$, i.e. on points corresponding to elements of $\scrE$. 
	
	For each $e\in\scrE$, let $S_i$ denote the divisor corresponding to $v_1(e)$ and define $m_e\in\Z$ to be the order of the zero of $f_1$ at $e$. Reversing the above argument, we see that $m=(m_e)_{e\in \scrE}$ defines an element of $ \rmH_1(\fkG,\Z)$ which is Galois invariant since $\mathrm{div}(f_S)$ is. 
	
	By definition the $S^{\mathrm{th}}$-component of $\Theta(m)$ has the same divisor as $f_S$, thus we may modify $\Theta(m)$ by an element of $\bigoplus_{S\subset Y}k_S^\times$ to get the element $(f_S)_{S\subset Y}$.
\end{proof}

We will consider an analogue of this construction with torsion coefficients. Indeed upon tensoring with $k_\lambda$ we obtain a map \begin{equation}\label{eq: level raising group surjects onto Chow group}\tilde{\Theta}_{k_\lambda}:\left(\bigoplus_{S\subset Y }k_S^\times\otimes_{\bbZ}k_\lambda\right)\oplus\rmH_1(\fkG,k_\lambda)^{\Gal(\overline{k}/k)}\rightarrow {\text{Ch}}^2_Y(X,1,k_\lambda)\end{equation}
and same proof as Proposition \ref{prop: level raising group surjects onto Chow group} shows that $\tilde{\Theta}_{k_\lambda}$ is surjective. Here we consider the abelian group $\bigoplus_{S\subset Y}k_S^\times$ as a $\bbZ$-module in taking the tensor product, and in defining $\mathrm{Ch}(X,1,k_\lambda)$ we may take the homology of the sequence (\ref{MotCoh-surface}) tensored with $k_\lambda$; the subgroup $\mathrm{Ch}_Y(X,1,k_\lambda)$ is then defined in the same way as Definition \ref{def:subgroup of higher Chow P1}.

\subsection{Motivic cohomology and level-raising}\label{sec: Motivic Cohomology and Level-raising}
We keep the notations of \S\ref{sec:basics}, but now $F$  will be a totally real field with $[F:\Q]=g$ {\it even} and $p$ be a prime which is \textit{inert} in $F$; we let $\fkp$ denote the unique prime of $F$ above $p$. We let $B$ be a quaternion algebra over $F$ whose ramification set does not intersect with $\Sigma_p\cup\Sigma_{\infty}$. We will apply the construction of the previous subsections for the case $X=\scrS_K(G)_{\bbF_{p^g}}$. Specifically we are interested in the cycle class map:
\begin{equation}\label{cycle-class}\text{Ch}^{g/2+1}(\sshg_{\bbF_{p^g}},1,k_\lambda)\rightarrow \text{H}^{g+1}(\sshg_{\bbF_{p^g}},k_\lambda(g/2+1)).\end{equation}

Recall we have the prime-to-$\rmR$ Hecke algebra $\bfT_{\rmR}$  and we let $\Pi$ be an irreducible cuspidal automorphic representation of $GL_2(F)$ defined over $\bfE$ as in \S\ref{sec:Ihara}. We assume $\fkp\notin \rmR$ and we write  $\fkm_{\rmR}\subset \bfT_{\rmR}$ (resp.   $\fkm\subset\bfT_{\rmR\cup\{\fkp\}}$) for the maximal ideal given by $(\phi_{\rmR}^{\Pi})^{-1}(\lambda)$ (resp.$(\phi_{\rmR\cup\{\fkp\}}^{\Pi})^{-1}(\lambda)$). Then  $\bfT_{\rmR\cup \{\fkp\}}$ acts on $\text{Ch}^{g/2+1}(\sshg_{\bbF_{p^g}},1,k_\lambda)$ and the map \ref{cycle-class} is equivariant for this action. Localizing at $\fkm$, we obtain $$\text{Ch}^{g/2+1}(\sshg_{\bbF_{p^g}},1,k_\lambda)_\fkm\rightarrow \text{H}^{g+1}(\sshg_{\bbF_{p^g}},k_\lambda(g/2+1))_\fkm$$ By the Hochschild-Serre spectral sequence  and Proposition \ref{prop: middle cohomology non-zero}, there is an isomorphism
$$\text{H}^{1}(\mathbb{F}_{p^{g}},\text{H}^g(\sshg_{\Fpbar},k_\lambda(g/2+1))_\fkm)\cong\text{H}^{g+1}(\sshg_{\bbF_{p^g}},k_\lambda(g/2+1))_\fkm,$$
where $\rmH^1(\bbF_{p^g},-)$ denotes the Galois cohomology. We let $$\mathrm{AJ}_\fkm:\text{Ch}^{g/2+1}(\sshg_{\bbF_{p^g}},1,k_\lambda)_\fkm\rightarrow \text{H}^{1}(\mathbb{F}_{p^{g}},\text{H}^g(\sshg_{\Fpbar},k_\lambda(g/2+1))_\fkm)$$ denote the induced map which we call the \textit{Abel--Jacobi map.}

We would like to show that $\mathrm{AJ}_\fkm$ is surjective modulo $\fkm$. In fact we will identify a certain subgroup   $$\text{Ch}_{\mathrm{lr}}^{g/2+1}(\sshg_{\bbF_{p^g}},1,k_\lambda)\subset \text{Ch}^{g/2+1}(\sshg_{\bbF_{p^g}},1,k_\lambda)$$ such that $\mathrm{AJ}_{\fkm}$ restricted to this subgroup  already surjects. Moreover as suggested by the notation, this subgroup is related to level-raising. The group $\text{Ch}_{\mathrm{lr}}^{g/2+1}(\sshg_{\bbF_{p^g}},1,k_\lambda)$ is defined using the geometry of Goren--Oort cycles as follows. We note that we may take $k_0=\bbF_p^{g}$ in this section.

Let $\fkB:=\fkB(\emptyset,g/2-1)$ as in \S\ref{sec:Goren--Oort dim 2}; thus $|\fkB|= {g\choose g/2-1}$. Let $\fka\in\fkB$, then by the discussion in \S\ref{sec:Goren--Oort dim 2}, the Goren--Oort cycle $Z_{\emptyset}(\fka)$ is a $(g/2-1)$-iterated $\bbP^1$-bundle over the Shimura surface $\scrS_K(G_{\emptyset_\fka,\rmT_{\fka}})_{\bbF_{p^{g}}}$ where $\rmT=\emptyset$. We have $|\Sigma_{\infty}-\emptyset_{\fka}|=2$ and we write $\Sigma_{\infty}-\emptyset_{\fka}=\{\tau_i,\tau_j\}$.

The Goren--Oort divisors $\scrS_K(G_{\emptyset_\fka,\rmT_{\fka}})_{\bbF_{p^{g}},\tau_i}$, $\scrS_K(G_{\emptyset_\fka,\rmT_{\fka}})_{\bbF_{p^{g}},\tau_j}$ are  $\bbP^1$-bundles over the discrete Shimura sets $\scrS_K(G_{\emptyset_{\fka,\tau_i},\rmT_{\fka,\tau_i}})_{\bbF_{p^g}}$ and $\scrS_K(G_{\emptyset_{\fka,\tau_j},\rmT_{\fka,\tau_j}})_{\bbF_{p^g}}$ via maps $\pi_{\tau_i}$, $\pi_{\tau_j}$  respectively. These Shimura sets are isomorphic over $\bbF_{p^g}$ and their $\Fpbar$-points are isomorphic to $$ G_{\Sigma_{\infty}}(\bbQ)\backslash G_{\Sigma_{\infty}}(\bbA_f)/K.$$ 

By Proposition \ref{prop: Goren--Oort 2 dim}, the intersection of these two divisors $\scrS_K(G_{\emptyset_\fka,\rmT_{\fka}})_{\bbF_{p^{g}},\{\tau_i,\tau_j\}}$ is identified with the discrete Shimura set $\scrS_{K_0(\fkp)}(G_{\emptyset_{\fka,\tau_i},\rmT_{\fka,\tau_i}})_{\bbF_{p^g}}$ and the induced diagram
$$\scrS_K(G_{\emptyset_{\fka,\tau_i},\rmT_{\fka,\tau_i}})_{\bbF_{p^g}}\xleftarrow{\pi_{\tau_i}}\scrS_{K_0(\fkp)}(G_{\emptyset_{\fka,\tau_i},\rmT_{\fka,\tau_i}})_{\bbF_{p^g}}\xrightarrow{\eta_{\tau_j}\circ\pi_{\tau_j}}\scrS_K(G_{\emptyset_{\fka,\tau_i},\rmT_{\fka,\tau_i}})_{\bbF_{p^g}}$$ is identified with the Hecke correspondence for  $\scrS_K(G_{\emptyset_{\fka,\tau_i},\rmT_{\fka,\tau_i}})_{\bbF_{p^g}}$. 

We may apply the construction of \S\ref{sec: dual-graph} in this setting to obtain a map $$\rmH_1(\fkG,k_\lambda)^{\Gal(\Fpbar/\bbF_{p^g})}\rightarrow\mathrm{Ch}^2(\scrS_K(G_{\emptyset_{\fka}})_{\bbF_{p^g}},1,k_\lambda),$$ where $\fkG$ is the dual graph associated to the configuration of $\bbP^1$'s in $\scrS_K(G_{\emptyset_\fka,\rmT_{\fka}})_{\Fpbar,\tau_i},$  $ \scrS_K(G_{\emptyset_\fka,\rmT_{\fka}})_{\Fpbar,\tau_j}$. We may describe $\rmH_1(\fkG,k_\lambda)$ more explicitly as follows. For any finite set $S$, we write $\Gamma(S,k_\lambda)$ for the abelian group of $k_\lambda$-valued functions on $X$. The maps $\pi_{\tau_i}$ and $\eta_{\tau_j}\circ\pi_{\tau_j}$ induce maps $$\pi_{\tau_i*},(\eta_{\tau_j}\circ\pi_{\tau_j})_*:\Gamma(\scrS_{K_0(\fkp)}(G_{\emptyset_{\fka,\tau_i},\rmT_{\fka,\tau_i}})_{\bbF_{p^g}}(\Fpbar),k_\lambda)\rightarrow \Gamma(\scrS_K(G_{\emptyset_{\fka,\tau_i},\rmT_{\fka,\tau_i}})_{\bbF_{p^g}}(\Fpbar),k_\lambda)$$Then $\rmH_1(\fkG,k_\lambda)$ is identified with  \begin{equation}
\label{eq: dual graph level raising subgroup}\rmK(\fka):=\ker\left((\pi_{\tau_i*},(\eta_{\tau_j}\circ\pi_{\tau_j})_*):\Gamma(\scrS_{K_0(\fkp)}(G_{\emptyset_{\fka,\tau_i},\rmT_{\fka,\tau_i}})_{\bbF_{p^g}}(\Fpbar),k_\lambda)\rightarrow \Gamma(\scrS_K(G_{\emptyset_{\fka,\tau_i},\rmT_{\fka,\tau_i}})_{\bbF_{p^g}}(\Fpbar),k_\lambda)^2\right).\end{equation}
We write $\Theta(\fka)$ for the map
\begin{equation}\label{eq: cycle class for Goren-Oort}\Theta({\fka}):\rmK(\fka)^{\Gal(\Fpbar/\bbF_{p^g})}\rightarrow \mathrm{Ch}^2(\scrS_K(G_{\emptyset_{\fka},\rmT_{\fka}})_{{\bbF_{p^g}}},1,k_\lambda)\end{equation}
which is equivariant for the action of $\bfT_{\rmR\cup\{\fkp\}}$.
\ignore{\begin{defn}
	We define the level raising subgroup $$\text{Ch}_{\mathrm{lr}}^2(\scrS_K(G_{\emptyset_\fka})_{\bbF_p^{g}},1,k_\lambda)\subset \text{Ch}^2(\scrS_K(G_{\emptyset_\fka})_{\bbF_p^{g}},1,k_\lambda)$$ to be the image of $\Theta_\fka$.
\end{defn}}Using (\ref{eq: mot coh fibration}) and (\ref{MotCoh-imm}) we obtain a map \begin{equation}\label{eq: Map between Chow group}\text{Ch}^2(\scrS_K(G_{\emptyset_\fka,\rmT_{\fka}})_{\bbF_p^{g}},1,k_\lambda)\xrightarrow{(\ref{eq: mot coh fibration})} \text{Ch}^2(Z_{\emptyset}(\fka),1,k_\lambda)\xrightarrow{(\ref{MotCoh-imm})}\text{Ch}^{g/2+1}(\scrS_K(G)_{\bbF_p^{g}},1,k_\lambda)\end{equation}
Composing (\ref{eq: Map between Chow group}) and (\ref{eq: cycle class for Goren-Oort}) and taking the direct sum over  $\fka\in\fkB$ we obtain a map:
\begin{equation}\label{eq: level raising Chow group}
\bigoplus_{\fka\in\fkB}\rmK(\fka)^{\Gal(\Fpbar/\bbF_{p^g})}\rightarrow\text{Ch}^{g/2+1}(\scrS_K(G)_{\bbF_{p^g}},1,k_\lambda)\end{equation}
\begin{defn}

	We define the level raising subgroup $$\text{Ch}_{\mathrm{lr}}^{g/2+1}(\scrS_K(G)_{\bbF_{p^g}},1,k_\lambda)\subset \text{Ch}^{g/2+1}(\scrS_K(G)_{\bbF_{p^g}},1,k_\lambda)$$ to be the image of (\ref{eq: level raising Chow group}).
\end{defn}

 Localizing (\ref{eq: level raising Chow group}) at $\fkm$ and composing with the map $\mathrm{AJ}_\fkm$ we obtain a map:

$$\mathrm{\Psi}_{\fkm}:\bigoplus_{\fka\in\fkB}\rmK(\fka)^{\Gal(\Fpbar/\bbF_{p^g})}_{\fkm}\rightarrow \text{H}^{1}(\mathbb{F}_{p^{g}},\text{H}^g(\sshg_{\Fpbar},k_\lambda(g/2+1))_\fkm)$$

\begin{defn}\label{def: level raising prime}A prime $p$ is a $\lambda$-level raising prime (with respect to $F,B,\Pi,K$) if the following four conditions are satisfied:
	
	(1) $p$ is inert in $F$ and coprime to $\rmR\cup\rmR_\lambda$.
	
	(2) $l\nmid \prod_{i=1}^g(p^{2gi}-1)$
	
	(3) $\phi_{\rmR}^\Pi(T_\fkp)^2\equiv (p^g+1)^2\mod\lambda$ and $\phi_{\rmR}^{\Pi}(S_\fkp)\equiv 1\mod\lambda$, where $\fkp$ is the unique prime of $F$ above $p$.
	
	(4) $l\nmid (p^{gn_{K_\infty}}-1)$ where $n_{K_\infty}$ is the order of $\sigma_p^g$ acting on $\scrS_K(G_{\emptyset_{\fka,\tau_i},\rmT_{\fka,\tau_i}})_{\bbF_{p^g}}(\Fpbar)$.
\end{defn}
\begin{rem}
	As in \cite[Remark 4.6]{LT}, it can be shown that there are infinitely many $\lambda$-level raising primes as long as there are rational primes inert in $F$ and $\lambda$ satisfies assumption \ref{ass: property of l}.
\end{rem}

\begin{lem}
Assume $p$ is a $\lambda$-level raising prime. Then	$\Gal(\Fpbar/\bbF_{p^g})$ acts trivially on $K(\fka)_\fkm$.
\end{lem}
\begin{proof}
	This follows from the definition of the Galois action in \S\ref{sec:Goren--Oort dim 2} and Definition \ref{def: level raising prime} (3).
\end{proof}
We will need to make the following additional assumption:
\begin{ass}
	\label{ass: Dim Jacquet Langlands}

	 $\rmH^g(\text{Sh}_K(G)_{\overline{\bbQ}},k_\lambda)/\fkm$ has dimension $2^g\text{dim}(\Pi_{B})^K$ over $k_\lambda$, where $\Pi_B$ denotes the Jacquet-Langlands transfer of $\Pi$ to $B\otimes_{\bbQ}\bbA$.\end{ass}\begin{thm}\label{thm:level-raising}Let $p$ be a $\lambda$-level raising prime and suppose that  Assumptions \ref{ass: property of l}  and \ref{ass: Dim Jacquet Langlands} are satisfied. Then the induced  map  \begin{equation}\label{eq: level raising map}
\bigoplus_{\fka\in\fkB}\rmK(\fka)/\fkm\rightarrow \rmH^{1}(\mathbb{F}_{p^{g}},\rmH^g(\sshg_{\Fpbar},k_\lambda(g/2+1))/\fkm)
	\end{equation}
is surjective.
\end{thm}

\begin{rem}In the case when $[F:\bbQ]=2$ and $B=GL_2(F)$, the surjectivity of $\Psi_{\fkm}$ implies level-raising for Hilbert modular forms. Indeed in this case there is a unique $\fka\in\fkB$ and  the description of $\rmK(\fka)$ in \ref{eq: dual graph level raising subgroup} shows that it is identified under the Jacquet--Langlands correspondence with the space of$\mod l$ Hilbert modular forms  of parallel weight 2 and level $K_0(\fkp)$ which are new at $\fkp$. In particular $\bfT_{\rmR\cup\{\fkp\}}$ acts on $K(\fka)$ via the $\fkp$-new quotient as in \cite{Ribet}, whereas it is well known $\bfT_{\rmR\cup\{\fkp\}}$ acts on $\text{H}^{1}(\mathbb{F}_{p^{g}},\text{H}^g(\sshg_{\Fpbar},k_\lambda(g/2+1)))$ via the $\fkp$-old quotient.	\end{rem}

Before embarking on the proof, we state an immediate corollary of Theorem \ref{thm:level-raising}, which is Theorem \ref{thm:intro main} of the introduction.

\begin{cor}
The map $\mathrm{AJ}_{\fkm}\mod\fkm$ restricted to $\mathrm{Ch}_{\mathrm{lr}}^{g/2+1}(\scrS_K(G)_{\bbF_{p^g}},1,k_\lambda)/\fkm$ is surjective.

\end{cor}

\subsection{Proof of Theorem \ref{thm:level-raising}} The rest of this section will be devoted to the proof of Theorem \ref{thm:level-raising}. For notational convenience, we will write $\scrS_K(G_{\emptyset_\fka})_{\bbF_{p^{g}}}$ for what was denoted  $\scrS_K(G_{\emptyset_\fka,\rmT_{\fka}})_{\bbF_{p^{g}}}$ in the previous subsection.

\ignore{Let $\fka\in \fkB(\emptyset,g/2+1)$ and we let $I_{\fka}:=(G_{\Sigma_{\infty}}(\bbQ)\backslash G_{\Sigma_{\infty}}(\bbA_f)/K)^2$ which we identify with the components of the two Goren--Oort divisors $\scrS_K(G_{\emptyset_\fka,\rmT_{\fka}})_{\Fpbar,\tau_i}$, $\scrS_K(G_{\emptyset_\fka,\rmT_{\fka}})_{\Fpbar,\tau_j}$. We let $S_i, i\in I$ be the corresponding component in  $\scrS_K(G_{\emptyset_\fka})_{\Fpbar}$ which are isomorphic to $\bbP^1$.
\begin{prop}
	There exists an isomorphism $$\mathrm{Ch}_{\mathrm{lr}}^{2}(\scrS_K(G_{\emptyset_{\fka}}),1,k_\lambda)_\fkm\cong \mathrm{Ch}_{\mathrm{I_\fka}}^{2}(\scrS_K(G_{\emptyset_{\fka}}),1,k_\lambda)_\fkm$$
\end{prop}
\begin{proof}
	By definition $\mathrm{Ch}_{\mathrm{lr}}^{2}(\scrS_K(G_{\emptyset_{\fka}}),1,k_\lambda)$ is the image of $\rmH_1(\fkG,k_\lambda)$ and $\mathrm{Ch}_{I_{\fka}}^{2}(\scrS_K(G_{\emptyset_{\fka}}),1,k_\lambda)$ is the image of $\left(\bigoplus_{i\in I_{\fka}}\bbF_{p^{g}}^\times\otimes_{\bbZ}k_\lambda\right)\oplus\rmH_1(\fkG,k_\lambda)$ by the surjectivity of (\ref{eq: level raising group surjects onto Chow group}). It easily seen that the constant functions $\bbF_p^{g}$ are Eisenstein, hence $\left(\bigoplus_{i\in I_{\fka}}\bbF_{p^{g}}^\times\otimes_{\bbZ}k_\lambda\right)_{\fkm}=0$. The result follows.
\end{proof}

It therefore suffices to show the surjectivity of the map $$\left(\bigoplus_{\fka\in\fkB}\mathrm{Ch}_{\mathrm{I_\fka}}^{2}(\scrS_K(G_{\emptyset_{\fka}})_{\bbF_{p^{g}}},1,k_\lambda)\right)/\fkm\rightarrow\text{H}^{1}(\mathbb{F}_{p^{g}},\text{H}^g(\sshg_{\Fpbar},k_\lambda(g/2+1))/\fkm)$$
}

Recall for $\fka\in\fkB(\emptyset,g/2+1)$ we have the correspondence:
$$\scrS_K(G_{\emptyset_\fka})_{\bbF_{p^{g}}}\xleftarrow{\pi_{\fka}} Z_{\emptyset}(\fka)\hookrightarrow \scrS_K(G)_{\bbF_{p^{g}}}$$
where $\pi_{\fka}$ is a $(g/2-1)$-iterated $\bbP^1$-bundle. This induces a Gysin map:
$$\mathrm{Gys}(\fka):\rmH^*(\scrS_K(G_{\emptyset_\fka})_{\Fpbar},k_\lambda(2))\rightarrow \rmH^{*+g-2}(\scrS_K(G)_{\Fpbar},k_\lambda(g/2+1)).$$
Then the map  $\bigoplus_{\fka\in\fkB}\rmK(\fka)_\fkm\rightarrow \rmH^1(\bbF_{p^{g}},\mathrm{H}^{g}(\scrS_K(G)_{\Fpbar},k_\lambda(g/2+1))_\fkm)$ factors as

\[\xymatrix{\mathrm{Ch}^{g/2+1}_{\mathrm{lr}}(\scrS_K(G_{\emptyset_\fka})_{\bbF_{p^{g}}},1,k_\lambda)_\fkm \ar[r]& \rmH^1(\bbF_{p^{g}},\mathrm{H}^{g}(\scrS_K(G)_{\Fpbar},k_\lambda(g/2+1))_\fkm)\\
\bigoplus_{\fka\in\fkB}\rmK(\fka)_\fkm \ar[r] \ar[u]& \bigoplus_{\fka\in\fkB}\rmH^1(\bbF_{p^{g}},\mathrm{H}^{2}(\scrS_K(G_{\emptyset_{\fka}})_{\Fpbar},k_\lambda(2))_\fkm)\ar[u]_{\sum_{\fka\in\fkB}\mathrm{Gys}(\fka)_\fkm}}\]
where the bottom map is the composition of the direct sums  of the maps $\Theta(\fka)_\fkm$ and the Abel--Jacobi maps  $$\mathrm{AJ}(\fka)_\fkm:\mathrm{Ch}^{2}(\scrS_K(G_{\emptyset_\fka})_{\Fpbar},1,k_\lambda)_\fkm\rightarrow \rmH^1(\bbF_{p^{g}},\mathrm{H}^{2}(\scrS_K(G_{\emptyset_{\fka}})_{\Fpbar},k_\lambda(2))_\fkm).$$ 
We write $\Psi(\fka)_\fkm:K(\fka)_\fkm\rightarrow\rmH^1(\bbF_{p^{g}},\mathrm{H}^{2}(\scrS_K(G_{\emptyset_{\fka}}),k_\lambda(2))_\fkm)$ for this composition.
\begin{prop}\label{prop:Gysin map surjective}
The map $\sum_{\fka\in\fkB}\mathrm{Gys}(\fka)\mod\fkm$ is surjective.
\end{prop}
\begin{proof}
	The proof is the same as in \cite[Proof of Theorem 4.7]{LT}. The idea is to show that $\sum_{\fka\in\fkB}\mathrm{Gys}(\fka)\mod \fkm$ is injective; the surjectivity then follows by a dimension count.
\end{proof}
\begin{rem}The proof of  Proposition \ref{prop:Gysin map surjective} uses  Definition \ref{def: level raising prime} (2), Assumption \ref{ass: property of l} (4) and the freeness result in Proposition \ref{prop: middle cohomology non-zero} in an essential way.
	\end{rem}

Thus it suffices to prove the following Proposition.

\begin{prop}\label{prop: AJ surjective for surface}
	For each $\fka\in\fkB$, the map $$\Psi(\fka)_\fkm:\rmK(\fka)_\fkm\rightarrow\rmH^1(\bbF_{p^{g}},\mathrm{H}^{2}(\scrS_K(G_{\emptyset_{\fka}}),k_\lambda(2))_\fkm) $$ is surjective.
\end{prop}

 Fix $\fka\in\fkB$. We let $\Sigma_\infty-\rmS_{\infty}=\{\tau_i,\tau_j\}$. To ease notation, we write $X=\scrS_K(G_{\emptyset_{\fka}})_{\bbF_{p^{g}}}$ and we write $Z_1=\scrS_K(G_{\emptyset_{\fka}})_{\bbF_{p^{g}},\tau_i}$, $Z_2=\scrS_K(G_{\emptyset_\fka})_{\bbF_{p^{g}},\tau_j}$ for the Goren--Oort divisors. We let $Z_{12}$ denote the intersection $Z_1\cap Z_2$ and we write $X^{\mathrm{ord}}$ for the complement of $Z_1\cup Z_2$ in $X$.

The first observation is that the  map $\Psi(\fka)_\fkm$ factors through a certain cohomology group with supports.

\begin{prop}
	The map $\Psi(\fka)_\fkm$ factors through $\rmH^3_{Z_1\cup Z_2}(X,k_\lambda(2))_\fkm$. Moreover the induced map $$\rmK(\fka)_\fkm\rightarrow\rmH^3_{Z_1\cup Z_2}(X,k_\lambda(2))_\fkm$$ is surjective.
\end{prop}
\begin{proof}
	The factoring property follows from the definition of the cycle class map. Indeed we have the following diagram with exact columns and rows
	\[\xymatrix{&  &0 & 0 \\0\ar[r] & \bfL \ar[r]& \rmH^1(Z_1\cap Z_2-Z_{12},k_\lambda(1)) \ar[r]^{\ \ \ \ \ \ \ c} \ar[u] & \rmH^0(Z_{12},k_\lambda) \ar[u]\\
		0\ar[r] & \bfK\ar[r] \ar[u]_a &\rmH^0(Z_1\cap Z_2-Z_{12},\mathbb{G}_m)\otimes_{\bbZ} k_\lambda \ar[r]^{\ \ \ \ \ \ \ \ \ \text{div}\otimes_\Z k_\lambda} \ar[u]_b &  \rmH^0(Z_{12},k_\lambda) \ar@{=}[u]
	.}\]

Here $\rmH^0(Z_1\cap Z_2-Z_{12},\mathbb{G}_m)$ is identified with the group of rational functions on $Z_1\cap Z_2$ with zeros and poles only at the points in $Z_{12}$, and the map $\text{div}\otimes_\Z k_\lambda$ is induced from the map taking such a function to its divisor. The groups $\bfL$ and $\bfK$ are by definition the kernels of the maps on the right. The map $b$ comes from the Kummer sequence and its surjectivity follows from the fact that $Z_1\cup Z_2-Z_{12}$ is a union of open subsets in $\bbA^1_{\bbF_{p^h}/\bbF_{p^g}}$ for $g|h$ hence all have trivial Picard group. Here we write $\bbA^1_{\bbF_{p^h}/\bbF_{p^g}}$ for the affine line $\bbA^1$ over $\bbF_{p^h}$ considered as a $\bbF_{p^g}$ scheme. The map $\mathrm{div}\otimes_\Z k_\lambda$ factors through $b$, hence we obtain the map $c$.

By purity, we may identify the top row of the exact sequence with the exact sequence of the triple $(X,Z_{12},Z_1\cup Z_2)$:
\begin{align*} \rmH^3_{Z_{12}}(X,k_\lambda(2))=0&\rightarrow \rmH^3_{Z_1\cup Z_2}(X,k_\lambda(2))\\&\rightarrow \rmH^3_{Z_1\cup Z_2-Z_{12}}(X-Z_{12},k_\lambda(2))\rightarrow \rmH^4_{Z_{12}}(X,k_\lambda(2)).
\end{align*}
It is easy to see from the definition that the group $\text{Ch}^2_{Z_1\cup Z_2}(X,1,k_\lambda)$ is a quotient of the group $\bfK$. Then the cycle class map is induced by $$\bfK\rightarrow \bfL\cong \rmH^3_{Z_1\cup Z_2}(X,k_\lambda(2))\rightarrow \rmH^3(X,k_\lambda(2)).$$
To deduce the moreover part, we note that the map $$\rmK(\fka)_\fkm\rightarrow \text{Ch}^2_{Z_1\cup Z_2}(X,1,k_\lambda)_\fkm$$ is surjective. Indeed  by (\ref{eq: level raising group surjects onto Chow group}) the map $$\left(\bigoplus_{S\subset Z_1\cup Z_2 }k_S^\times\otimes_{\bbZ}k_\lambda\right)\oplus \rmK(\fka)^{\Gal(\overline{\bbF}_p/\bbF_{p^g})}\rightarrow \text{Ch}^2_{Z_1\cup Z_2}(X,1,k_\lambda)$$ is surjective. But  $k_S^\times\otimes_{\bbZ}k_\lambda=0$ for all $S$ by Definition \ref{def: level raising prime} (4). It therefore suffices to show that the map $a$ is surjective. We write $$\bfA:=\text{Im}\left(\rmH^0(Z_1\cap Z_2-Z_{12},\mathbb{G}_m)\rightarrow \rmH^0(Z_{12},k_\lambda)\right)$$ $$\bfB:=\text{Im}\left(\rmH^1(Z_1\cap Z_2-Z_{12},k_\lambda(1))\rightarrow \rmH^0(Z_{12},k_\lambda)\right).$$
The map $\bfA\rightarrow \bfB$ is injective, hence by the snake Lemma, $a$ is surjective. 
\end{proof}

Thus in order to prove Proposition \ref{prop: AJ surjective for surface}, and hence Theorem \ref{thm:level-raising}, it suffices to show that the map $$\Phi'_\fkm:\rmH^3_{Z_1\cup Z_2}(X,k_\lambda(2))_\fkm\rightarrow \rmH^3(X,k_\lambda(2))_\fkm$$ is surjective.

By the Hochschild-Serre spectral sequence, this map fits into the following diagram with exact rows:
\[\xymatrix{0 \ar[r] & \rmH^1({\bbF_{p^{g}}},\rmH^2_{Z_1\cup Z_2}(X_{\Fpbar},k_\lambda(2))_{\fkm}) \ar[r]\ar[d] & \rmH_{Z_1\cup Z_2}^3(X,k_\lambda(2))_{\fkm}\ar[r]\ar[d]_{\Phi'_\fkm}& \rmH^0(\bbF_{p^{g}},\rmH^3_{Z_1\cup  Z_2}(X_{\Fpbar},k_\lambda(2))_{\fkm}) \ar[r]\ar[d] & 0\\0 \ar[r]& \rmH^1({\bbF_{p^{g}}},\rmH^2(X_{\Fpbar},k_\lambda(2))_{\fkm}) \ar[r] & \rmH^3(X,k_\lambda(2))_{\fkm}\ar[r]& \rmH^0(\bbF_{p^{g}},\rmH^3(X_{\Fpbar},k_\lambda(2))_{\fkm}) \ar[r] & 0
}\]

Since $\rmH^3(X_{\Fpbar},k_\lambda(2))_{\fkm}=0$ by Proposition \ref{prop: middle cohomology non-zero}, we have an isomorphism $$\rmH^3(X,k_\lambda(2))_\fkm\cong \rmH^1(\bbF_{p^{g}},\rmH^2(X_{\Fpbar},k_\lambda(2))_\fkm).$$ The diagram above then induces a map  $$\Phi_\fkm: \rmH^0(\bbF_{p^g},\rmH^3_{Z_1\cup  Z_2}(X_{\Fpbar},k_\lambda(2))_{\fkm})\rightarrow \rmH^1(\bbF_{p^g},\rmH^2(X_{\Fpbar},k_\lambda(2))_\fkm)/\rmH^1(\bbF_{p^g},\rmH^2_{Z_1\cup Z_2}(X_{\Fpbar},k_\lambda(2))_\fkm)$$
and  surjectivity of $\Phi_\fkm'$ is equivalent to the surjectivity of $\Phi_\fkm$. It therefore suffices to prove $\Phi_{\fkm}$ is surjective.

We may dualize the above to obtain a map $$\Phi_\fkm^*:\left(\rmH^1(\bbF_{p^g},\rmH^2(X_{\Fpbar},k_\lambda(2))_\fkm)/\rmH^1(\bbF_{p^g},\rmH_{Z_1\cup Z_2}^2(X_{\Fpbar},k_\lambda(2))_\fkm)\right)^*\rightarrow \rmH^0(\bbF_{p^g},\rmH^3_{Z_1\cup  Z_2}(X_{\Fpbar},k_\lambda(2))_{\fkm})^*.$$
Consider the exact sequence of cohomology
$$\rmH^1(X_{\Fpbar},k_\lambda)\rightarrow\rmH^1(Z_{1\Fpbar}\cup Z_{2\Fpbar},k_\lambda)\rightarrow \rmH^2_c(X^{\mathrm{ord}}_{\Fpbar},k_\lambda)\rightarrow \rmH^2(X_{\Fpbar},k_\lambda)\rightarrow \rmH^2(Z_{1\Fpbar}\cup Z_{2\Fpbar},k_\lambda)$$
arising from the triple $Z_1\cup Z_2\xrightarrow{i} X\xleftarrow{j}X^{\mathrm{ord}}$. When $X$ is the special fiber of the Hilbert modular surface (hence non-compact), we abuse notation and write $\rmH^2_c(X^{\mathrm{ord}}_{\Fpbar},k_\lambda)$ for $\rmH^2(X_{\Fpbar},j_!k_\lambda)$. Upon localizing at $\fkm$, $\rmH^1(X_{\Fpbar},k_\lambda)_\fkm=0$, and we  obtain the following boundary map for the long exact sequence of Galois cohomology:
\begin{equation}
\label{eq:dual of level raising map}
\ker\left(\rmH^0(\bbF_{p^g},\rmH^2(X_{\Fpbar},k_\lambda)_\fkm)\rightarrow \rmH^0(\bbF_{p^g},\rmH^2(Z_{1\Fpbar}\cup Z_{2\Fpbar},k_\lambda)_\fkm)\right)\rightarrow \rmH^1(\bbF_{p^g},\rmH^1(Z_{1\Fpbar}\cup Z_{2\Fpbar},k_\lambda)_\fkm).\end{equation}
 By Poincar\'e duality and the duality of Galois cohomology over finite fields, (\ref{eq:dual of level raising map}) is identified with $\Phi_{\fkm}^*$. We mention that in the case of Hilbert modular surfaces, we need to use the canonical isomorphism $$\rmH^2(X_{\Fpbar},k_\lambda)_\fkm\cong \rmH_c^2(X_{\Fpbar},k_\lambda)_\fkm,$$ which follows from \cite[Corollary 4.6]{LS} and the fact that the cuspidal part of the cohomology of Hilbert modular surfaces is the same as the cuspidal part of the cohomology with compact support, see \cite[Theorem 2.3]{Dim}.

Let $X_0(\fkp)$ be the special fiber $\scrS_{K_0(\fkp)}(G_{\rmS,\rmT})_{\bbF_{p^{g}}}$ for the Shimura variety with Iwahori level structure at $p$ constructed in Appendix \ref{sec:Appendix}. We claim $\Phi_\fkm^*$ can be related to a certain map relating the geometry of $X$ and $X_0(\fkp)$.

By Corollary \ref{cor:Iwahori level structure} there is a decomposition $$X_0(\fkp)_{\Fpbar}=\mathcal{X}_{1\Fpbar}\cup\mathcal{X}_{2\Fpbar}\cup\mathcal{X}_{3\Fpbar}$$ where $\calX_{1\Fpbar}$ and $\calX_{2\Fpbar}$ are the two copies of $X_{\Fpbar}$ corresponding to the essential Frobenius and Verschiebung isogenies respectively and $\calX_{3\Fpbar}$ is the ``supersingular locus.'' For $k,l,m\in\{1,2,3\}$ distinct, we write $$i_k:\calX_{k\Fpbar}\rightarrow X_0(\fkp)_{\Fpbar}$$ $$i_{kl}:\calX_{k\Fpbar}\cap\calX_{l\Fpbar}\rightarrow X_0(\fkp)_{\Fpbar}$$
$$i_{klm}:\calX_{k\Fpbar}\cap\calX_{l\Fpbar}\cap\calX_{m\Fpbar}\rightarrow X_0(\fkp)_{\Fpbar}$$
for the closed immersions. Then we have an exact sequence of sheaves on $X_0(\fkp)_{\Fpbar}$:
$$0\rightarrow k_\lambda\xrightarrow{f} i_{1*}k_\lambda\oplus i_{2*}k_\lambda\oplus i_{3*}k_\lambda\rightarrow i_{23*}k_\lambda\oplus i_{13*}k_\lambda\oplus i_{23*}k_\lambda\xrightarrow{g} i_{123*}k_\lambda\rightarrow 0.$$
Here the maps are induced by (pushing forward) the usual unit maps of the adjunction. We define $$C:=\text{coker}(f)=\ker(g).$$ Then the above exact sequence breaks up into two short exact sequences of sheaves on $X_0(\fkp)_{\Fpbar}$.
\begin{equation}\label{eq1}0\rightarrow k_\lambda\xrightarrow{f} i_{1*}k_\lambda\oplus i_{2*}k_\lambda\oplus i_{3*}k_\lambda\rightarrow C\rightarrow 0
\end{equation}
\begin{equation}\label{eq2}0\rightarrow C\rightarrow i_{23*}k_\lambda\oplus i_{13*}k_\lambda\oplus i_{23*}k_\lambda\xrightarrow{g} i_{123*}k_\lambda\rightarrow 0
\end{equation}
 Taking cohomology of the sequence (\ref{eq1}), we obtain:
\[\xymatrixcolsep{4pc}\xymatrix{0 \ar[r]& \bfR \ar[r] & \rmH^2(X_0(\fkp)_{\Fpbar},k_\lambda)_\fkm\ar[r]^-{ (i_{1}^*,i_{2}^*,i_{3}^*)}& \rmH^2(X_{\Fpbar},k_\lambda)_\fkm^2\oplus\rmH^2(\calX_{3\Fpbar},k_\lambda)_\fkm}\]
where $\bfR\cong \rmH^1(X_0(\fkp)_{\Fpbar},C)_\fkm/\rmH^1(\calX_{3\Fpbar},k_\lambda)_\fkm$. 
Now consider the map $$\pi_{1}^*+\pi_2^*:\rmH^2(X_{\Fpbar},k_\lambda)_\fkm^2\rightarrow \rmH^2(X_0(\fkp)_{\Fpbar},k_\lambda)_\fkm$$
induced by the degeneracy maps $\pi_1,\pi_2:X_0(\fkp)\rightarrow X$. 
We consider the kernel $\bfS$ of the composition
$$\rmH^2(X_0(\fkp)_{\Fpbar},k_\lambda)^2\xrightarrow{\pi_1^*+\pi_2^*} \rmH^2(X_{\Fpbar},k_\lambda)_\fkm\xrightarrow{(i_{1}^*,i_{2}^*,i_{3}^*)} \rmH^2(X_0(\fkp)_{\Fpbar},k_\lambda)^2_\fkm\oplus\rmH^2(\calX_{3\Fpbar},k_\lambda)_\fkm.$$ 
We obtain a map $\bfS\rightarrow \bfR$ which fits into the diagram:
\[\xymatrixcolsep{4pc}\xymatrix{0 \ar[r]& \bfR \ar[r] & \rmH^2(X_0(\fkp)_{\Fpbar},k_\lambda)_\fkm\ar[r]^-{ (i_{1}^*,i_{2}^*,i_{3}^*)}& \rmH^2(X_{\Fpbar},k_\lambda)_\fkm^2\oplus\rmH^2(\calX_{3\Fpbar},k_\lambda)_\fkm\\
	& \bfS\ar[u]^{\Psi_\fkm}\ar[r]& \rmH^2(X_{\Fpbar},k_\lambda)^2_\fkm\ar[u]_{\pi_1^*+\pi_2^*}}\]

By Theorem \ref{thm: Ihara's Lemma} and proper base change, the map $\Psi_\fkm$ is injective when $B$ is not totally totally split. For Hilbert modular surfaces, it is also injective by  \cite[Theorem 3.1]{Dim} and \cite[Corollary 4.6]{LS}. Therefore in order to prove Proposition \ref{prop: AJ surjective for surface}, it suffices to prove the following Proposition.

\begin{prop} The map $\Psi_\fkm$ can be identified with the map $\Phi^*_\fkm$.
	\end{prop}
\begin{proof}
	We first identify the groups $\bfS$ and $\bfR$ with the corresponding domain and codomain for the map $\Phi^*_\fkm$. 
	
	First note that the composition $$\rmH^2(X_{\Fpbar},k_\lambda)^2\xrightarrow{\pi_1^*+\pi_2^*} \rmH^2(X_0(\fkp)_{\Fpbar},k_\lambda)_\fkm\xrightarrow{(i_{1}^*,i_{2}^*)} \rmH^2(X_{\Fpbar},k_\lambda)^2_\fkm$$ induces the endomorphism
	\begin{equation}\label{eq: end of Cohomology}\left(\begin{matrix}1 &\text{Fr}'\\ \text{Fr}' S_\fkp^{-1} & 1\end{matrix}\right),\end{equation} where $\mathrm{Fr}'$ is the essential Frobenius as defined in \ref{sec: Appendix essentail Frobenius}.  By Definition \ref{def: level raising prime} (3), this is the same as the endomorphism $$\left(\begin{matrix}1 &\text{Fr}'\\ \text{Fr}'& 1\end{matrix}\right).$$ By Proposition \ref{prop: ess Frob squared} and Definition \ref{def: level raising prime} (3), $\mathrm{Fr}'^2=\mathrm{Fr}_\fkp$ the $p^g$-Frobenius, thus we may identify the kernel of (\ref{eq: end of Cohomology}) with the image of the injective map 
	$$\rmH^0(\bbF_{p^g},\rmH^2(X_{\Fpbar},k_\lambda))_\fkm\cong\rmH^2(X_{\Fpbar},k_\lambda)^{\mathrm{Fr}_\fkp=1}_\fkm\xrightarrow{(-\text{Fr}',\text{id})}\rmH^2(X_{\Fpbar},k_\lambda)_\fkm^2.$$
	
	Therefore we have an identification
	$$\bfS\cong\ker\left(\rmH^0(\bbF_{p^g},\rmH^2(X_{\Fpbar},k_\lambda))_\fkm\rightarrow \rmH^2(X_0(\fkp)_{\Fpbar},k_\lambda)_\fkm\xrightarrow{i_{3}^*}\rmH^2(\calX_{3\Fpbar},k_\lambda)_\fkm\right).$$
This map factors as $$\rmH^0(\bbF_{p^g},\rmH^2(X_{\Fpbar},k_\lambda))_\fkm\rightarrow \rmH^2(Z_{1\Fpbar}\cup Z_{2\Fpbar},k_\lambda)_\fkm\hookrightarrow \rmH^2(\calX_{3\Fpbar},k_\lambda)_\fkm$$	
Here the second map is an injection by \cite[Lemma 3.26 (2)]{Liu1} since $\calX_{3\Fpbar}\cong \mathcal{Z}_{1\Fpbar}\cup\mathcal{Z}_{2\Fpbar}$ where $\mathcal{Z}_{1\Fpbar}$ and $ \mathcal{Z}_{2\Fpbar}$ are $\bbP^1$-fibrations over $Z_{1\Fpbar}$ and $Z_{2\Fpbar}$  respectively and the intersection  $\mathcal{Z}_{1\Fpbar}\cap\mathcal{Z}_{2\Fpbar}$ is $0$-dimensional. Therefore we may identify $\bfS$ with the domain in the definition of $\Phi_\fkm^*$,

Consider the following commutative
diagram of $\text{Gal}(\Fpbar/\bbF_{p^g})$-modules with exact rows:

\begin{equation}\xymatrix{& &\rmH^2_c(X^{\text{ord}}_{\Fpbar},k_\lambda)_\fkm\ar[d]^{\pi_2^*-\pi_1^*\text{Fr}'}\ar[r] &\rmH^2(X_{\Fpbar},k_\lambda)_\fkm \ar[d]^{\pi_2^*-\pi_1^*\text{Fr}'} &\\
	0 \ar[r] & \rmH^1(\calX_{3\Fpbar},k_\lambda)_\fkm \ar[r]^-{\alpha} \ar[d]^{\Delta}&\rmH^2_c(X_{\Fpbar}^{\text{ord}},k_\lambda)_\fkm^2 \ar[r]\ar@{=}[d]&\rmH^2(X_0(\fkp)_{\Fpbar},k_\lambda)_\fkm \ar[r]^-{i_{3}^*}\ar[d]^{(i_{1}^*,i_{2}^*)}& \rmH^2(\calX_{3\Fpbar},k_\lambda)_\fkm\ar[d]^\Delta\\
0 \ar[r] & \rmH^1(Z_{1\Fpbar}\cup Z_{2\Fpbar},k_\lambda)_\fkm^2 \ar[r]^-{\gamma\oplus\gamma} &\rmH^2_c(X_{\Fpbar}^{\text{ord}},k_\lambda)_\fkm^2 \ar[r]&\rmH^2(X_{\Fpbar},k_\lambda)_\fkm^2 \ar[r]^{\delta\oplus\delta}& \rmH^2(Z_{1\Fpbar}\cup Z_{2\Fpbar},k_\lambda)_\fkm^2}.\end{equation}
Here the middle row is the exact sequence of cohomology arising from the triple $$\calX_{3\Fpbar}\xrightarrow{i_3} X_0(\fkp)_{\Fpbar}\leftarrow X_0(\fkp)_{\Fpbar}-\calX_{3\Fpbar}\cong X^{\mathrm{ord}}\sqcup X^{\mathrm{ord}}.$$ The maps  $\Delta$  and $\Delta'$ are induced by the inclusion maps $i_{12}, i_{13}$ restricted to $Z_{1\Fpbar}\cup Z_{2\Fpbar}$.

 Now recall $\Phi_\fkm^*$ is identified with the coboundary map after taking cohomology of the short exact sequence:
\begin{equation}\label{coboundary}0\rightarrow \rmH^1(Z_{1\Fpbar}\cup Z_{2\Fpbar},k_\lambda)_\fkm \xrightarrow{\gamma} \rmH^2_c(X_{\Fpbar}^{\text{ord}},k_\lambda)_\fkm\rightarrow\text{coker}(\gamma)\rightarrow 0\end{equation}
To identify $\bfR$ with the codomain $\rmH^1(\bbF_{p^g},\rmH^1(Z_{1\Fpbar}\cup Z_{2\Fpbar},k_\lambda))_\fkm,$ note that by definition we have a canonical identification $$\bfR\cong \ker\left(\text{coker}(\alpha)\xrightarrow{(i_{1*},i_{2*})}\text{coker}(\gamma\oplus\gamma)\right).$$ We then use the snake lemma applied to  the diagram:
\[\xymatrix{0 \ar[r] & \rmH^1(\calX_{3\Fpbar},k_\lambda)_\fkm \ar[r]^-{\alpha} \ar[d]^\Delta&\rmH^2_c(X_{\Fpbar}^{\text{ord}},k_\lambda)_\fkm^2 \ar[r]\ar@{=}[d]&\text{coker}(\alpha) \ar[r]\ar[d]& 0\\
0 \ar[r] & \rmH^1(Z_{1\Fpbar}\cup Z_{2\Fpbar},k_\lambda)_\fkm^2 \ar[r]^-{\gamma\oplus\gamma} &\rmH^2_c(X_{\Fpbar}^{\text{ord}},k_\lambda)_\fkm^2 \ar[r]&\text{coker}(\gamma\oplus\gamma) \ar[r] & 0}\] to deduce an isomorphism 
$$\bfR\cong \rmH^1(Z_{1\Fpbar}\cup Z_{2\Fpbar},k_\lambda)_\fkm\cong \rmH^1(\bbF_{p^g},\rmH^1(Z_{1\Fpbar}\cup Z_{2\Fpbar},k_\lambda)_\fkm)$$ where we have identified $\text{coker}(\Delta)$ with $\rmH^1(Z_{1\Fpbar}\cup Z_{2\Fpbar},k_\lambda)_\fkm$ via projection onto the second factor and the second isomorphism follows from Definition \ref{def: level raising prime} (3) which implies $\Gal(\Fpbar/\bbF_{p^g})$ acts trivially on $\rmH^1(Z_{1\Fpbar}\cup Z_{2\Fpbar},k_\lambda)_\fkm$. Using these identifications, we may now compare the maps $\Phi_\fkm^*$ and $\Psi_\fkm$.

Let $x\in\bfS\subset\rmH^0(\bbF_{p^g},\rmH^2(X_{\Fpbar},k_\lambda))_\fkm$. We will compute $\Psi_\fkm(x)\in\bfR$ considered as an element of $\text{coker}(\Delta)\cong\rmH^1(Z_{1\Fpbar}\cup Z_{2\Fpbar},k_\lambda)_\fkm $ and show that it coincides with the image under the coboundary map of (\ref{coboundary}). 

To compute $\Psi_\fkm(x)$ we let
 $\tilde{x}\in \rmH^2_c(X^{\text{ord}}_{\Fpbar},k_\lambda)_\fkm$ be a lift of $x$, which exists since $x$ maps to $0$ in $\rmH^2(Z_{1\Fpbar}\cup Z_{2\Fpbar},k_\lambda)_\fkm$. Moreover we may assume $\tilde{x}$ is fixed by $S_\fkp$ since $S_\fkp$ acts semi-simply on $\rmH^2_c(X^{\text{ord}}_{\Fpbar},k_\lambda)_\fkm$. Then $\pi_2^*(\tilde{x})-\pi_1^*\text{Fr}'(\tilde{x})\in \rmH^2_c(X_{\Fpbar}^{\text{ord}},k_\lambda)^2_\fkm$ is an element lifting $\pi_2^*(x)-\pi_1^*\text{Fr}'(x)\in\ker(i_{1}^*,i_{2}^*)$, hence comes from an element $$c(\tilde{x})\in \rmH^1(Z_{1\Fpbar}\cup Z_{2\Fpbar},k_\lambda)_\fkm^2.$$By definition, $\Psi_\fkm(x)$ is the image of $\pi_2^*(x)-\pi_1^*\mathrm{Fr}'(x)\in\ker(i_1^*,i_2^*)$ in  $\mathrm{coker}(\Delta)$. Therefore the image of $c(\tilde{x})$ is identified with $\Psi_{\fkm}(x)$.
 
 However we also have $$\pi_2^*(\tilde{x})-\pi_1^*\text{Fr}'(\tilde{x})=(S_{\fkp}^{-1}\text{Fr}'(\tilde{x}),\tilde{x})-(\text{Fr}'(\tilde{x}),\text{Fr}_\fkp(\tilde{x}))\cong (0,(1-\text{Fr}_\fkp)(\tilde{x})).$$
 By definition of the coboundary map, the image of this element in $\text{coker}(\Delta)\cong \rmH^1(Z_{1\Fpbar}\cup Z_{2\Fpbar},\Lambda)_\fkm$ is equal to $\Phi_\fkm ^*(x)$.
\end{proof}

We finish by stating the following corollary.

\begin{cor}
Let $p$ be a $\lambda$-level raising prime and suppose that  Assumption \ref{ass: property of l} is satisfied. Then the cycle  class map $$\mathrm{H}_\calM^{3}(\scrS_K(G_{\emptyset_\fka})_{\bbF_{p^g}},k_\lambda(2))_\fkm\rightarrow \mathrm{H}^{3}(\scrS_K(G_{\emptyset_{\fka}})_{\bbF_{p^g}},k_\lambda(2))_\fkm$$ is an isomorphism.
\end{cor}
\begin{proof}
	The surjectivity follows from Proposition \ref{prop: AJ surjective for surface}. The injectivity is  \cite[Theorem 6.17]{Voe1}.
\end{proof}

\appendix
\section{Bad reduction of quaternionic Shimura surfaces}

	In this section we give a global description of the special fiber of certain quaternionic Shimura surfaces with {\it Iwahori level structure at $p$}. These are associated to the group $G_{\rmS}$ with $|\Sigma_\infty-\rmS_\infty|=2$. The idea is that only the two places $\Sigma_\infty-\rmS_\infty$ should contribute to the geometry, and thus the structural results we obtain are completely analogous to the result of Stamm \cite{Stamm} in the case of the Hilbert modular surface. Indeed the proofs are also completely analogous, with some extra technical difficulties since we must transfer the results from unitary Shimura varieties. The key difference is that we must replace the notion of the usual notion of Frobenius and Verschiebung in the case of Hilbert modular surfaces with the notion of essential Frobenius and Verschiebung. For this reason, we will refer to \cite{Stamm} for the some of the computations, because they are exactly the same.
\subsection{Moduli interpretation and local models}\label{sec:Appendix}
	 We keep the notation of \S\ref{sec:basics}. In particular $B/F$ is a totally indefinite quaternion algebra and $\rmS\subset \Sigma_p\cup\Sigma_{\infty}$ a set of even cardinality. In fact in this section we will make the following further assumptions:
	
	(1) $g:=[F:\bbQ]$ is even and the prime $p$ is inert in $F$. We write $\fkp$ for the unique prime of $F$ above $p$. 
	
	(2) $\fkp\notin \rmS$  and $|\Sigma_{\infty}-\rmS_{\infty}|=2$.
	
	We will also exclude the case $B=GL_2(F)$ and $g=2$, i.e. the case of the Hilbert modular surfaces as this case is already covered in \cite{Stamm}. In particular this assumption implies the Shimura varieties we consider are compact. We fix an identification $\Sigma_\infty\cong \{\tau_1,\dotsc,\tau_g\}$ such that $\sigma(\tau_i)=\tau_{i+1}$ and $\Sigma_{\infty}-\rmS_{\infty}=\{\tau_1,\tau_c\}$ for some $c\in\{2,\dotsc,g\}$.
	
	Let $\rmT\subset \rmS_\infty$ which  in this appendix we will assume satisfies $2|\rmT|=|\rmS|$ and let $K_p$ be the standard hyperspecial subgroup of $B^\times_{\rmS}\otimes_F F_\fkp\cong GL_2(\calO_{F_\fkp})$.  We also define $K_{0,\fkp}$ to denote the following compact open subgroup $$K_0(\fkp)=\{\left(\begin{matrix}
	a &b\\c&d
	\end{matrix}\right)\in GL_2(\calO_{F_\fkp})|\left(\begin{matrix}
	a &b\\c&d
	\end{matrix}\right)\equiv \left(\begin{matrix}
	* &*\\0&*
	\end{matrix}\right)\mod\fkp \};$$
	it is an Iwahori subgroup of $GL_2(F_\fkp)$. 
	
	For $K^p$ sufficiently small, we define $K=K_pK^p$ and $K_0(\fkp)=K_{0,\fkp}K^p$. We then have the Shimura varieties $\text{Sh}_{K_0(\fkp)}(G_{\rmS,\rmT}),\  \text{Sh}_K(G_{\rmS,\rmT})$ defined over the common reflex field $\bfE_{\rmS,\rmT}$. These are equipped with finite \'etale maps $$\pi_1,\pi_2: \text{Sh}_{K_0(\fkp)}(G_{\rmS,\rmT})\rightarrow \text{Sh}_{K}(G_{\rmS,\rmT}).$$
	 Using a similar procedure as in \S\ref{sec:unitary}, we may define an integral model for $\text{Sh}_{K_0(\fkp)}(G_{\rmS,\rmT})$  as follows.

	We fix a choice  $E/F$ of CM-extension and a pair $\tilde{\rmS}=(\rmS,\tilde{\rmS}_\infty)$ as in \S\ref{sec:unitary}; in particular $E$ is split over the prime $\fkp$ and we write $\fkq$ and $\overline{\fkq}$ for the primes above $\fkp$. We make the following assumption on $\tilde{\rmS}_\infty$.
	
	\begin{ass}
		\label{ass: tildeS} Let $\tilde{\tau}\in\Sigma_{E,\infty/\fkq}$ be a lift of $\tau\in\Sigma_{\infty/\fkq}$. Then $\tilde{\tau}\in\tilde{\rmS}_\infty$ if and only if $\tau\in\rmT$.
	\end{ass}	 We let $D_{\rmS}:=B_{\rmS}\otimes E$ and $W:=D_{\rmS}$ considered as a right $D_{\rmS}$-module of rank 1. Then $W$ is equipped with a pairing $\psi$ of $(\ref{eq:pairing})$. Recall the associated groups $G_{\tilde{\rmS}}'$ and $G_{\tilde{\rmS}}''$ which fit in the diagram $$G_{\rmS}\leftarrow G_{\rmS}\times T_E\rightarrow G''_{\tilde{\rmS}}\leftarrow G'_{\tilde{S}}$$
	We  let $K_{0,\fkp}''$ denote the image of $K_{0,\fkp}\times K_{E,p}$ in $G''_{\tilde{\rmS}}(\bbQ_p)$.
	
To define the level structures for $G'_{\tilde{\rmS}}$, note that an element $g_p\in G_{\tilde{\rmS}}'(\bbQ_p)$ corresponds to an element of $\text{End}(W\otimes_{\bbQ}\bbQ_p)$ such that $$\psi(vg_p,wg_p)=c(g_p)\psi(v,w),\ \forall v,w\in W\otimes_{\bbQ}\bbQ_p.$$
	Here we abuse notation so that $\psi$ also denotes the base change of the pairing $\psi$ to $\bbQ_p$. We fix an isomorphism $\calO_{B_S}\otimes_F F_\fkp\cong\text{Mat}_2(\calO_{F_\fkp})$ which induces an identification $$\calO_{D_S}\otimes_FF_{\fkp}\cong \text{Mat}_2(\calO_{E_\fkq})\times\text{Mat}_2(\calO_{E_{\overline{\fkq}}}).$$We then define the following lattice chain $\Lambda_1\subset\Lambda_2$ of $\calO_{D_S}\otimes_F F_\fkp$-modules where
	$$\Lambda_1=\left(\begin{matrix}
	\calO_\fkq &\fkq\\
	\calO_{\fkq} &\fkq
	\end{matrix}\right)\oplus\left(\begin{matrix}
	\calO_{\overline{\fkq}}& \calO_{\overline{\fkq}}\\
	\calO_{\overline{\fkq}} & \calO_{\overline{\fkq}}
	\end{matrix}\right),\ \ \Lambda_2=\left(\begin{matrix}
	\calO_\fkq & \calO_\fkq \\
	\calO_{\fkq} & \calO_\fkq 
	\end{matrix}\right)\oplus\left(\begin{matrix}
	{\overline{\fkq}}^{-1} &\calO_{\overline{\fkq}}\\
	{\overline{\fkq}}^{-1} &\calO_{\overline{\fkq}}
	\end{matrix}\right)$$
	
	Let $K_{0,\fkp}'\subset G_{\tilde{\rmS}}'(\bbQ_p)$ denote the compact open subgroup stabilizing the lattice chain $\Lambda_2\subset\Lambda_1$. For a sufficiently small compact open subgroup of $K'^p\subset G_{\tilde{\rmS}}'(\bbA_f^p)$ (resp. $K''^p\subset G_{\tilde{\rmS}}''(\bbA_f^p)$), we write $K'_0(\fkp)$ (resp. $K''_0(\fkp)$) for the compact open subgroup $K'_{0,\fkp}K'^p\subset G_{\tilde{\rmS}}'(\bbA_f)$ (resp. $K''_{0,\fkp}K''^p\subset G_{\tilde{\rmS}}''(\bbA_f))$. We then have the associated Shimura varieties $\text{Sh}_{K_0'(\fkp)}(G_{\tilde{\rmS}}')$, $\text{Sh}_{K_0''(\fkp)}(G_{\tilde{\rmS}}'')$ and the inverse limit schemes: $$\text{Sh}_{K'_{0,\fkp}}(G_{\tilde{\rmS}}'):=\varprojlim_{K'^p}\text{Sh}_{K'_0(\fkp)}(G'_{\tilde{\rmS}}),\ \ \ \text{Sh}_{K''_{0,\fkp}}(G_{\tilde{\rmS}}''):=\varprojlim_{K''^p}\text{Sh}_{K''_0(\fkp)}(G''_{\tilde{\rmS}})$$
	
	As in the quaternionic case, there are finite \'etale maps $$\pi'_1,\pi_2':\text{Sh}_{K'_0(\fkp)}(G_{\tilde{\rmS}}')\rightarrow\text{Sh}_{K'}(G_{\tilde{\rmS}}')$$
	$$\pi''_1,\pi_2'':\text{Sh}_{K''_0(\fkp)}(G_{\tilde{\rmS}}'')\rightarrow\text{Sh}_{K''}(G_{\tilde{\rmS}}'').$$
	
	As before we also have the identification of neutral components:
	$$\text{Sh}_{K_{0,\fkp}}(G_{\rmS,\rmT}))_{\overline{\bbQ}_p}^\circ\xleftarrow{\sim}\text{Sh}_{K''_{0,\fkp}}(G''_{\tilde{\rmS}})_{\overline{\bbQ}_p}^\circ\xrightarrow{\sim}\text{Sh}_{K'_{0,\fkp}}(G'_{\tilde{\rmS}})_{\overline{\bbQ}_p}^\circ.$$
	In fact these isomorphisms descend to $\bbQ_p^{\mathrm{ur}}$, and this will allow us to transfer the integral models that we will construct.
	
	In order to define integral models, we follow the procedure in \cite{RZ} adapted to our situation.
	
	We first make the following definition.
	
	\begin{defn}
		Let $A/S$ be the abelian variety associated to some point of $\underline{\rmSh}_{K'}(G'_{\tilde{\rmS}})$. A {\it cyclic isogeny} $$\rho:A\rightarrow A'$$ is an isogeny of abelian varieties over $S$ such that the kernel  of $\rho$ is an $\calO_{D_{\rmS}}$-stable subgroup $H_\fkq\oplus H_{\overline{\fkq}}$ of $A[\fkq]\oplus A[\overline{\fkq}]$ of order $|k_\fkp|^4$ such that $H_\fkq$ and $H_{\overline{\fkq}}$ are dual to one another under the pairing $$H_\fkq\times H_{\overline{\fkq}}\rightarrow \mu_p$$ induced by the polarization. Moreover we require  that  the induced action of $\calO_{D_{\rmS}}$ on $A'$ to satisfy the condition (\ref{eq:char poly 1}).
	\end{defn}

	We now consider the moduli problem $\underline{\text{Sh}}_{K'_0(\fkp)}(G'_{\tilde{\rmS}})$ that associates to an  $\calO_{{\bfE}_{\tilde{\rmS},\tilde{v}}}$-scheme $S$ the set of isomorphism classes of tuples $(\rho:A\rightarrow A',\iota,\lambda,\epsilon_{K'^p})$ where :
	
	$\bullet$ $(A,\iota,\lambda,\epsilon_{K'^p})$ is an $S$-point of $\underline{\text{Sh}}_{K'}(G'_{\tilde{\rmS}})$.
	
	$\bullet$ $\rho:A\rightarrow A'$ is a cyclic isogeny.
	
	Then $\underline{\text{Sh}}_{K'_0(\fkp)}(G'_{\tilde{\rmS}})$ is representable by a quasi-projective variety over $\calO_{{\bfE}_{\tilde{\rmS},\tilde{v}}}$  and it is an integral model for ${\text{Sh}}_{K'_0(\fkp)}(G'_{\tilde{\rmS}})$. The two degeneracy maps extend to maps also denoted $\pi_1,\pi_2$ of $\calO_{{\bfE}_{\tilde{\rmS},\tilde{v}}}$-schemes:
	$$\pi_1,\pi_2:\underline{\text{Sh}}_{K'_0(\fkp)}(G'_{\tilde{\rmS}})\rightarrow\underline{\text{Sh}}_{K'}(G'_{\tilde{\rmS}}).$$

These maps can be described explicitly in terms of the moduli interpretation as follows. $\pi_1$ sends the tuple $(\rho:A\rightarrow A',\iota,\lambda,\epsilon_{K'^p})$ to $(A,\iota,\lambda,\epsilon_{K'^p})$. To define $\pi_2$, note that given a tuple $(\rho:A\rightarrow A',\iota,\lambda,\epsilon_{K'^p})$, since $H_\fkq\oplus H_{\overline{\fkq}}:=\ker(\rho)$ is $\calO_{D_{\rmS}}$-stable, the action of $\calO_{D_{\rmS}}$ on $A$ extends to an action of $A'$. Moreover the polarization $\lambda$ and level structure $\epsilon_{K'^p}$ induce a polarization $\lambda'$ and level structure $\epsilon'_{\rmK'^p}$ on $A'$. It is easy to check the tuple $(A',i',\lambda',\epsilon'_{K'^p})$ satisfies the conditions defining a point of $\underline{\text{Sh}}_{K'}(G'_{\tilde{\rmS}})$. This defines the map $\pi_2$.

In order to transfer the integral model to the quaternionic side, we can use the construction in \cite[\S2]{TX}. Note however that since the models we construct do not satisfy the correct extension property, there is a subtlety in defining the  $G^{\mathrm{ad}}(\bbQ)^+$-action on the integral model. We may instead use a direct description of this action as in  \cite[\S4.4]{KP} by twisting abelian varieties; the rest of the argument then goes through and we obtain an integral model $\underline{\Sh}_{K_0(\fkp)}(G_{\rmS,\rmT})$ for $\Sh_{K_0(\fkp)}(G_{\rmS,\rmT})$ over $\calO_{{\bfE}_{\tilde{\rmS},\tilde{v}}}$. \begin{rem}Alternatively, we may use \cite[Corollar 2.13]{TX} to define a universal $p$-divisible group with $\calD^\circ$-structure over $\underline{\rmSh}_K(G_{\rmS,\rmT})_{\calO_{{\bfE}_{\tilde{\rmS},\tilde{v}}}}$ in the sense of Definition \ref{def: D^0 structure}. We may then define $\Sh_{K_0(\fkp)}(G_{\rmS,\rmT})$ as classifying cyclic isogenies of this universal $p$-divisible group. We refer to \cite{Car} for the details in the case of Shimura curves. We prefer to transfer the results from the unitary Shimura variety since in this case we may directly apply theorems which are known for certain PEL type Shimura varieties.
	\end{rem}

The integral model $\underline{\Sh}_{K_0(\fkp)}(G_{\rmS,\rmT})$ is equipped with degeneracy maps $$\pi_1,\pi_2:\underline{\Sh}_{K_0(\fkp)}(G_{\rmS,\rmT})\rightarrow \underline{\Sh}_{K}(G_{\rmS,\rmT}).$$
Note also that we are not interested in any canonicity properties of these models. The only thing we use is the existence of such a model with the necessary geometric properties that we will describe in this section. The integral model $\underline{\Sh}_{K_0(\fkp)}(G_{\rmS,\rmT})$ is constructed from the connected component $$\underline{\Sh}_{K_{0,\fkp}}(G_{\rmS,\rmT})_{\Z_p^{\mathrm{ur}}}^\circ\xrightarrow{\sim}\underline{\Sh}_{K'_{0,\fkp}}(G'_{\tilde{\rmS}})_{\Z_p^{\mathrm{ur}}}^\circ$$
	and the action of a certain group $\calE_{G_{\rmS,\rmT},p}$, cf. \cite[2.11]{TX}.
	
	For ease of notation, we shall write $X'$ for the special fiber of $\underline{\mathrm{Sh}}_{K'}(G'_{\tilde{\rmS}})$ and $X'_0(\fkp)$ for the special fiber of $\underline{\text{Sh}}_{K'_0(\fkp)}(G'_{\tilde{\rmS}})$. We will also write $X$ and $X_0(\fkp)$ for the special fibers of $\underline{\text{Sh}}_{K}(G_{\rmS,\rmT})$ and $\underline{\text{Sh}}_{K_0(\fkp)}(G_{\rmS,\rmT})$ respectively.
	
	We begin with a few basic facts concerning these moduli spaces.
	\begin{prop}
(1) $\underline{\mathrm{Sh}}_{K_0'(\fkp)}(G'_{\tilde{\rmS}})$ is a flat normal scheme over $\calO_{{\bfE}_{\tilde{\rmS},\tilde{v}}}$ with reduced special fiber.

(2) Each irreducible component of $X'_0(\fkp)$ is smooth of dimension 2.
	\end{prop}
\begin{proof}
	Both these properties follow from the corresponding properties of the local models; see \cite{Go1}. 
\end{proof}

\subsection{Group theoretic preliminaries}\label{sec: Appendix B(G)}
Let $G$ be a reductive group over $\Z_p$; in particular its generic fiber is quasi-split and splits over an unramified extension of $\bbQ_p$. Let $k$ an algebraically closed field of characteristic $p$, $L:=W(k)[\frac{1}{p}]$ and $\Ok_L=W(k)$. We write  $\sigma$ for the Frobenius element of $\text{Aut}(L/\bbQ_p)$. We fix a maximal $\bbQ_p^{\mathrm{ur}}$-split  torus  $T$ and a Borel subgroup $B$ containing $G$. We let $X_*(T)$ (resp. $X^*(T)$) denote the group of cocharacters (resp. characters) of $T$, and $X_*(T)_+$ (resp. $X^*(T)_+$) the submonoid of dominant cocharacters (resp. characters) with respect to the choice of Borel subgroup $B$.

For $b\in G(L)$ we let $[b]=\{g^{-1}b\sigma(g)\in G(L)|  g\in G(L)\}$ denote its $\sigma$-conjugacy class in $G(L)$ and we write $B(G)$ to denote the set of all $\sigma$-conjugacy classes. The set $B(G)$ has been classified by Kottwitz in \cite{Ko1}, \cite{Ko2}. 

For $b\in G(L)$ we let $\overline{\nu}_b\in X_*(T)_{\bbQ,+}^\sigma$ denote its dominant Newton cocharacter; it depends only on the image of $b$ in $B(G)$. We let $$\kappa_G:G(L)\rightarrow \pi_1(G)_\Gamma$$ denote the Kottwitz homomorphism, where $\pi_1(G)_\Gamma$ denotes the $\Gamma$-coinvariants of $\pi_1(G)$. This induces a map, also denoted $\kappa_G$ from $B(G)$ to $\pi_1(G)_\Gamma$. 

By \cite[\S4.13]{Ko2}, the map $$B(G)\rightarrow X_*(T)_{\bbQ,+}^\sigma\times \pi_1(G)_\Gamma,\ \ [b]\mapsto (\overline{\nu}_b,\kappa_G(b))$$
is injective.

We define a partial order on the set $X_*(T)_{\bbQ,+}^\sigma\times \pi_1(G)_\Gamma$ by setting $(\nu_1,\kappa_1)\leq (\nu_2, \kappa_2)$ if $\kappa_1=\kappa_2$ and $\nu_2-\nu_1$ is a non-negative rational linear combination of positive coroots.

\begin{eg}\label{eg:B(G)}(1) Let $G=GL_n$. Then we have a bijection
	$$B(G)\leftrightarrow \{\text{isocrystals over $k$ of height $n$}\}$$
	given by taking $[b]\in B(G)$ to the isocrystal $(L^n,b\sigma)$. We may take $T$ to be the diagonal matrices and $B$ the upper triangular matrices. There is an identification $X_*(T)\cong \bbQ^n$ and we have $\pi_1(G)\cong \bbZ$. The first isomorphism identifies $$X_*(T)_{\bbQ,+} \leftrightarrow \{(\nu_1,\dotsc,\nu_n)\in\bbQ^n|\nu_1\geq\dotsc\geq\nu_n)\in\bbQ^n$$
	and $\kappa_G$ takes $b\in G(L)$ to the valuation of its determinant. For $[b]\in B(G)$, the element $\overline{\nu}_b$ corresponds under the above identification to the Newton slopes of the associated isocrystal. In this case $\kappa_G(b)$ is determined by the $\overline{\nu}_b$, and Kottwitz's classification recovers the Dieudonn\'e-Manin classification of isocrystals by their Newton slopes.

	(2)
	Let $\rmF$ denote a finite unramified extension of $\bbQ_p$ of degree $d$ and let $G=\text{Res}_{\calO_{\rmF}/\bbZ_p}GL_n$. As in the previous example, the association $[b]\mapsto (\rmF\otimes_{\bbQ_p}L,b\sigma)$ defines a bijection between $B(G)$ and the set of isocrystals of height $dn$ with an action of $\rmF$. Given such an isocrystal $N$, we have a decomposition $$N\cong \prod_{\tau:\rmF\rightarrow L}N_\tau$$ where $N_\tau$ is the subspace of $N$ over which $\rmF$ acts via the embedding $\tau:\rmF\rightarrow L$. As $N_\tau$ is fixed by $b\in G(L)$ and $\sigma$ induces a bijection between $N_\tau$ and $N_{\sigma\tau}$, it follows that $(b\sigma)^d$ takes $N_\tau$ to itself. This gives $N_\tau$ the structure of  a $\sigma^d$-isocrystal which is easily seen to be independent of the choice of $\tau$. One checks that the association $N\mapsto N_\tau$ induces a bijection between.
	$$\{\text{isocrystals of height $dn$ with an action of $\rmF$}\}\leftrightarrow \{\text{$\sigma^d$-isocrystals of height $n$}\}.$$	
	
	 We let $T$ denote the diagonal torus and $B$ the Borel subgroup of upper triangular matrices. We have an
	  identification $$X_*(T)^\sigma_{\bbQ}\cong \bbQ^n$$ identifying $$X_*(T)^\sigma_{\bbQ,+}\leftrightarrow \{(\nu_1,\dotsc,\nu_n)|\nu_1\geq\dotsc\geq\nu_n\}.$$ For $[b]\in B(G)$, the Newton cocharacter $\overline{\nu}_b$ corresponds under the above identification to the slopes of the $\sigma^d$-isocrystal $(N_\tau,(b\sigma)^d)$.
\end{eg}
 By the Cartan decomposition we have  an identification $$G(L)\cong \coprod_{\mu\in X_*(T)_+} G(\calO_L)\mu(p)G(\calO_L)$$
 Now fix a cocharacter $\mu \in X_*(T)_+$. We define $\overline{\mu}\in X_*(T)$ to be the Galois average of $\mu$. More precisely we take a Galois extension $\rmE/\bbQ_p$ over which $\mu$ is defined and we define $$\overline{\mu}=\frac{1}{[\rmE:\bbQ_p]}\sum_{\tau\in \text{Gal}(\rmE/\Q_p)}\tau(\mu).$$
 We also define $\mu^\natural$ to be the image of $\mu$ in $\pi_1(G)_\Gamma$.

 \begin{prop}[{\cite[Thm. 4.2]{RaRi}}]\label{RaRi} Let $b\in G(\calO_L)\mu(p)G(\calO_L)$. Then we have $$(\overline{\nu}_b,\kappa_G(b))\leq(\overline{\mu},\mu^\natural)$$
 \end{prop}

	\subsection{$p$-divisible groups with $\calO$-structure and Dieudonn\'e theory.}\label{sec: Appendix O-structure}
	In this section we recall the notion of $p$-divisible groups with an action of the ring of integers of a finite unramified extension of $\bbQ_p$. 
	
	Let $\rmF$ be a finite unramified extension of $\bbQ_p$ of degree $d$ with ring of integers $\calO$. We let $q=p^d$ denote the cardinality of its residue field.
	\begin{defn}
		Let $S$ be a scheme. A $p$-divisible group with $\calO$-structure over $S$ is a pair $({\mathscr{G}},\iota)$ where ${\mathscr{G}}$ is a $p$-divisible group  over $S$ and $\iota:\calO\rightarrow \text{End}({\mathscr{G}})$ is a homomorphism.
	\end{defn}
For any $p$-divisible group with $\calO$-structure $({\mathscr{G}},\iota)$, we write ${\mathscr{G}}[p^n]$ for the kernel of multiplication by $p^n$. Then there exists an integer $h:=\text{ht}_\calO {\mathscr{G}}$, the $\calO$-height of ${\mathscr{G}}$,  such that ${\mathscr{G}}[p^n]$ has rank $q^{nh}$. It is easily verified that we have the equality $$\text{ht}{\mathscr{G}}=[\rmF:\bbQ_p]\text{ht}_{\calO}{\mathscr{G}}$$ where $\text{ht}{\mathscr{G}}$ is the usual height of ${\mathscr{G}}$ as a $p$-divisible group.

Let $S$ be a scheme in which $p$ is locally nilpotent. For $\scrG$ a $p$-divisible group over $S$, we write $\bbD({\mathscr{G}})$ for  the contravariant Dieudonn\'e crystal of ${\mathscr{G}}$. This is a locally free crystal on the crystalline site of $S$, equipped with a map $\sigma^*\bbD({\mathscr{G}})\rightarrow \bbD({\mathscr{G}})$.

Now suppose ${\mathscr{G}}$ is a $p$-divisible group with $\calO$-structure over an algebraically closed field $k$ of characteristic $p$ of $\calO$-height $n$. Then we identify $\bbD({\mathscr{G}})$ with its Dieudonn\'e module (i.e. $\bbD({\mathscr{G}})$ evaluated at $\calO_L:=W(k)$) which is a finite free $\calO_L$-module of rank $dn$ equipped with an injective $\sigma$ semi-linear map $\varphi:\bbD({\mathscr{G}})\rightarrow \bbD({\mathscr{G}})$ and a action of $\calO$. Fixing an $\calO\otimes_{\bbZ_p}\calO_L$-basis of $\bbD({\mathscr{G}})$, we obtain an element $b\in G(L)$ where $G=\text{Res}_{\calO/\bbZ_p}GL_n$, such that $\varphi=b\sigma$ under the identification $\bbD({\mathscr{G}})\cong (\calO_F\otimes_{\bbZ_p}\calO_L)^n$. The element $b$ is well-defined up to $\sigma$-conjugation by $G(\calO_L)$. 

We define the {\it Hodge polygon} of ${\mathscr{G}}$ to be the element $\mu\in X_*(T)^+$ such that $$b\in G(\calO_L)\sigma(\mu(p))G(\calO_L).$$
By Proposition \ref{RaRi}, it follows that $(\nu_b,\kappa_G(b))\leq(\overline{\mu},\mu^\natural)$. If $S$ is a scheme of characteristic $p$, its Hodge polygon  (resp. Newton polygon) is the function assigning to any geometric point $\overline{s}$ of $S$ the Hodge polygon (resp. the Newton polygon) of the base change of ${\mathscr{G}}$ to $\overline{s}$.

We now describe how the study of $p$-divisible groups which arise naturally from the moduli problem in the last subsection can be reduced to the case of $p$-divisible groups with $\calO$-structure. Recall we have the integral PEL datum $(D_{\rmS},*,W,\psi,\calO_{D_{\rmS}},\Lambda_2\subset \Lambda_1)$ and we write $\calD$ for the base change of this datum to $\bbZ_p$. Recall $g=[F:\bbQ]$. Let $D_{\rmS,p}$ denote the completion of $D_{\rmS}$ at the place $p$  and $F_p$ (resp. $E_p$) the completion of $F$ (resp. $E$) at $p$. Then $F_p\cong F_\fkp$ and $E_p\cong E_{\fkq}\times E_{\overline{\fkq}}$; we write $\calO$ for the ring of integers $\calO_{F_{\fkp}}$. Fixing isomorphisms $E_\fkq\cong F_\fkp$ and $E_{\overline{\fkq}}\cong F_{\fkp}$, we obtain an isomorphism $$D_{\rmS,p}\cong\text{Mat}_2(E_\fkq)\times\text{Mat}_2(E_{\overline{\fkq}})\cong \text{Mat}_2(F_\fkp)\times\text{Mat}_2(F_\fkp).$$ Let $\calO_{D_{\rmS,p}}$ denote the maximal order $\text{Mat}_2(\calO)\times\text{Mat}_2(\calO)$. Then the involution $*$ on $D_{\rmS,p}$ can be identified with $$(a,b)\mapsto (b^t,a^t).$$

\ignore{	\begin{defn}
	Let $S$ be a $\calO$-scheme.	A $p$-divisible group with $\calD$-structure is a triple $({\mathscr{G}},\lambda,\iota)$ where:
	
	$\bullet$ ${\mathscr{G}}$ is a $p$-divisible group over $S$ of height $\dim W$.
	
	$\bullet$ $\iota:\calO_{D_{\rmS},p}\rightarrow \text{End}({\mathscr{G}})$ is a homomorphism satisfying \begin{equation}\label{eq:char poly 4}\text{char}(\iota(a)|\text{Lie}{\mathscr{G}})=\text{char}(a|W^1).\end{equation}
	
	$\bullet$ $\lambda:{\mathscr{G}}\rightarrow {\mathscr{G}}^\vee$ is a $\bbZ_p^\times$-homogeneous polarization such that $$\iota(a)=\lambda^{-1}\circ\iota(a^*)^\vee\circ\lambda. $$
		
	\end{defn}}

Recall we have defined the notion of $p$-divisible group with $\calD$-structure in Definition \ref{def:D-structure}. It follows easily from the definitions that if $(A,\iota,\lambda,\epsilon_{K'^p})$ is an $S$-point of $\underline{\Sh}_{K'}(G_{\tilde{\rmS}}')$, then the associated $p$-divisible group $A[p^\infty]$ together with the induced $\calO_{D_{\rmS,p}}$-action and polarization is a $p$-divisible group with $\calD$-structure.

We write $W^1$ for the  sub $F_{\fkp}$ vector space over which $D_{\rmS,p}$ acts via the first factor, and we write $W^\circ$ for subspace $eW^1$ where $e$ is the idempotent $\left(\begin{matrix}
	1 &0 \\0 & 0
\end{matrix}\right)$.
We define an integer  $s_\tau\in\{0,1,2\}$ for $\tau\in\rmS_{\infty}$ by $$\label{eq:s tau}
s_{\tau}=\begin{cases} 0 & \text{if $\tau\in{\rmT}$}\\
2 & \text{if $\tau\in\rmS_\infty-\rmT$}\\
1 &\text{otherwise}
\end{cases}.
$$

\begin{defn}\label{def: D^0 structure}
	Let $S$ be  scheme over $\calO_{{\bfE}_{\tilde{\rmS},\tilde{v}}}$. A $p$-divisible group with $\mathcal{D}^\circ$-structure is a $p$-divisible group ${\mathscr{G}}$ with $\calO$-structure such that for $a\in\calO$, we have an equality:
\begin{equation}\label{eq:char poly 2}\mathrm{char}(\iota(a)|\text{Lie}{\mathscr{G}})=\prod_{\tau\in\rmS_{\infty}} (T-\tau(a))^{s_\tau}.\end{equation}
In particular this implies that $\mathrm{ht}_{\calO}{\mathscr{G}}=2$.
\end{defn}

Let ${\mathscr{G}}/S$ be a $p$-divisible with $\calD$-structure. Then we have a decomposition $${\mathscr{G}}={\mathscr{G}}_\fkq\times {\mathscr{G}}_{\overline{\fkq}}$$
where $\calO_{D_{\rmS},p}$ acts on ${\mathscr{G}}_{\fkq}$ via the projection to $\text{Mat}_2(E_{\fkq})$ and on ${\mathscr{G}}_{\overline{\fkq}}$ via the projection to $\text{Mat}_2(E_{\overline{\fkq}})$. Moreover by \cite[Lemma 4.1]{Ham2}, there exists a $p$-divisible group ${\mathscr{G}}'$ with $\calO$-structure such that ${\mathscr{G}}_{\fkq}\cong {\mathscr{G}}'^2$. The condition (\ref{eq:char poly}) and the Assumption \ref{ass: tildeS} implies that ${\mathscr{G}}'$ is a $p$-divisible group with $\calD^\circ$-structure.

The following Proposition follows from the discussion above and \cite[Corollary 4.5 (2)]{Ham2}.

\begin{prop}\label{prop:red-pdiv}
The association ${\mathscr{G}}\mapsto {\mathscr{G}}'$ induces an equivalence of categories
\begin{equation}\label{eq:equiv cat. pdiv gps}\{\text{$p$-divisible groups with $\calD$-structure}\}\xrightarrow\sim\{\text{$p$-divisible groups with $\calD^\circ$-structure}\}\end{equation}
preserving isogenies.
\end{prop}
Let ${\mathscr{G}}'$ be a $p$-divisible group with $\calD^\circ$-structure over an algebraically closed field $k$ of characteristic $p$. Fixing a trivialization of $\bbD({\mathscr{G}})(W(k))$ respecting the $\calO$-structure, we obtain an element $b\in G(L)$ where $G=\mathrm{Res}_{\calO/\bbZ_p}GL_2$. Let $T$ be the diagonal maximal torus of $G$, then we may identify $X_*(T)$ (resp. $X_*(T)_{\bbQ}$) with $g$ copies of $\Z^{2}$ (resp. $g$ copies of $\Q^2$) and $X_*(T)_+$  (resp. $X_*(T)_{\bbQ,+}$) with the subset such that for each factor of $\Z^2$, the terms $(a,b)$ are decreasing. The condition \ref{eq:char poly 2} implies that the Hodge polygon $\mu$ of ${\mathscr{G}}$ corresponds to $[(a_i,b_i)]_{i=1,\dotsc,g}$ where 
\begin{align*}a_i=&{\begin{cases} 1 & \text{if $i=1,c$ or $\tau_i\in\rmS_\infty-\rmT$}\\
0 & \text{if $\tau_i\in\rmT$}
\end{cases}}\\
b_i=&{\begin{cases} 1 & \text{if $\tau_i\in\rmS_\infty-\rmT$}\\
0 & \text{if $i=1,c$ or $\tau_i\in\rmT$}
\end{cases}}\end{align*}
We write $\mu\in X_*(T)$ for this cocharacter.

\begin{prop}\label{prop:Newton}
(1)	There exists exactly two elements $[b]\in B(G,\mu)$

(2) For the maximal element $[\mu(p)]$, there exists a unique $p$-divisible group with $\calO$-structure with these Newton slopes.
\end{prop}
\begin{proof}
 (1) It is easy to check using the explicit description of $\mu$ that under the identification $X_*(T)^\sigma\cong \bbQ^2$, that the only two elements of $B(G,\{\mu\})$ correspond to $\nu^{\mathrm{ord}}=(\frac{g+1}{2},\frac{g-1}{2})$ and
$\nu^{\mathrm{ss}}=(\frac{g}{2},\frac{g}{2})$. Here we use the fact that $2|\rmT|=|\rmS|$.
	
	(2) This is \cite[Theorem 3.2.7]{Mo}; note in this case the Newton cocharacter is $\nu^{\mathrm{ord}}$.
\end{proof}
We say a $p$-divisible group $\scrG/k$ with $\calD^\circ$-structure is ordinary (resp. supersingular) if the corresponding Newton vector is equal to $\nu^{\mathrm{ord}}$ (resp. $\nu^{\mathrm{ss}}$). Similarly a $p$-divisible group $\scrG$ with $\calD$-structure is ordinary (resp. supersingular) if the corresponding $p$-divisible group with $\calD^\circ$-structure is. For $k$ an algebraically closed field of characteristic $p$, we write $\scrG^{\mathrm{ord}}$ for the unique isomorphism class of ordinary $p$-divisible groups from part (2) of Proposition \ref{prop:Newton}.

If  $A_{X'}$ denotes the universal abelian variety over $X'$, we write $X'^{\mathrm{ss}}$  (resp. $X'^{\mathrm{ord}}$) for the locus where $A_{X'}[p^\infty]$ is supersingular (resp. ordinary). Similarly if $A\rightarrow A'$ denotes the universal cyclic isogeny over $X'_0(\fkp)$, we define $X'_0(\fkp)^{\mathrm{ss}}$  (resp. $X'_0(\fkp)^{\mathrm{ord}}$) as the locus where $A[p^\infty]$ (equivalently $A'[p^\infty]$) is supersingular (resp. ordinary). It is easy to see that $X'^{\mathrm{ss}}$  is the union of the Goren--Oort divisors corresponding to $\tau_1$ and $\tau_c$, see \cite[\S3]{LT} for example.

\begin{prop}\label{prop:ord point smooth}
Let $x\in X'_0(\fkp)^{\mathrm{ord}}(\overline{\bbF}_p)$. Then $x$ is a smooth point of $X'_0(\fkp)$.
\end{prop}
\begin{proof}
	The local model in this case has a stratification by the $\mu$-admissible  set $\Adm(\{\mu\})$ and this induces a stratification of $X'_0(\fkp)_{\overline{\bbF}_p}$, cf. \cite[\S9]{HZ}. The strata corresponding to translation elements are all smooth. By \cite[Theorem 2.6]{HeNie1},  $X'_0(\fkp)^{\mathrm{ord}}$ is contained in these strata, hence $x$ is a smooth point of $X'_0(\fkp)$. Note that the Axioms of \cite{HR} for these Shimura varieties have been verified in \cite{HZ} so that \cite[Theorem 2.6]{HeNie1} is applicable.
\end{proof}

\begin{defn} Let $(\scrG,\lambda,\iota)$ and $(\scrG',\lambda',\iota')$ be $p$-divisible groups with $\calD$-structure over $S$. A \textit{cyclic isogeny} between $\scrG$ and $\scrG'$ over $S$ is an isogeny $f: \scrG\rightarrow\scrG'$  such that
	
	$\bullet$ $f$ is compatible with the  actions $\iota,\iota'$ and the polarizations $\lambda,\lambda'$.
	
	$\bullet$ $\ker(f)= K_\fkq \oplus K_{\overline{\fkq}}\subset \scrG_\fkq[\fkq]\oplus \scrG_{\overline{\fkq}}[\fkq]$  is of order $|k_\fkp|^4$ and $K_{\fkq}$ is dual to $K_{\overline{\fkq}}$ under the pairing induced by $\lambda$.
	
	
Similarly, for  $p$-divisible groups with $\calD^\circ$-structure $\scrG$ and $\scrG'$ over $S$, a {\it cyclic isogeny} between $\scrG$ and $\scrG'$ is an isogeny $f:\scrG\rightarrow \scrG'$ such that:

$\bullet$ $f$ is compatible with the action of $\calO$ on $\scrG$ and $\scrG'$.

$\bullet$ $\ker(f)\subset \scrG[p]$ is of order $|k_\fkp|$ .

\end{defn}

It is easy to check that under the equivalence of categories in Proposition \ref{prop:red-pdiv}, the cyclic isogenies correspond to one another. Moreover if $(A,\iota,\lambda,\epsilon_{K'^p})$ is an $S$-point of $\underline{\mathrm{Sh}}_{K'}(G'_{\tilde{\rmS}})$, then a cyclic isogeny $A\rightarrow A'$ corresponds precisely to a cyclic isogeny of the associated $p$-divisible group.
	\subsection{Essential Frobenius and Verschiebung isogenies.}\label{sec: Appendix essentail Frobenius}
In this section we define two canonical cyclic isogenies associated to a point in $X'$. These isogenies will define sections to the projections $\pi_1,\pi_2:X'_0(\fkp)\rightarrow X'$ analogous to the Frobenius and Verschiebung isogenies in the case of the Hilbert modular surface (see \cite[\S4]{Stamm}). For this reason we will call these isogenies the essential Frobenius and essential Verschiebung isogenies respectively.

We let $q=p^g$ and let $(A,\iota,\lambda,\epsilon_{K'^p})$ be an $S$-point of $X'$ where $S$ is a smooth $\bbF_q$-scheme.  We first define the essential Frobenius isogeny $A\rightarrow A'$. Note that in order to define a cyclic isogeny it suffices to define a cyclic isogeny of the associated $p$-divisible group $\scrG:=A[p^\infty]$. Let $\pdiv=(\pdiv')^2\times(\pdiv'^{\vee})^2$ denote the decomposition of $\scrG$ coming from Proposition \ref{prop:red-pdiv}; then it suffices to define a cyclic isogeny of $\scrG'$. In order to do this we introduce some notation.

Let $k$ be a perfect field of characteristic $p$ and $R$ a smooth $k$-algebra. By a frame for $R$ we mean a $p$-adically complete and separated flat $W(k)$-algebra lifting $R$ together with a lift of Frobenius $\sigma:\scrR\rightarrow\scrR$. 

A {\it Dieudonn\'e $F$-crystal over $R$} is a quadruple $(M,F,V,\nabla)$ where

$\bullet$ $M$ is a finite locally free $\scrR$ module.

$\bullet$ $F:\sigma^*M\rightarrow M$ and $V:M\rightarrow \sigma^*M$ are injective $\scrR$-linear maps such that $FV=p$ and $VF=p$.

$\bullet$ $\nabla$ is a topologically nilpotent integrable connection such that $F$ is parallel for $\nabla$.

By \cite{DeJ}, $\bbD(\scrG)(\scrR)$ is a Dieudonn\'e $F$-crystal and the association $\scrG\mapsto \bbD(\scrG)(R)$ induces an anti-equivalence of categories between $p$-divisible groups over $R$ and Dieudonn\'e $F$-crystals. Similarly the association induces an anti-equivalence of categories between $p$-divisible groups with $\calO$-structure and Dieudonn\'e crystals with an action of $\calO$.

It will follow from the canonicity of the construction that we may assume $S$ is affine, since we may glue the construction over an affine cover. Thus let $S=\Spec R$ be a smooth $\bbF_q$-scheme and $\scrR$ a frame for $R$ as above.  Then in order to define a cyclic isogeny, it suffices to find an $\scrR$-lattice $M\subset \bbD(\scrG)(\scrR)$ satisfying the following conditions:

(1)  $M$ is stable for the action of $\calO$ 

(2) $M$ is stable under $F,V,\nabla$.

(3) We have the inclusions $p\bbD(\scrG)(\scrR)\subset M\subset \bbD(\scrG)(\scrR)$.


(4) (\ref{eq:char poly 2}) holds for the module $M/VM$.

We now construct such a subgroup. For $\tau\in\Sigma_{\infty}$, we identify this with an embedding $\tau:\calO\rightarrow W(\bbF_q)$ and we let $\bbD(\scrG)(\scrR)_\tau$ the submodule where $\calO$ acts via $\tau$. Similarly to  \S\ref{sec: Goren-Oort} we define the essential Verschiebung to be $$V_{\text{es},\tau}:\bbD(\scrG)(\scrR)_\tau\rightarrow \bbD(\scrG)(\scrR)_{\sigma^{-1}(\tau)}$$ to be the usual Verschiebung if $\sigma^{-1}(\tau)\notin \rmS_{\infty}$ or $\sigma^{-1}(\tau)\in \rmS_{\infty}-\mathrm{T}$ and the inverse of Frobenius  otherwise. For $\tau\in \Sigma_{\infty}$, let $n_\tau$ be the smallest positive integer  such that $\sigma^{-n_\tau}(\tau)\notin \rmS_\infty$. We define $M_\tau\in \bbD(\scrG)(\scrR)_\tau$ to be the preimage of $p\bbD(\scrG)(\scrR)_{\sigma^{-n_\tau}(\tau)}$ under the map $V_{\mathrm{es}}^{n_\tau}$. Then we define $M:=\bigoplus_{\tau\in\Sigma_{\infty}}M_\tau$.  Then $M$ is a locally free $\scrR$-submodule of $\bbD(\scrG)(\scrR)$ of full rank.

\begin{prop}
	$M\subset \bbD(\scrG)(\scrR)$ satisfies the properties (1), (2), (3), (4), above.
\end{prop}
\begin{proof}
	(1) is clear since $M$ is a direct sum of locally free submodules of $\bbD(\scrG)(\scrR)_\tau$. 
	To show $M$ is stable under $F$, we must show $F(M_\tau)\subset M_{\sigma(\tau)}$. We consider the separate cases   $\tau\in\rmS_\infty$ and $\tau\notin \rmS_\infty$. If $\tau\in \rmS_\infty$,  we have
	$$V^{n_{\sigma(\tau)}}_{\text{es}}|_{\bbD(\scrG)(\scrR)_{\tau}}=V_{\text{es}}^{n_{\tau}}\circ V_{\text{es},\sigma(\tau)}$$  and  $V_{\text{es},\sigma(\tau)}$ is the usual Verschiebung or the inverse of Frobenius. Since $VF=p$, we have $V_{\mathrm{es},\sigma(\tau)}FM_\tau\subset M_{\tau}$ and hence $FM_\tau\subset V_{\mathrm{es},\sigma(\tau)}^{-1}M_\tau=M_{\sigma(\tau)}$.
	
	In the second case we have $$M_{\sigma(\tau)}=V_{\text{es},\sigma(\tau)}^{-1}(p\bbD(\scrG)(\scrR)_\tau)=F(\bbD(\scrG)(\scrR)_\tau).$$ Since $$M_\tau=(V_{\mathrm{es}}^{n_\tau})^{-1}(p\bbD(\scrG)(\scrR)_{\sigma^{n_\tau}(\tau)})\subset \bbD(\scrG)(\scrR)_\tau,$$ we have $F(M_\tau)\subset M_{\sigma(\tau)}$.  The verification of the stability for $V$ follows similarly.
	
	The stability under $\nabla$ follows from the stability of $p\bbD(\scrG)(\scrR)_\tau$ under $\nabla$ and the fact that $F$ and $V$ are horizontal for $\nabla$. It follows that (2) is satisfied. 
	
	(3) follows from the inclusions $$pV\bbD(\scrG)(\scrR)_{\sigma(\tau)}\subset p\bbD(\scrG)(\scrR)_\tau\subset V \bbD(\scrG)(\scrR)_{\sigma(\tau)}$$  for $\tau\in \Sigma_{\infty}-\rmS_{\infty}$.
	
	(4) is equivalent to the condition  $$\dim (M/VM)_{\tau}=\begin{cases}1 & \text{if $\tau\in\Sigma_\infty-\rmS_{\infty}$}\\
	0 &\text{if $\tau\in\rmS_{\infty}-\rmT$}\\
	2& \text{if $\tau\in\rmT$}
	\end{cases}$$
	which follows similarly.
\end{proof}
It follows that $M$ corresponds to a $p$-divisible group $\scrG^{(p')}$ equipped with a cyclic isogeny $\text{Fr}':\scrG\rightarrow \scrG^{(p')}$. We define $\scrG^{(p')}$ to be {\it the essential Frobenius twist of $\scrG$} and $\text{Fr}'$ to be the {\it essential Frobenius isogeny}. Similarly we may define the {\it essential Verschiebung} $$\text{Ver}':\scrG^{(p')}\rightarrow \scrG'$$ to be defined by the submodule $$p\bbD(\scrG)(\scrR)\subset M$$
It follows from the definition that $\text{Fr}'\circ\text{Ver}'=p$ and $\text{Ver}'\circ\text{Fr}'=p$.

Let $A$ be the universal abelian scheme over $X'$. Taking a cover of $X'$ by smooth affine opens and gluing the construction we obtain isogenies $$\text{Fr}'_A:A\rightarrow A^{(p')}, \ \ \ \text{Ver}_A':A^{(p')}\rightarrow A.$$ 

\begin{prop}\label{prop: ess Frob squared}
	Let $A$ be an abelian variety with $\calD^\circ$-structure over a scheme $S$ of characteristic $p$. Then $$(A^{(p')})^{(p')}=A^{(p)}$$ where $A^{(p)}$ is the usual $p^g$-Frobenius twist. Moreover the composition $\mathrm{Fr}'\circ\mathrm{Fr}'$ considered as a map $X'\rightarrow X'$ corresponds to $S_\fkp^{-(g/2-1)}\mathrm{Fr}_\fkp$ where $S_\fkp$ is the standard Hecke operator at $p$ and  $\mathrm{Fr}_\fkp$ is the $p^g$-Frobenius.
	\end{prop}
\begin{proof}We may reduce to the case $S=\Spec R$ a smooth affine scheme. Moreover, it suffices to prove this for the  $p$-divisible $\scrG'$  associated to $A[p^\infty]$ by Proposition \ref{prop:red-pdiv}. 
	
	The submodule $F^ge\bbD(\scrG')(\scrR)\subset \bbD(\scrG')(\scrR)$ corresponds to the $p^g$-Frobenius isogeny. We let $M\subset\bbD(\scrG')(\scrR)$ be the submodule corresponding to $\scrG'^{(p')}$ and $M'\subset M$ the submodule corresponding to $(\scrG'^{(p')})^{(p')}$. 
	Then by the definition of $M'$, we have that $p^{\frac{g}{2}-1}M'=F^g\bbD(\scrG')(\scrR)$. In particular, $F^g\bbD(\scrG')(\scrR)\subset M'$ corresponds to the isogeny \begin{equation}\label{eq: two Frob}p^{\frac{g}{2}-1}:(\scrG'^{(p')})^{(p')}\rightarrow (\scrG'^{(p')})^{(p')}.\end{equation}
	Therefore $(\scrG'^{(p')})^{(p')}$  and $\scrG^{(p)}$ are isomorphic and the moreover part follows.
\end{proof}

\begin{prop}\label{Frob-Ver}Let $\scrG$ be an ordinary $p$-divisible with $\calO$-structure over a reduced and irreducible $\bbF_q$-scheme $S$. Let $f:\scrG\rightarrow \scrG'$ be a cyclic isogeny, then $f=\mathrm{Fr}'$ or $f=\mathrm{Ver}'$.
\end{prop}
\begin{proof}First assume $S=\Spec k$ where $k$ is an algebraically closed field of characteristic $p$. Then $\scrG\cong\scrG^{\ord}$ by Proposition \ref{prop:Newton}. The same proof as in \cite[Proposition 4.3]{Stamm} shows that $f=\mathrm{Fr}'$ or $f=\mathrm{Ver}.'$
	
	Let $S^{\text{Fr}'}$ (resp $S^{\text{Ver}'}$) denote the subscheme of $S$ where $f$ coincides with $\text{Fr}'$ (resp. $\text{Ver}'$). Then $S^{\text{Fr}'}$ and $S^{\text{Ver}'}$ are closed subschemes of $S$, since the locus where two subgroups of $\scrG$ coincide is closed. For each closed $s$ point of $S$, we have by the case above that $f_s=\text{Fr}'$ or $f_s=\text{Ver}'$. Therefore $S^{\text{Fr}'}\cup S^{\text{Ver}'}=S$, hence by irreducibility $S=S^{\text{Ver}'}$ or $S=S^{\text{Fr}'}$. 
	
\end{proof}

The association $(A,\iota,\lambda,\epsilon_{K'^p})\mapsto (A^{(p')},\iota',\lambda',\epsilon'_{K'^p})$ induces a map $$\text{Fr}':X'\rightarrow X'$$ By Proposition \ref{prop: ess Frob squared}, we have $\text{Fr}'\circ\text{Fr}'=S_\fkp^{-(g/2-1)}\text{Fr}_\fkp$.
These isogenies induce maps $$\calF,\calV:X'\rightarrow X'_0(\fkp)$$
which are defined by $$\calF(A,\iota,\lambda,\epsilon_{K'^p})=(\text{Fr}'_A:A\rightarrow A^{(p')},\iota,\lambda,\epsilon_{K'^p})$$ $$ \calV(A,\iota,\lambda,\epsilon_{K'^p})=(\text{Ver}'_A:A^{(p')}\rightarrow A,\iota',\lambda',\epsilon'_{K'^p}) .$$
By definition we have the following properties $$\pi_1\circ \calF=\text{id}_{X'},\ \ \pi_2\circ\calF=\text{Fr}'$$ $$\pi_1\circ\calV=S^{-1}_{\fkp}\text{Fr}',\ \  \pi_2\circ\calV=\text{id}_{X'}.$$

We have the following proposition regarding these maps.
\begin{prop}\label{FV-section}
	(1) $\calF$ and $\calV$ are closed immersions.
	
	(2) Let $(\rho:A\rightarrow A',\iota,\lambda,\epsilon_{K'^p})$ be an $S$-point of  $X_0(\fkp)^{\ord}$, then $x\in \calF(X)\cup\calV(X)$.  
	
\end{prop}
\begin{proof}
	(1) We prove the result for $\calF$, the case of $\calV$ is analogous. First note that $\calF$ is injective on points since $\pi_1\circ\calF=\text{id}_X$. Therefore it suffices to show $\calF$ is proper; the valuative criterion in this case follows from standard properties of Neron models.
	
	(2) Follows directly from Proposition \ref{Frob-Ver}.
\end{proof}
\subsection{Global structure of quaternionic Shimura surfaces with Iwahori level structure}In this subsection we prove the main theorem. We first need to study the fibers of the maps $\pi_1$.

Let $\scrG/\overline{\bbF}_p$ be a supersingular $p$-divisible group with $\calD^\circ$-structure.  We will define a universal cyclic isogeny of $\scrG$; this will parameterize the fiber of the projection $\pi_1:X'_0(\fkp)_{\overline{\bbF}_p}\rightarrow X'_{\overline{\bbF}_p}$. As in \S\ref{sec: Goren-Oort}, we may define a version of essential Verschiebung for any $p$-divisible group with $\calD^\circ$ structure  $\scrG$ over an $\overline{\bbF}_p$-scheme $S$. We have the exact sequence of sheaves over $S$ \begin{equation}\label{eq:exact seq Dieudonne}0\rightarrow\omega_{\scrG,\tau_i}\rightarrow \mathbb{D}(\mathscr{G})(S)_{\tau_i}\rightarrow \omega_{\scrG^\vee,\tau_i}^*\rightarrow 0\end{equation} for each $i=1,\dotsc,g$, where $\scrG^\vee$ denotes the Cartier dual and $\omega_{\scrG}$ (resp. $\omega_{\scrG^\vee}$) is the sheaf of invariant differential forms on $\scrG$ (resp. $\scrG^\vee$). If $\scrG$ arises from an abelian variety with $\calD$-structure $A$, then $\omega_{\scrG^\vee,\tau_i}\cong \omega^\circ_{A^\vee,\tilde{\tau}_i}$ as in \S\ref{sec: Goren-Oort}, where $\tilde{\tau}_i\in\Sigma_{E,\infty/\fkq}$ lifts $\tau_i$. We obtain sections $$h_{\tau_i}\in\Gamma(S,\omega_{\scrG^\vee,\sigma^{-n_{\tau_i}}\tau_i}^{\otimes p^{n_{\tau_i}}}\otimes\omega_{\scrG^\vee,\tau_i}^{\otimes-1})$$ for $i=1,c,$ using the same construction from \S\ref{sec: Goren-Oort}; here we need to use the essential \textit{Frobenius} to define these sections since we are using contravariant Dieudonn\'e theory. Applying this to the universal $p$-divisible group with $\calD^\circ$-structure on $X'$, these may be identified with the partial Hasse invariants defined in \S\ref{sec: Goren-Oort}. 

Let $\scrG/\Fpbar$ be a supersingular $p$-divisible group with $\calD^\circ$-structure. We separate the following two cases:

(1) Both $h_{\tau_1}$ and $h_{\tau_c}$ vanish.

(2) Exactly one of $h_{\tau_1}$ or $h_{\tau_c}$ vanish.

The following Proposition can be proved in the same way as \cite[\S5]{Stamm} where the analogous calculations are carried out for the Hilbert modular surface.
\begin{prop}Let $\calO_L:=W(\Fpbar)$.
	
	In Case (1), there exists an $\calO_L$-basis $\{e_i,f_i\}$ of $\bbD(\scrG)(\calO_L)_{\tau_i}$ for $i=1,...,g$ such that $$\varphi(e_i)=f_{i+1},  \varphi(f_i)=pe_{i+1}, \text{ for $i=1,c$}$$
	$$\varphi(e_i)=e_{i+1}, \varphi(f_i)=f_{i+1}, \text{ for $\tau_i\in \rmS_\infty-\rmT$}$$
	$$\varphi(e_i)=pe_{i+1}, \varphi(f_i)=pf_{i+1}, \text{ for $\tau_i\in \rmT$}$$
	
	In Case (2); we assume without loss of generality $h_{\tau_c}\neq0 $. Then there exists an element $u\in\overline{\bbF}-\bbF_{p^g}$ and $\calO_L$-basis $\{e_i,f_i\}$ of $\bbD(\scrG)(\Ok_L)_{\tau_i}$ such that  
	$$\varphi(e_i)=f_{i+1},\ \varphi(f_i)=pe_{i+1}-[u^{p^{g-c+1}}]f_{i+1}, \text{ for $i=1$}$$
	$$\varphi(e_i)=f_{i+1}+[u^p]e_{i+1},\ \varphi(f_i)=pe_{i+1}, \text{ for $i=c$}$$
	$$\varphi(e_i)=e_{i+1},\ \varphi(f_i)=pe_{i+1}, \text{ for $\tau_i\in \rmS_\infty-\rmT$}$$
	$$\varphi(e_i)=pe_{i+1},\ \varphi(f_i)=pf_{i+1}, \text{ for $\tau_i\in \rmT$}$$
	where $[u]\in\calO_L$ denotes the Teichmuller lifting.
\end{prop}
\ignore{\begin{proof}We write $\varphi_{\mathrm{es},\tau}:\bbD(\scrG)(\Fpbar)_\tau\rightarrow\bbD(\scrG)(\Fpbar)_{\sigma(\tau)}$ for the essential Frobenius which is defined to be the usual Frobenius for $s_\tau=1,2$ or the inverse of Verschiebung if $s_\tau=0$. We write $\varphi_{\mathrm{es}}:\bbD(\scrG)(\Fpbar)\rightarrow\bbD(\scrG)(\Fpbar)$ for the sum of these maps.
	
	Suppose we are in Case (1). Then by definition, the maps $\omega_{\scrG^\vee,\tau_1}^*\rightarrow \omega_{\scrG^\vee,\tau_c}^*$ and $\omega_{\scrG^\vee,\tau_c}^*\rightarrow \omega_{\scrG^\vee,\tau_1}^*$ induced by $\varphi_{\mathrm{es}}^{c-1}$ and $\varphi_{\mathrm{es}}^{g-c+1}$ are both zero. For $i=1,c$, let $e_i\in \bbD(\scrG)(\Ok_L)_{\tau_i}$ be a lift of a generator of $\omega^*_{\scrG^\vee,\tau_i}$; since $\scrG$ is supersingular we may choose $e_i$ such that $\varphi^g(e_i)=p^{g/2}e_i$. Let $f_2=\varphi( e_1)$ and for $j=2,\dotsc,c-1$  define $f_{j+1}=\varphi_{\mathrm{es}}(f_{j})$. Similarly let $f_{c+1}=\varphi(e_c)$ and for $j=c+1,\dotsc,g$ define $f_{j+1}=\varphi_{\mathrm{es}}(f_{j})$. Since $h_{\tau_1}$ and $h_{\tau_c}$ vanish, the image of $f_i$ in  $\bbD(\scrG)(\Fpbar)_{\tau_i}$ lands in $\omega_{\scrG,\tau_i}$  for $i=1,c$. It follows that $\varphi(f_i)=pe_{i+1}$ for some $e_{i+1}\in \bbD(\scrG)(\Ok_L)_{\tau_{i+1}}.$ For $j\neq 1,2,c,c+1$, we let $e_{i}=\varphi_{\mathrm{es}}(e_{i-1})$. One checks that $e_i,f_i$ is a suitable basis as in the Proposition.
	
	Now suppose we are in Case (2). We let $f_1\in\bbD(\scrG)(\Ok_L)_{\tau_1}$  be a lift of a generator of $\omega^*_{\scrG^\vee,\tau_1}$ such that $\varphi^g(f_1)=p^{g/2}f_1$. Let $e_2=\varphi(f_1)$ and for $i=2,\dotsc,c-1$ we let $e_{i+1}=\varphi_{\mathrm{es}}(e_i)$. Then as above, we have $\varphi(e_c)=pf_{c+1}$ for some $f_{c+1}\in\bbD(\scrG)(\Ok_L)_{\tau_{c+1}}$ and we let $f_{i+1}=\varphi_{\mathrm{es}}(f_i)$ for $i=c+1,\dotsc,g$. Now let $f_c\in\in\bbD(\scrG)(\Ok_L)_{\tau_c}$ be a lift of a generator of $\omega^*_{\scrG^\vee,\tau_c}$. Then since $h_{\tau_1}$ does not vanish. There exists $u\in\Fpbar$ such that the image of
	
\end{proof}}
We now define the  universal cyclic  isogeny of $\scrG$. In the two cases above we fix a basis of $\bbD(\scrG)(\Ok_L)_{\tau_i}$ as in the previous Proposition.

Case (1): In this case we will define a cyclic isogeny of $\scrG\times T$ where $T$ is the scheme consisting of two copies of $\bbP^1$ intersecting transversally at an $\overline{\bbF}_p$-point. To do this we first introduce some notation. 

Let $S=\Spec R$ be a smooth $\overline{\bbF}_p$-scheme. 

Let $\scrR=\Ok_L\langle x\rangle$ denote the ring of restricted power series over $\Ok_L$ equipped with the lift of Frobenius given by the usual Frobenius on $\Ok_L$ and $x\mapsto x^p$. Then $\scrR$ is a frame for $R=\overline{\bbF}_p[x]$.

We let $M$ be the $\scrR$-module $\bbD(\scrG)(\Ok_L)\otimes_{\Ok_L}\scrR$, equipped with the induced Frobenius, Verschiebung and the trivial connection $$\nabla:=1\otimes \text{d},$$ 
where $d:\mathscr{R}\rightarrow \Omega^1_{\scrR/\Ok_L}$ is the universal derivation.
Then $(M,F,V,\nabla)$ is the Dieudonn\'e $F$-crystal associated to $\scrG\otimes \scrR$. We now define the submodule $\tilde{M}_0\subset M$ to be the submodule generated by generated by $pM$ and $$\langle e_i+x^{p^{i-c}}f_i|i=c+1,\dotsc,g,1\rangle\cup\langle f_i|i=2,\dotsc,c\rangle.$$

One checks that $\tilde{M}_0$ is stable under $F,V$ and $\nabla$ hence corresponds to a cyclic isogeny $$\tilde{\rho}_0:\scrG\times \Spec R\rightarrow \tilde{\scrG}_0$$ where  $\tilde{\scrG}_0$ is a $p$-divisible group with $\calO$-structure over $R$.

Similarly letting $\scrR'=\Ok_L\langle y\rangle$, we define the submodule $\tilde{M}_0'$ of $M':=\bbD(\scrG)(\Ok_L)\otimes_{\Ok_L}\scrR'$ to be generated by $pM'$ and $$\langle y^{p^{i-c}}e_i+f_i|i=c+1,\dotsc,g,1\rangle\cup\langle f_i|i=2,\dotsc,c\rangle $$ As before this corresponds to a cyclic isogeny $$\tilde{\rho}_0':\scrG\times\Spec R\rightarrow \tilde{\scrG}_0'.$$ Using the identification $x\leftrightarrow \frac{1}{y}$, we see that $\tilde{\rho_0}$ and $\tilde{\rho}_0'$ glue to give an isogeny $$\rho_0:\scrG\times \bbP^1\rightarrow \scrG_0.$$

Similarly we may define an isogeny $$\rho_1:\scrG\times \bbP^1\rightarrow \scrG_1$$  by gluing the isogenies corresponding to submodules $\tilde{M}_1\subset M$ and $\tilde{M}_1'\subset M'$ where  $$\langle f_i|i=c+1,\dotsc,g,1\rangle\cup\langle e_i+x^{p^{i-1}}f_i|i=2,\dotsc,c\rangle. $$ and $$\langle f_i|i=c+1,\dotsc,g,1\rangle\cup\langle y^{p^{i-1}}e_i+f_i|i=2,\dotsc,c\rangle.$$

Let $z_0=(1:0)$ and $z_1=(0:1)$. Then $\rho_0|_{z_0}:\scrG\rightarrow \scrG_{0,z_0}$ and $\rho_1|_{z_1}:\scrG\rightarrow \scrG_{1,z_1}$ agree and hence we may glue $\rho_0$ and $\rho_1$ to give an isogeny $$\rho:\scrG\times T\rightarrow \scrG'.$$

Case (2): Assume $h_{\tau_c}\neq 0$. 

As in case (1), we let $\mathscr{R}=\Ok_L\langle x\rangle$ and  $M=\bbD(\scrG)(\scrR)$. We define the submodule $\tilde{M}$ of $M$ to be generated by $pM$ and $$\langle e_i+x^{p^{i-c}}f_i|i=c+1,...,g,1\rangle\cup\langle f_i|i=2,...,c\rangle.$$
Similarly $\tilde{M}'$ is the submodule generated by  $pM$ and $$\langle y^{p^{i-c}}e_i+f_i|i=c+1,...,g,1\rangle\cup\langle f_i|i=2,...,c\rangle.$$
As in Case (1), these submodules correspond to isogenies which glue along $x\leftrightarrow\frac{1}{y}$ and define an isogeny $\scrG\times\bbP^1\rightarrow \scrG'$.

If $h_{\tau_1}\neq 0$, we may switch the roles of $1$ and $c$ to obtain an isogeny $\scrG\times\bbP^1\rightarrow \scrG'$.

	Let $x\in X'(\overline{\bbF}_p)$ lie in the intersection of the two Goren--Oort divisors $X'_{\tau_1,\Fpbar}$, $X'_{\tau_c,\Fpbar}$ corresponding to $\tau_1$ and $\tau_c$; we write $A_x$ for the associated abelian variety and $\scrG_{x}'$ for the $p$-divisible group with $\calD^\circ$-structure associated to $A_x[p^\infty]$ as in Proposition \ref{prop:red-pdiv}.  Then we are in Case (1) and have defined a cyclic isogeny $$\rho:\scrG'_x\times T\rightarrow \scrG'$$ where $T$ is the union of two copies of $\bbP^1$ intersecting transversally at point. This induces a cyclic isogeny $A_x\times T\rightarrow A'$. We therefore obtain a map $$\beta:T\rightarrow X'_0(\fkp)_{\overline{\bbF}_p}$$ whose image lies in the fiber of $x$ under the degeneracy map $\pi_1$.
	
	Similarly if $x\in X'(\overline{\bbF}_p)$ lies in a unique Goren--Oort strata for $\tau_1$  or $\tau_c$, then we are in Case (2) and we obtain a cyclic isogeny $A_x\times\bbP^1\rightarrow A'$.  This induces a map $$\alpha:\bbP^1\rightarrow X'_0(\fkp)_{\overline{\bbF}_p}$$ whose image lies in the fiber of $x$ under $\pi_1$.

\begin{prop}\label{prop:fiber}
	In Case (1), the map $\beta:T\rightarrow \pi_1^{-1}(x)$ is a bijection on $\Fpbar$-points.
	
	In Case (2), the map  $\alpha:\bbP^1\rightarrow \pi_1^{-1}(x)$ is a bijection on $\Fpbar$-points.
\end{prop}
\begin{proof}
This follows from the same calculation as in \cite[Proposition 6.5]{Stamm}
\end{proof}
	
	\begin{thm}\label{thm:Iwahori-structure}
	(1)	$X'_0(\fkp)_{\overline{\bbF}_p}$ can be decomposed as $$\calV(X_{\overline{\bbF}_p}')\cup\calF(X_{\overline{\bbF}_p}')\cup X'_0(\fkp)_{\overline{\bbF}_p}^{\mathrm{ss}}$$
	where each term is a union of irreducible components of $X'_0(\fkp)_{\overline{\bbF}_p}$ of dimension 2. Any irreducible component of $X'_0(\fkp)_{\overline{\bbF}_p}$ is contained in exactly one of the terms  and $\calV(X_{\overline{\bbF}_p}')\cap\calF(X_{\overline{\bbF}_p}')$ is 0-dimensional.
	
	(2) For each irreducible component $C$ of $X'^{\mathrm{ss}}_{\overline{\bbF}_p}$, there is a unique irreducible component $R$ of $X'_0(\fkp)$ such that $\pi_1(R)=C$. Moreover the projection $\pi_{1}:R\rightarrow C$ exhibits $R$ as a $\bbP^1$-bundle over $C\cong\bbP^1$ and for any other irreducible component $C'$ of $X'^{\mathrm{ss}}_{\overline{\bbF}_p}$, we have:
	
	$$R\cap R'=\begin{cases}
	\{pt\} & \text{if } C\cap C'\neq\emptyset\\
	\emptyset & \text{otherwise}
	\end{cases}.$$
	Similarly, the result holds if we replace $\pi_1$ by $\pi_2$.
	\end{thm}

The rest of this subsection will be devoted to the proof. We refer to \cite[p409]{Stamm} for a pictorial representation of the geometry.
\begin{proof}
(1) By the theory of local models every irreducible component of $X'_0(\fkp)$ has dimension 2. Now since $\calF$ and $\calV$ are closed immersions, $\calV(X_{\overline{\bbF}_p}')$ and $\calF(X'_{\overline{\bbF}_p})$ are unions of irreducible components of $X'_0(\fkp)_{\overline{\bbF}_p}$. Let $Z$ be an irreducible component of $X'_0(\fkp)_{\overline{\bbF}_p}$ not contained in $\calF(X'_{\overline{\bbF}_p})\cup\calV(X'_{\overline{\bbF}_p})$ and let $x\in Z(\Fpbar)$. If $x$ corresponds to an ordinary abelian variety, then $x\in \calF(X'_{\overline{\bbF}_p})\cup\calV(X'_{\overline{\bbF}_p})$ by Proposition \ref{FV-section}, hence it lies in more than one irreducible component of $X'_0(\fkp)_{\overline{\bbF}_p}$. In particular $x$ is not a smooth point of $X'_0(\fkp)_{\overline{\bbF}_p}$ contradicting Proposition \ref{prop:ord point smooth}, therefore $x$ is supersingular.

To show $\calV(X_{\overline{\bbF}_p}')\cap\calF(X_{\overline{\bbF}_p}')$ is 0-dimensional, we note that by the closure relations of the Kottwitz-Rapoport stratification on the local model, the intersection is contained in the minimal stratum which is $0$-dimensional. Part (1) follows.

(2) Let $R$ be an irreducible component of $X'_0(\fkp)_{\overline{\bbF}_p}$. Then since $\pi_1|_{X'_0(\fkp)_{\overline{\bbF}_p}^{\mathrm{ss}}}$ is surjective with one dimensional fibers, there exists a unique irreducible component $C$ of $X'^{\mathrm{ss}}_{\overline{\bbF}_p}$ such that $\pi_1(R)=C$. Now Proposition \ref{prop:fiber} implies the fibers of $\pi_1|_{X'_0(\fkp)^{\mathrm{ss}}_{\overline{\bbF}_p}}$ are irreducible over a dense subset of $C$, hence $R$ is the unique irreducible component of $X'_0(\fkp)^{\mathrm{ss}}_{\overline{\bbF}_p}$ mapping to $C$.

To show $R$ is a $\bbP^1$-bundle over $C$, we first show the morphism $\pi_1|_{R}$ is smooth. Since $R$ and $C$ are both smooth and ${\overline{\bbF}_p}$ is algebraically closed, it suffices to show the map is surjective on tangent spaces. This follows in the same way as \cite[Proposition 6.7]{Stamm}.

For $x$ a smooth point of $X_{\overline{\bbF}_p}'^{\mathrm{ss}}$, we have a map $\mathbb{P}^1\rightarrow \pi_1^{-1}(x)$ which is bijective on $\overline{\bbF}_p$-points by Proposition \ref{prop:fiber}. Since $\pi_1^{-1}(x)$ is smooth, it follows that $\mathbb{P}^1\cong \pi_1^{-1}(x)$. Thus it suffices to show that for $x$ a non-smooth point of $C\subset X'^{\mathrm{ss}}$, we have $\pi_1^{-1}(x)\cap R\cong\bbP^1$. In this case there is a map $T\rightarrow \pi^{-1}(x)$ which is bijective on $\overline{\bbF}_p$-points by Proposition \ref{prop:fiber}, where $T$ is the transversal intersection of two copies of $\bbP^1$. Since each component of $\pi^{-1}(x)$ is smooth, the same argument as above shows that $T\cong\pi^{-1}(x)$.

Since $\pi_1^{-1}(x)\cap R$ is smooth and one dimensional, it follows that $\pi_1^{-1}(x)\cap R\cong \bbP^1$.

\end{proof}

Transferring to the quaternionic side, we obtain the analogous results for quaternionic Shimura varieties.

\begin{cor}\label{cor:Iwahori level structure}
	(1)	$X_0(\fkp)_{\overline{\bbF}_p}$ can be decomposed as $$\calV(X_{\overline{\bbF}_p})\cup\calF(X_{\overline{\bbF}_p})\cup X_0(\fkp)_{\overline{\bbF}_p}^{\mathrm{ss}}$$
	where each term is a union of irreducible components of $X_0(\fkp)_{\overline{\bbF}_p}$ of dimension 2. Any irreducible component of $X_0(\fkp)_{\overline{\bbF}_p}$ is contained in exactly one of the terms and $\calV(X_{\overline{\bbF}_p})\cap\calF(X_{\overline{\bbF}_p})$ is 0-dimensional.
	
	(2) For each irreducible component $C$ of $X^{\mathrm{ss}}_{\overline{\bbF}_p}$, there is a unique irreducible component $R$ of $X_0(\fkp)$ such that $\pi_1(R)=C$. Moreover the projection $\pi_{1}:R\rightarrow C$ exhibits $R$ as a $\bbP^1$-bundle over $C\cong\bbP^1$ and for any other irreducible component $C'$ of $X^{\mathrm{ss}}_{\overline{\bbF}_p}$, we have:
	
	$$R\cap R'=\begin{cases}
	\{pt\} & \text{if } C\cap C'\neq\emptyset\\
	\emptyset & \text{otherwise}
	\end{cases}.$$
	Similarly, the result holds if we replace $\pi_1$ by $\pi_2$.
\end{cor}

\bibliographystyle{amsalpha}
\bibliography{bibfile}

\end{document}